\documentclass[10pt,british]{article}
\usepackage[T1]{fontenc}
\usepackage[latin9]{inputenc}
\usepackage{geometry}
\usepackage{amsthm}
\usepackage{amsmath,mathtools}
\usepackage{amssymb}
\usepackage{graphicx}

\makeatletter
\theoremstyle{plain}
\newtheorem{thm}{\protect\theoremname}[section]
  \theoremstyle{definition}
  \newtheorem{defn}[thm]{\protect\definitionname}
  \theoremstyle{plain}
  \newtheorem{cor}[thm]{\protect\corollaryname}
  \theoremstyle{plain}
  \newtheorem{conjecture}[thm]{\protect\conjecturename}
  \theoremstyle{remark}
  \newtheorem{rem}[thm]{\protect\remarkname}
  \theoremstyle{plain}
  \newtheorem{prop}[thm]{\protect\propositionname}
  \theoremstyle{plain}
  \newtheorem{lem}[thm]{\protect\lemmaname}
  \theoremstyle{remark}
  \newtheorem*{claim*}{\protect\claimname}
 \newcommand\thmsname{\protect\theoremname}
 \newcommand\nm@thmtype{theorem}
 \theoremstyle{plain}
 
 \newenvironment{namedthm}[1][Undefined Theorem Name]{
   \ifx{#1}{Undefined Theorem Name}\renewcommand\nm@thmtype{theorem*}
   \else\renewcommand\thmsname{#1}\renewcommand\nm@thmtype{namedtheorem}
   \fi
   \begin{\nm@thmtype}}
   {\end{\nm@thmtype}}
  \theoremstyle{remark}
  \newtheorem{claim}[thm]{\protect\claimname}
  \theoremstyle{remark}
  \newtheorem*{rem*}{\protect\remarkname}

\usepackage{hyperref,calc}
\theoremstyle{plain}
\newtheorem{obs}[thm]{Observation}
\numberwithin{equation}{section}

\DeclareMathOperator{\diam}{diam\!}

\makeatother

\usepackage{babel}
  \providecommand{\claimname}{Claim}
  \providecommand{\conjecturename}{Conjecture}
  \providecommand{\corollaryname}{Corollary}
  \providecommand{\definitionname}{Definition}
  \providecommand{\lemmaname}{Lemma}
  \providecommand{\propositionname}{Proposition}
  \providecommand{\remarkname}{Remark}
  \providecommand{\theoremname}{Theorem}
\providecommand{\theoremname}{Theorem}

\begin{document}

\title{Frozen percolation in two dimensions}

\author{Demeter Kiss%
\thanks{research supported by NWO%
}}
\maketitle
\begin{abstract}
Aldous \cite{Aldous2000} introduced a modification of the bond percolation
process on the binary tree where clusters stop growing (freeze) as
soon as they become infinite.  We investigate the site version of
this process on the triangular lattice where clusters freeze as soon
as they reach $L^{\infty}$ diameter at least $N$ for some parameter
$N.$ We show, informally speaking, that in the limit $N\rightarrow\infty,$
the clusters only freeze in the critical window of site percolation
on the triangular lattice. Hence the fraction of vertices that eventually
(i. e. at time $1$) are in a frozen cluster tends to $0$ as $N$
goes to infinity. We also show that the diameter of the open cluster
at time $1$ of a given vertex is, with high probability, smaller
than $N$ but of order $N.$ This shows that the process on the triangular
lattice has a behaviour quite different from Aldous' process. We also
indicate which modifications have to be made to adapt the proofs to
the case of the $N$-parameter frozen bond percolation process on
the square lattice. This extends our results to the square lattice,
and answers the questions posed by van den Berg, de Lima and Nolin
in \cite{Nolin2008}.
\end{abstract}

\tableofcontents{}

\section{\label{sec: introduction}Introduction}

Stochastic processes where small fragments merge and form larger ones
are quite useful tools to model physical phenomena at scales ranging
from molecular \cite{1943JChPh..11...45S} to astronomical ones \cite{Wetherill1990336}.
The majority of the mathematical literature on such coagulation processes
treats mean field models: The rate at which the fragments (clusters)
merge is governed only by their sizes - neither the physical location
nor their shape affect this rate. See \cite{MR2253162} for a review.
Stockmayer \cite{1943JChPh..11...45S}, introduced a mean field model
for polymerization where small clusters (sol) merge, however, as soon
as a large cluster (gel) forms, it stops growing. In contrast to the
mean field models, we consider a model which takes the geometry of
the space and the shape of the clusters into account. Following van
den Berg, de Lima and Nolin \cite{Berg2012a}, and Aldous \cite{Aldous2000},
we introduce the following adaptation of Stockmayer's model. Let $G=\left(V,E\right)$
be a graph which represents the underlying geometry and $N\in\mathbb{N}.$
For every vertex $v\in V,$ independently from each other, we assign
a random time $\tau_{v}$ which is uniformly distributed on $\left[0,1\right]$.
At time $t=0,$ all of the vertices of $G$ are closed. As time increases,
a vertex $v$ tries to become open at time $t=\tau_{v}.$ It succeeds
if and only if all of its neighbours' open clusters (open connected
components) at time $t$ have size less than $N.$ Note that as soon
as the diameter of a cluster reaches $N,$ it stops growing, i.e freezes.
Hence the name $N$-parameter frozen percolation. Note that we can
also consider an edge (bond) version of the model above where edges
turn open from closed. This edge version of the process was introduced
by van den Berg et al \cite{Berg2012a}.

We are particularly interested in the $N$-parameter frozen percolation
models for large $N$ on graphs such as $d$ dimensional lattices,
since they are discrete approximations of the space $\mathbb{R}^{d}.$
Herein we restrict to the case where $d=2.$ We will mainly work on
the triangular lattice. We will see that the behaviour of this model
is rich and interesting too, but in a very different way from the
model studied by Aldous \cite{Aldous2000}.

Let us turn to the model introduced and constructed by Aldous \cite{Aldous2000}.
It is the edge version of the model on the binary tree where we replace
the parameter $N$ by $\infty$ in the description above. An edge
$e$ of the binary tree opens at time $\tau_{e}$ as long as the open
clusters of the endpoints of $e$ are finite. In view of this model,
one could also try to construct a similar, so called $\infty$-parameter,
model on the triangular lattice. However Benjamini and Schramm \cite{Benjamini1999}
showed that it is impossible. Exactly this non-existence result motivated
van den Berg, de Lima and Nolin \cite{Berg2012a} to extend the model
of Aldous for finite parameter $N$: in this case, the $N$-parameter
frozen percolation process (both the vertex and the edge version)
is a finite range interacting particle system, hence the general theory
\cite{Liggett2005} gives existence. One could ask if the $N$-parameter
processes for large but finite $N$ provide a reason for the existence
of the $\infty$-parameter frozen bond percolation on the binary tree
and the non-existence of the $\infty$-parameter frozen site percolation
on the triangular lattice. Before we answer this question, let us
specify the two dimensional model which plays a central role in this
paper.

We work on the triangular lattice $\mathbb{T}=\left(V,E\right)$ with
its usual embedding in the plane $\mathbb{R}^{2}.$ That is, the vertex
set $V$ is the lattice generated by the vectors $\underline{e}_{1}=\left(1,0\right)$
and $\underline{e}_{2}=\left(\cos\left(\pi/3\right),\sin\left(\pi/3\right)\right):$
\begin{equation}
V:=\left\{ a\underline{e}_{1}+b\underline{e}_{2}\left|a,b\in\mathbb{Z}\right.\right\} .\label{eq: vertex set}
\end{equation}
The vertices $u$ and $v$ are neighbours, i.e $\left(u,v\right)\in E$
or $u\sim v$ if their $L^{2}$ distance is $1.$ We consider the
model where we freeze clusters as soon as they reach $L^{\infty}$
diameter (inherited from $\mathbb{R}^{2}$) at least $N.$ For the
case where the underlying lattice is $\mathbb{Z}^{2}$ and for different
choices for diameters of clusters see the discussion below Conjecture
\ref{conj: scaling limit}. 

Van den Berg, Kiss and Nolin \cite{ECP1694} investigated the edge
version of the $N$-parameter process on the binary tree. They found
that as $N\rightarrow\infty,$ the $N$-parameter process on the binary
tree converges to the $\infty$-parameter process in some weak sense.
This result raises the question if there is a limit of the $N$-parameter
frozen percolation processes on the triangular lattice as $N$ goes
to infinity. The non-existence of the $\infty$-parameter process
suggests that the $N$-parameter model may have a remarkable (anomalous)
behaviour in the limit $N\rightarrow\infty.$ It turns out that there
is a limiting process, but this process is, in some sense, trivial: 
\begin{thm}
\label{thm: origin do not freeze}As $N\rightarrow\infty$ the probability
that in the $N$-parameter frozen percolation process the open cluster
of the origin freezes goes to $0.$
\end{thm}
To get some intuition for the behaviour of the process, let us for
the moment forget about freezing, and call the resulting process the
percolation process.  That is, at time $\tau_{v}$ the vertex $v$
becomes open no matter how big are the open clusters of its neighbours.
Thus at time $t,$ a vertex $v$ is open with probability $t$ independently
from the other vertices. Hence at time $t$ we see ordinary site percolation
with parameter $t.$ Recall from \cite{Russo1981} that the critical
parameter for site percolation on the triangular lattice is $p_{c}=1/2.$
So at each time $t\leq1/2$ there is no open infinite cluster, and
there is a unique infinite open cluster when $t>1/2.$ Moreover, by
\cite{Aizenmann1987} at time $t<1/2,$ the distribution of the size
of the open clusters has an exponential decay. Note that if a site
is open in the $N$-parameter frozen percolation process at time $t,$
then it is also open in the percolation process at time $t.$ Hence
at time $t<1/2$ the $N$-parameter frozen percolation process and
the percolation process does not differ too much when $N$ is large:
even without freezing, for all $K>0$ the probability that there is
an open cluster with diameter at least $N$ in a box with side length
$KN$ goes to $0$ as $N\rightarrow\infty.$ To our knowledge, there
is no simple argument showing that, roughly speaking, freezing does
not take place at times that are essentially bigger than $1/2,$ which
is one of our main results:
\begin{thm}
\label{thm: no supercritical feezing}For all $K>0$ and $t>1/2,$
the probability that after time $t$ a frozen cluster forms which
intersects a given box with side length $KN$ goes to $0$ as $N\rightarrow\infty.$
\end{thm}
Compare Theorem \ref{thm: no supercritical feezing} with \cite{Aldous2000,Berg2012a}
where it was shown that clusters freeze throughout the time horizon
$\left[1/2,1\right]$ for $N\in\mathbb{N}\cup\left\{ \infty\right\} $
in the edge version of the $N$-parameter frozen percolation process
on the binary tree. (Note that the critical parameter is $1/2$ for
site percolation on the binary tree.) As it turns out, our method
provides a much stronger result than Theorem \ref{thm: no supercritical feezing}.
To state it we need some more notation.

Let $\mathbb{P}$ denote the probability measure corresponding to
the percolation process. For a fixed $p\in\left[0,1\right],$ we call
a vertex $v\in V$ $p$-open ($p$-closed), if its $\tau$ value is
less (greater) than $p.$ We denote by $\mathbb{P}_{p}$ the distribution
of $p$-open vertices.

We borrow some of the notation from \cite{Nolin2008}. Recall the
definition of $V$ from (\ref{eq: vertex set}). The $L^{\infty}$
distance of vertices in $\mathbb{T}$ is the $L^{\infty}$ distance
inherited from $\mathbb{R}^{2}.$ That is, for $v,w\in V$ the distance
$d\left(v,w\right)$ between $v=\left(v_{1},v_{2}\right)$ and $w=\left(w_{1},w_{2}\right)$
is 
\begin{align*}
d\left(v,w\right) & =\left\Vert v-w\right\Vert _{\infty}\\
 & =\max\left\{ \left|v_{1}-w_{1}\right|,\left|v_{2}-w_{2}\right|\right\} .
\end{align*}
For $a,b,c,d\in\mathbb{R},$ with $a<b,c<d$ we define the parallelogram
\[
\left[a,b\right]\boxtimes\left[c,d\right]:=\left\{ k\underline{e}_{1}+l\underline{e}_{2}\,|\, k\in\left[a,b\right]\cap\mathbb{Z},\, l\in\left[c,d\right]\cap\mathbb{Z}\right\} .
\]
We denote the outer boundary of a set of vertices $S\subseteq V$
by 
\begin{equation}
\partial S:=\left\{ v\in V\setminus S\,|\,\exists u\in S:\, u\sim v\right\} .\label{eq: def outer boundary}
\end{equation}
Let $cl\left(S\right)=S\cup\partial S$ denote the closure of $S.$
For the parallelogram centred around the vertex $v$ with radius $a>0$
we write 
\[
B\left(v;a\right):=\left[-a,a\right]\boxtimes\left[-a,a\right]+v.
\]
We denote the annulus centred around $v\in V$ with inner radius $a>0$
and outer radius $b>a$ by 
\[
A\left(v;a,b\right):=B\left(v;b\right)\backslash B\left(v;a\right).
\]
We call $B\left(v;a\right)$ the inner, $B\left(v;b\right)$ the outer
parallelogram of $A\left(v;a,b\right).$ 

We say that there is an open (closed) arm in an annulus $A\left(v;a,b\right)$
if there is an open (closed) path from $\partial B\left(v;a\right)$
to $\partial B\left(v;b\right)$ in $A\left(v;a,b\right).$ We write
$o$ for open and $c$ for closed. A colour sequence of length $k$
is an element of $\left\{ o,c\right\} ^{k}.$ For $\sigma\in\left\{ o,c\right\} ^{k},$
we denote by $\mathcal{A}_{k,\sigma}\left(v;a,b\right)$ the event
that there are $k$ disjoint arms in $A\left(v;a,b\right)$ such that
the vertices of each of the arms are either all open or all closed,
moreover, if we take a counter-clockwise ordering of these arms, then
their colours follow a cyclic permutation of $\sigma.$

In the case where $v=\underline{0}=\left(0,0\right)$ we omit the
first argument in our notation, that is $B\left(a\right)=B\left(\underline{0};a\right)$
etc. For the critical arm probabilities we use the notation
\begin{equation}
\pi_{k,\sigma}\left(a,b\right):=\mathbb{P}_{1/2}\left(\mathcal{A}_{k,\sigma}\left(a,b\right)\right).\label{eq:crit arm prob}
\end{equation}

In the following we use the near critical parameter scale which was
introduced in \cite{Garban2010}. For a positive parameter $N$ and
$\lambda\in\mathbb{R}$ it is defined as 
\begin{equation}
p_{\lambda}\left(N\right):=\frac{1}{2}+\lambda\frac{N^{-2}}{\pi_{4,alt}\left(1,N\right)}\label{eq:def p lambda}
\end{equation}
where $alt$ denotes the colour sequence $\left(o,c,o,c\right).$

Before we proceed, let us stop here and let us briefly explain the
formula (\ref{eq:def p lambda}). Suppose that a vertex $v$ is a
closed pivotal vertex, i.e. it is on the boundary of two different
open cluster with diameter at least $N.$ The two open clusters provide
two disjoint open arms starting from neighbouring vertices of $v.$
Since the open clusters are different, they have to be separated by
closed paths, which provide two disjoint closed arms starting from
$v.$ Hence the event $\mathcal{A}_{4,alt}\left(v;1,N\right)$ occurs.
By (\ref{eq:crit arm prob}), we get that the expected number of pivotal
vertices in $B\left(N\right)$ is $O\left(N^{2}\pi_{4,alt}\left(1,N\right)\right).$
Let $\lambda>0$. Let us look at the percolation process in the parallelogram
$B\left(N\right)$ in the time interval $\left[1/2,p_{\lambda}\left(N\right)\right].$
The probability that a vertex opens in this time interval is $p_{\lambda}\left(N\right)-1/2.$
By a combination of (\ref{eq:crit arm prob}) and (\ref{eq:def p lambda})
we see that the expected number of pivotal vertices which open in
this interval is $O\left(1\right).$ Hence the parameter scale in
(\ref{eq:def p lambda}) corresponds to the time scale where open
clusters of diameter $O\left(N\right)$ merge. See \cite{Garban2010,Garban}
for more details.

The considerations above suggest that the parameter scale (\ref{eq:def p lambda})
is indeed useful for investigating the $N$-parameter frozen percolation
process. We write $\mathbb{P}_{N}$ for the probability measure corresponding
to the $N$-parameter frozen percolation process. The following stronger
version of Theorem \ref{thm: no supercritical feezing} is our main
result.
\begin{thm}
\label{thm: main}For any $\varepsilon,K>0$ there exists $\lambda=\lambda\left(\varepsilon,K\right)$
and $N_{0}=N_{0}\left(\varepsilon,K\right)$ such that 
\[
\mathbb{P}_{N}\left(\mbox{a cluster intersecting }B\left(KN\right)\mbox{ freezes after time }p_{\lambda}\left(N\right)\right)<\varepsilon
\]
for all $N\geq N_{0}.$
\end{thm}
In \cite{Berg2012a} the authors investigated the diameter of the
open cluster of the origin at time $1.$ Their main result is the
following.
\begin{defn}
\label{def: open cluster}For $t\in\left[0,1\right]$ let $C\left(v;t\right)$
denote the open cluster of $v\in V$ at time $t\in\left[0,1\right].$
We set $C\left(t\right):=C\left(\underline{0};t\right).$
\end{defn}

\begin{defn}
For $C\subset V,$ let $\diam\left(C\right)$ denote the $L^{\infty}$-diameter
of $C.$\end{defn}
\begin{thm}
[Theorem 1.1 of \cite{Berg2012a}] \label{thm: Rob, Bernardo, Pierre}For
the bond version of the $N$-parameter frozen percolation on the square
lattice we have

\[
\liminf_{N\rightarrow\infty}\mathbb{P}_{N}\left(\diam\left(C\left(1\right)\right)\in\left(aN,bN\right)\right)>0
\]
for $a,b\in\left(0,1\right)$ with $a<b.$
\end{thm}
Analogous result holds for the (site version of) $N$-parameter process
on the triangular lattice. In the following corollary we supplement
this result. It is an extension of Theorem \ref{thm: origin do not freeze}.
\begin{cor}
\label{cor: macroscopic clusters}For any $\varepsilon>0$ there exists
$a=a\left(\varepsilon\right),b=b\left(\varepsilon\right)\in\left(0,1\right)$
with $a<b$ and $N_{0}=N_{0}\left(\varepsilon\right)$ such that
\[
\mathbb{P}_{N}\left(\diam\left(C\left(1\right)\right)\in\left(aN,bN\right)\right)>1-\varepsilon
\]
for all $N\geq N_{0}.$ 
\end{cor}
The results above suggest the following intuitive and informal description
of the behaviour of $N$-parameter frozen percolation processes on
the triangular lattice for large $N$: At time $0$ all the vertices
are closed. Then they open independently from each other as in the
percolation process till time close to $1/2.$ Then in the scaling
window (\ref{eq:def p lambda}), frozen clusters form, and by the
end of the window, they give a tiling of $\mathbb{T}$ such that all
the holes (non-frozen connected components) have diameter less than
$N$ but, typically, of order $N.$ After the window, the closed vertices
in these holes open as in the percolation process restricted to these
holes. At time $1$ the non-frozen vertices are all open.

Hence the interesting time scale is (\ref{eq:def p lambda}), moreover
it raises the question if there is some kind of limiting process which
governs the behaviour of the $N$-parameter frozen percolation processes
as $N\rightarrow\infty$ in the scaling window (\ref{eq:def p lambda}).
We have the following, somewhat informal, conjecture:
\begin{conjecture}
\label{conj: scaling limit}When we scale space by $N$ and time according
to (\ref{eq:def p lambda}), we get a non-trivial scaling limit, which
is measurable with respect to the near critical ensemble of \cite{Garban2010,Garban}.
Moreover, the scaling limit completely describes the frozen clusters
of the $N$-parameter frozen percolation as $N\rightarrow\infty.$
\end{conjecture}
Let us mention some generalizations of our results. We considered
the site version of the $N$-parameter frozen percolation on the triangular
lattice above. Straightforward adaptations of the proofs give the
same results for the bond version of the $N$-parameter frozen percolation
on the square lattice. See Remark \ref{rem: mod to square} for more
details. Our results remain valid when use some different distance
instead of the $L^{\infty}$ distance in the definition of the $N$-parameter
frozen percolation process, as long as the used distance resembles
the $L^{\infty}$ distance. Examples of such distances include the
$L^{p}$ distances for some $p\geq1,$ or when we rotate the lattice
$\mathbb{T}.$ Finally let us mention that when we freeze clusters
when their \emph{volume} (number of its vertices) reach $N,$ we get
a quite different process.

\medskip

Let us briefly discuss some related results. A version of the $N$-parameter
frozen percolation process on $\mathbb{Z}$ and the binary tree were
investigated in \cite{Brouwer2005}. We already referred to \cite{Aldous2000}
where Aldous introduced the $\infty$-parameter frozen percolation
process on the binary tree. However, we did not mention that this
model has another interesting, so called self organized critical (SOC),
behaviour: For all $t>1/2,$ the distribution of the active clusters
at time $t$ have the same distribution as critical clusters. Clearly,
the $N$-parameter frozen percolation process on the triangular lattice
does not have this property. A mean field version of the frozen percolation
model on the complete graph was investigated by R\'{a}th in \cite{Rath2009}.
He showed that this model has similar SOC properties. Let us mention
some results on another closely related model, the so called self-destructive
percolation. Van den Berg and Brouwer \cite{Berg2004} introduced
the model and investigated its properties in the cases where the underlying
graph is the binary tree and the square lattice $\mathbb{Z}^{2}.$
Recently, the model on $\mathbb{Z}^{d}$ for large $d$ \cite{Ahlberg2013a}
and on non-amenable graphs \cite{Ahlberg2013} was investigated. Finally,
we refer to \cite{Bertoin2010} where a dynamics similar to frozen
percolation was investigated on uniform Cayley trees.

\medskip

The organization of the paper is the following. In Section \ref{sec:notation and preliminary res},
we introduce some more notation, and briefly discuss the results from
percolation theory required to prove our main result: We start with
some classical correlation inequalities in Section \ref{sub: corr ineq}.
In Section \ref{sub: mixed arms} we introduce mixed arm events where
some of the arms can use only the upper half of the annulus, while
others can use the whole annulus. Here we also recall some of their
well-known properties and discuss some new ones. In particular, we
note that the exponent of the arm events increases when we increase
the number of arms which have to stay in the upper half plane. The
proof of this statement is postponed to Section \ref{sub: winding}
of the Appendix. In Section \ref{sub: near crit scale - corr length}
we describe the connection between the correlation length with the
near critical scaling (\ref{eq:def p lambda}). We prove Theorem \ref{thm: main}
and Corollary \ref{cor: macroscopic clusters} in Section \ref{sec: pf main}
assuming two technical results Proposition \ref{prop: active diameter}
and \ref{prop: big then freeze}. In Section \ref{sec: pf prop big then freeze}
we introduce some more notation and the notion of thick paths. There
we prove Proposition \ref{prop: big then freeze}. In this proof a
deterministic (combinatorial/geometric) result, Lemma \ref{lem: det gridpath},
plays an important role. The proof of this lemma is postponed to Section
\ref{sub: there are thick paths} of the Appendix. The most technical
part of the paper is Section \ref{sec: pf of prop active diameter}
where we prove Proposition \ref{prop: active diameter}. In Section
\ref{sub: lowest of lowest in parallelograms} and \ref{sub: lowesr of lowest in reular regions}
we investigate the vertical position of the lowest point of the lowest
closed crossing in regions with half open half closed boundary conditions.
We combine these results with the ones in Section \ref{sec:notation and preliminary res}
and conclude the proof of Proposition \ref{prop: active diameter}
in Section \ref{sec: pf prop big then freeze}. This finishes the
proof of the main result.

\section*{Acknowledgement}

The author thanks Jacob van den Berg, Federico Camia, Pierre Nolin
and G\'{a}bor Pete for fruitful discussions. He thanks G\'{a}bor
Pete and Art\"{e}m Sapozhnikov for calling his attention to Lemma
\ref{lem: near critical arms}, which was instrumental in the proof
of the main result. He is grateful to Jacob van den Berg, Ren\'{e}
Conijn and Art\"{e}m Sapozhnikov for their numerous comments on the
earlier versions of the paper. In particular, the author thanks Art\"{e}m
Sapozhnikov for his remarks which led to a simpler proof of Lemma
\ref{lem: lowest of lowest}.

\section{\label{sec:notation and preliminary res}Preliminary results on near
critical percolation}

We recall some classical results from percolation theory in this section.
With suitable modifications, the results of this section also hold
for bond percolation on the square lattice unless it is indicated
otherwise.

\subsection{\label{sub: corr ineq}Correlation inequalities}

We use the following two inequalities throughout the paper. See Section
2.2 and 2.3 of \cite{Grimmett1999} for more details. We refer to
the first theorem as FKG, and as BK for the second.
\begin{defn}
\label{def: inc events}Let $A\subset\left\{ o,c\right\} ^{V}$ and
$U\subseteq V.$ We say that an event $A\subset\left\{ o,c\right\} ^{V}$
is increasing (decreasing) in the configuration in $U,$ if for all
$\omega\in A$ we have $\omega'\in A$ where 
\[
\omega'_{v}=\begin{cases}
\omega_{v}\mbox{ or }o\mbox{ (}c\mbox{)} & \mbox{for }v\in U\\
\omega_{v} & \mbox{for }v\in V\setminus U.
\end{cases}
\]
That is, turning some closed (open) vertices in $U$ into open (closed)
ones can only help the occurrence of $A.$ In the case where $U=V$
we simply say that $A$ is an increasing (decreasing) event. \end{defn}
\begin{thm}
[FKG] \label{thm: FKG}For any pair of increasing events $A,B$ we
have
\[
\mathbb{P}_{p}\left(A\cap B\right)\geq\mathbb{P}_{p}\left(A\right)\mathbb{P}_{p}\left(B\right).
\]

\end{thm}

\begin{thm}
[BK]\label{thm: BK}Let $A,B$ be increasing events, then
\[
\mathbb{P}_{p}\left(A\square B\right)\leq\mathbb{P}_{p}\left(A\right)\mathbb{P}_{p}\left(B\right),
\]
where $A\square B$ denotes the disjoint occurrence of the events
$A$ and $B.$
\end{thm}

\subsection{\label{sub: mixed arms}Mixed arm events, critical arm exponents}

Recall the definition of arm events from the introduction. There the
arms were allowed to use the whole annulus. We introduce the mixed
arm events, where some of the arms lie in the upper half of the annulus,
while others can use the whole annulus:
\begin{defn}
\label{def: full plane mixed arms}Let $l,k\in\mathbb{N}$ with $0\leq l\leq k,$
and a colour sequence $\sigma$ of length $k.$ Let $v\in V$ and
$a,b\in\left(1,\infty\right)$ with $a<b.$ The full plane $k,l$
mixed arm event with colour sequence $\sigma$ in the annulus $A\left(v;a,b\right)$
is denoted by $\mathcal{A}_{k,l,\sigma}\left(v;a,b\right).$ It is
the normal $k$ arm event $\mathcal{A}_{k,\sigma}\left(v;a,b\right)$
of the Introduction with the extra condition that there is a counter-clockwise
ordering of the arms such that the colour of the arms follow $\sigma,$
and the first $l$ arms lie in the half annulus $A\left(v;a,b\right)\cap\left(\mathbb{Z}\boxtimes\left[0,\infty\right)+v\right)$.
When $v=0,$ we omit the first argument from these notations.

We extend the definition (\ref{eq:crit arm prob}) for mixed arm events
by defining 
\[
\pi_{k,l,\sigma}\left(a,b\right):=\mathbb{P}_{1/2}\left(\mathcal{A}_{k,l,\sigma}\left(a,b\right)\right).
\]
\end{defn}
\begin{rem}
In the case $k=l,$ we get the so called half plane arm events. 
\end{rem}
We fix $n_{0}\left(k\right)=10k$ for $k\in\mathbb{N}.$ Note that
the event $\mathcal{A}_{k,l,\sigma}\left(n,N\right)$ is non-empty
whenever $n_{0}\left(k\right)<n<N.$ Let us summarize the known critical
arm exponents for site percolation on the triangular lattice. To our
knowledge, Theorem \ref{thm: arm exponents} in its generality is
not known to hold for bond percolation on $\mathbb{Z}^{2}.$
\begin{thm}
[Theorem 3 and 4 of \cite{Smirnov2001}]\label{thm: arm exponents}
Let $l,k\in\mathbb{N}$ and $\sigma$ be a colour sequence of length
$k.$ We define $a_{k,l}\left(\sigma\right)$ 
\begin{itemize}
\item for $k=1,\,$$l=0$ and any colour sequence $\sigma$ as 
\[
\alpha_{1,0}\left(\sigma\right):=\frac{5}{48},
\]

\item for $k>1$ and $l=0,$ when $\sigma$ contains both colours, as 
\[
\alpha_{k,0}\left(\sigma\right):=\frac{k^{2}-1}{12},
\]

\item for $k=l\geq1$ and any colour sequence $\sigma$ as 
\[
\alpha_{k,k}\left(\sigma\right):=\frac{k\left(k+1\right)}{6}.
\]

\end{itemize}
In these cases we have 
\[
\pi_{k,l,\sigma}\left(n_{0}\left(k\right),N\right)=N^{-\alpha_{k,l}\left(\sigma\right)+o\left(1\right)}
\]
as $N\rightarrow\infty,$ 
\end{thm}
To our knowledge, for general $k$ and $l,$ neither the value, nor
the existence of the exponents is known. We expect that the exponents
do exist. We will see in Proposition \ref{prop: gen mixed arm exp},
that if $\alpha_{k,l}\left(\sigma\right)$ and $\alpha_{k,m}\left(\sigma\right)$
exists for some $k,l,m\in\mathbb{N}$ and $\sigma\in\left\{ o,c\right\} ^{k}$
with $m<l$ , then $\alpha_{k,m}\left(\sigma\right)<\alpha_{k,l}\left(\sigma\right).$
Since we do not need such general result, we only prove the following
proposition in detail.
\begin{prop}
\label{prop: mixed arm exp}For any $k\geq1,$ there are positive
constants $c=c\left(k\right),\,\varepsilon=\varepsilon\left(k\right)$
such that
\begin{equation}
\pi_{k,l,\sigma}\left(n_{0}\left(k\right),N\right)\leq cN^{-\varepsilon}\pi_{k,0,\sigma}\left(n_{0}\left(k\right),N\right)\label{eq: mixed arm exp}
\end{equation}
for $l=1,2,\ldots,k$ uniformly in $N$ and in the colour sequence
$\sigma.$\end{prop}
\begin{rem}
\label{rem: arms with exponent 2}(i) We do not need the exact values
of the critical exponents of Theorem \ref{thm: arm exponents}. For
our purposes it is enough to show that certain arm events have exponents
at least $2.$

(ii) Proposition \ref{prop: mixed arm exp} and its generalization
also hold for mixed arm events in bond percolation on the square lattice. \end{rem}
\begin{proof}
[Proof of Proposition \ref{prop: mixed arm exp}] Proposition \ref{prop: mixed arm exp}
is a simple corollary of Proposition \ref{prop: arms do wind} of
the Appendix. Loosely speaking, it states that conditioning on the
event that we have $k$ arms in $A\left(a,b\right),$ these arms wind
around the origin in $O\left(\log\left(b/a\right)\right)$ disjoint
sub-annuli of $A\left(a,b\right)$ with probability at least $1-\left(\frac{a}{b}\right)^{\kappa}$
for some $\kappa>0.$ The proof of Proposition \ref{prop: arms do wind}
can be found in the Appendix.\end{proof}
\begin{rem}
\label{rem: ord big arm exp} Recall that we do not know in general
if the exponents $\alpha_{k,l}\left(\sigma\right)$ exist or not.
Nonetheless, on the triangular lattice, Proposition \ref{prop: mixed arm exp}
and Theorem \ref{thm: arm exponents} and the BK inequality (Theorem
\ref{thm: BK}) give that for any colour sequence $\sigma,$ there
is an upper bound with exponent strictly larger than $2$ for $\pi_{k,l,\sigma}\left(n_{0}\left(k\right),N\right)$
when
\begin{itemize}
\item $k\geq6,$ and $l\geq0,$ or
\item $k\geq5$ and $l\geq1,$ or
\item $k\geq4$ and $l\geq3$.
\end{itemize}
For arm events with exponents larger than $2$ in the case of bond
percolation on the square lattice see Remark \ref{rem: arms in square lattice}
below.
\end{rem}
Another well-known attribute of critical arm events is their quasi-multiplicative
property. For the full plane, respectively for half plane, arm events
this property is shown to hold in Proposition 17 of \cite{Nolin2008},
respectively in Section 1.4.6 of \cite{Nolin2008}. Simple modifications
of these arguments apply to mixed arm events. We introduce the notation
$\asymp$ when the ratio of the two quantities is bounded away from
$0$ and $\infty.$ We have:
\begin{prop}
\label{prop: quasi-multiplivativity}Let $k\geq1$ and $\sigma\in\left\{ o,c\right\} ^{k}.$
Then 
\begin{align*}
\pi_{k,l,\sigma}\left(n_{1},n_{2}\right)\pi_{k,l,\sigma}\left(n_{2},n_{3}\right) & \asymp\pi_{k,l,\sigma}\left(n_{1},n_{3}\right)
\end{align*}
uniformly in $n_{0}\left(k\right)\leq n_{1}\leq n_{2}\leq n_{3}.$ 
\end{prop}
In the following lemma we consider arm events where the open arms
are $p$-open and the closed arms are $q$-closed where $p,q\in\left[0,1\right]$
with $p$ not necessarily equal to $q.$ When $p$ and $q$ are of
the form (\ref{eq:def p lambda}), then we call these arm events near
critical arm events. In this case the probabilities of these events
are comparable to critical arm event probabilities. The following
lemma is a generalization of Lemma 2.1 of \cite{Garban2008} and Lemma
6.3 of \cite{Damron2009}.
\begin{lem}
\label{lem: near critical arms} Let $v\in V,$ $\lambda_{1},\lambda_{2}\in\mathbb{R}$
and $a,b\in\left(0,1\right)$ with $a<b.$ Let $\mathcal{A}_{k,l,\sigma}^{\lambda_{1},\lambda_{2},N}\left(v;aN,bN\right)$
denote the modification of the event $\mathcal{A}_{k,l,\sigma}\left(v;aN,bN\right)$
where the open arms are $p_{\lambda_{2}}\left(N\right)$-open and
the closed arms are $p_{\lambda_{1}}\left(N\right)$-closed. Then
there are positive constants $c=c\left(\lambda_{1},\lambda_{2},k\right)$
and $N_{0}=N_{0}\left(\lambda_{1},\lambda_{2},a,b,k\right)$ such
that
\[
\mathbb{P}\left(\mathcal{A}_{k,l,\sigma}^{\lambda_{1},\lambda_{2},N}\left(v;aN,bN\right)\right)\leq c\pi_{k,l,\sigma}\left(aN,bN\right)
\]
for $N\geq N_{0}.$\end{lem}
\begin{proof}
[Proof of Lemma \ref{lem: near critical arms}] It follows from
either of the proof of Lemma 2.1 of \cite{Garban2008} or from the
proof of Lemma 6.3 of \cite{Damron2009}.
\end{proof}
In the following events we collect some of the near critical arm events
which have upper bounds with exponents strictly larger than $2.$
These events play a crucial role in our main result.
\begin{defn}
\label{def: NA}Let $a,b\in\left(0,1\right),$ $\lambda_{1},\lambda_{2}\in\mathbb{R},$
$K>0$ and $N\in\mathbb{N}$ with $a<b.$ Let $\mathcal{NA}^{c}\left(a,b,\lambda_{1},\lambda_{2},K,N\right)$
denote the union of the events $\mathcal{A}_{k,l,\sigma}^{\lambda_{1},\lambda_{2},N}\left(v;aN,bN\right)$
for $\left(k,l\right)\in\left\{ \left(4,3\right),\left(5,1\right),\left(6,0\right)\right\} ,$
$\sigma\in\left\{ o,c\right\} ^{k},v\in B\left(KN\right)$ as well
as the versions of these events where the half plane arms can only
use the lower, left or right half of the annulus $A\left(v;aN,bN\right).$
We define $\mathcal{NA}\left(a,b,\lambda_{1},\lambda_{2},K,N\right)$
as the complement of the event above.
\end{defn}
We show that that for fixed $b,K,\lambda_{1}$ and $\lambda_{2},$
we can set $a\in\left(0,1\right)$ so that the probability of $\mathcal{NA}\left(a,b,\lambda_{1},\lambda_{2},K,N\right)$
becomes as close to $1$ as we require for large $N.$ More precisely,
we prove the following:
\begin{cor}
\label{cor: no many arms}There is $\tilde{\varepsilon}>0$ such that
for all $a,b\in\left(0,1\right),$ with $a<b$ and $\lambda_{1},\lambda_{2}\in\mathbb{R}$
there are positive constants $c=c\left(\lambda_{1},\lambda_{2},K\right)$
and $N_{0}=N_{0}\left(a,b,\lambda_{1},\lambda_{2},K\right)$ such
that

\[
\mathbb{P}\left(\mathcal{NA}\left(a,b,\lambda_{1},\lambda_{2},K,N\right)\right)\geq1-c\frac{a^{\tilde{\varepsilon}}}{b^{2+\tilde{\varepsilon}}}
\]
for $N\geq N_{0}.$\end{cor}
\begin{proof}
[Proof of Corollary \ref{cor: no many arms}] Suppose that one of
the arm events in Definition \ref{def: NA}, for example $\mathcal{A}_{k,l,\sigma}^{\lambda_{1},\lambda_{2},N}\left(v;aN,bN\right)$
for some $v\in B\left(KN\right),$ occurs. Then the event $\mathcal{A}_{k,l,\sigma}^{\lambda_{1},\lambda_{2},N}\left(\left\lfloor 2aN\right\rfloor z;2aN,\frac{b}{2}N\right)$
occurs for some $z\in V$ with $z\in B\left(\left\lceil \frac{a+K}{2a}\right\rceil \right).$ 

Combination of Remark \ref{rem: ord big arm exp} and Lemma \ref{lem: near critical arms}
gives that there are constants $c'=c'\left(\lambda_{1},\lambda_{2}\right),$
$N_{0}=N_{0}\left(a,b,\lambda_{1},\lambda_{2}\right),$ and a universal
constant $\tilde{\varepsilon}>0$ such that the probability of one
of these events is at most 
\begin{equation}
c'\left(\frac{2a}{b/2}\right)^{2+\varepsilon}\label{eq: pf cor no many arms}
\end{equation}
for $N\geq N_{0}.$ The same argument works for other arm events which
appear in Definition \ref{def: NA}, and provide an upper bound similar
to (\ref{eq: pf cor no many arms}). Hence (\ref{eq: pf cor no many arms})
combined with $\left|B\left(\left\lceil \frac{a+K}{2a}\right\rceil \right)\right|=O\left(a^{-2}\right)$
concludes the proof of Corollary \ref{cor: no many arms}.\end{proof}
\begin{rem}
\label{rem: arms in square lattice}To our knowledge it is not known
if the direct analogue of Corollary \ref{cor: no many arms} holds
on the square lattice. The reason is that the exponent $\alpha_{5,0}\left(\sigma\right)$
and $\alpha_{3,3}\left(\sigma\right)$ is not known for general $\sigma.$
See Remark 26 of \cite{Nolin2008}.

We recall the proof of Theorem 24 and Remark 26 of \cite{Nolin2008},
where it is shown that $\alpha_{5,0}\left(o,c,o,o,c\right)=2$ and
$\alpha_{3,3}\left(c,o,c\right)=2$ on the square lattice. This implies
that a version of Corollary \ref{cor: no many arms} holds for the
square lattice if we modify Definition \ref{def: NA} so that we only
forbid the occurrence of those arm events where the required set of
arms contain
\begin{itemize}
\item three half plane arm events with colour sequence $\left(o,c,o\right)$
or $\left(c,o,c\right),$ or
\item five full plane arms with colour sequence $\left(o,c,o,o,c\right)$
or $\left(c,o,c,c,o\right)$
\end{itemize}
as a subset.
\end{rem}

\subsection{\label{sub: near crit scale - corr length}Near-critical scaling
and correlation length}

Recall that in Section \ref{sec: introduction} we already gave an
explanation for the near critical parameter scale (\ref{eq:def p lambda}).
In this section we give a different interpretation of this parameter
scale, which is connected to the correlation length introduced by
Kesten in \cite{Kesten1987}.

We say that there is an open (closed) horizontal crossing of a parallelogram
$B:=\left[a,b\right]\boxtimes\left[c,d\right]$ if there is an open
(closed) path connecting $\left\{ \left\lceil a\right\rceil \right\} \boxtimes\left[c,d\right]$
and $\left\{ \left\lfloor b\right\rfloor \right\} \boxtimes\left[c,d\right]$
in$\left[a,b\right]\boxtimes\left[c,d\right].$ For the event that
there is an open (closed) horizontal crossing of $B$ we use the notation
$\mathcal{H}_{o}\left(B\right)$ ($\mathcal{H}_{c}\left(B\right)$).
One can define similar events for vertical crossings, which we denote
by $\mathcal{V}_{o}\left(B\right)$ and $\mathcal{V}_{c}\left(B\right).$
For $\varepsilon\in\left(0,1/2\right)$ the correlation length is
defined as

\[
L_{\varepsilon}\left(p\right)=\begin{cases}
\min\left\{ n\,|\,\mathbb{P}_{p}\left(\mathcal{H}_{o}\left(B\left(n\right)\right)\right)\leq\varepsilon\right\}  & \mbox{when }p<p_{c}\\
\min\left\{ n\,|\,\mathbb{P}_{p}\left(\mathcal{H}_{o}\left(B\left(n\right)\right)\right)\geq1-\varepsilon\right\}  & \mbox{when }p>p_{c}.
\end{cases}
\]

\begin{rem}
The particular choice of $\varepsilon$ is not important in this definition.
Indeed, Corollary 37 of \cite{Nolin2008}, or alternatively Corollary
2 of \cite{Kesten1987}, gives that 
\[
L_{\varepsilon}\left(p\right)\asymp L_{\varepsilon'}\left(p\right)
\]
for any $\varepsilon,\varepsilon'\in\left(0,1/2\right)$ uniformly
in $p\in\left(0,1\right).$ 
\end{rem}
\medskip

We show that the control over the near critical parameter $\lambda$
gives a control over the correlation length in Corollary \ref{cor: char length is O(N)}
and \ref{cor: small char length} below. Recall the remark after Lemma
8 of \cite{Kesten1987}:
\begin{prop}
\label{prop: Kesten}For any fixed $\varepsilon\in\left(0,1/2\right),$
we have
\[
\left|p-p_{c}\right|\left(L_{\varepsilon}\left(p\right)\right)^{2}\pi_{4,0,alt}\left(1,L_{\varepsilon}\left(p\right)\right)\asymp1
\]
uniformly for $p\neq1/2.$
\end{prop}
Note that for fixed $\varepsilon>0,$ the correlation length $L_{\varepsilon}\left(p\right)$
is a decreasing (increasing) function of $p$ for $p>p_{c}$ ($p<p_{c}$).
Combination of this and Proposition \ref{prop: quasi-multiplivativity}
we get:
\begin{cor}
\label{cor: char length is O(N)}For all $\lambda\in\mathbb{R}\setminus\left\{ 0\right\} $
and $\varepsilon\in\left(0,1/2\right),$ 
\begin{equation}
L_{\varepsilon}\left(p_{\lambda}\left(N\right)\right)\asymp N.\label{eq: bd on corr length}
\end{equation}

\end{cor}

\begin{cor}
\label{cor: small char length}For any $C>0$ and $\varepsilon\in\left(0,1/2\right)$
there exits $\lambda_{1}=\lambda_{1}\left(C,\varepsilon\right)>0$
and $N_{1}=N_{1}\left(C,\varepsilon\right)$ such that for any $\lambda\in\mathbb{R}$
with $\left|\lambda\right|\geq\lambda_{1}$ we have 
\[
L_{\varepsilon}\left(p_{\lambda}\left(N\right)\right)\leq CN
\]
for $N\geq N_{1}.$ Also, for any $c>0,$ and $\varepsilon\in\left(0,1/2\right)$
there exists $\lambda_{2}\left(c,\varepsilon\right)>0$ and $N_{2}=N_{2}\left(c,\varepsilon\right)$
such that for any $\lambda\in\mathbb{R}\setminus\left\{ 0\right\} $
with $\left|\lambda\right|\leq\lambda_{2}$ we have 
\[
L_{\varepsilon}\left(p_{\lambda}\left(N\right)\right)\geq cN
\]
for $N\geq N_{2}.$\end{cor}
\begin{rem}
On the triangular lattice, a ratio limit theorem for $\pi_{4,0,alt},$
Proposition 4.7 of \cite{Garban2010} holds. This combined with the
definition of $L_{\varepsilon}\left(p\right),$ and Proposition \ref{prop: Kesten}
shows that the following stronger statement holds on the triangular
lattice: 
\begin{claim*}
For all $\lambda_{1},\lambda_{2}\in\mathbb{R}$ with $\lambda_{1}\leq\lambda_{2},$
$\lambda_{1}\lambda_{2}>0$ and $\varepsilon\in\left(0,1/2\right)$
there are positive constants $c=c\left(\varepsilon\right),C=C\left(\varepsilon\right)$
and $N_{0}=N_{0}\left(\varepsilon,\lambda_{1},\lambda_{2}\right)$
such that
\[
cN\left|\lambda\right|^{-4/3}\leq L_{\varepsilon}\left(p_{\lambda}\left(N\right)\right)\leq CN\left|\lambda\right|^{-4/3}
\]
for all $\lambda\in$$\left[\lambda_{1},\lambda_{2}\right]$ and $N\geq N_{0}.$
\end{claim*}
\end{rem}
Standard Russo-Seymour-Welsh (RSW) techniques and the definition of
the correlation length give that the control over the correlation
length gives a control over the crossing probabilities of parallelograms.
This combined with the two corollaries above show that the control
over the near critical parameter gives control over the crossing probabilities.
See Corollary \ref{cor: posprob of crossing} and \ref{cor: set prob of crossing}
below:
\begin{cor}
\label{cor: posprob of crossing}For all $\lambda\in\mathbb{R}$ and
$a,b\in\left(0,\infty\right),$ there are constants $c=c\left(a,b,\lambda\right)\in\left(0,1\right),$
$C=C\left(a,b,\lambda\right)\in\left(0,1\right)$ and $N_{0}=N_{0}\left(a,b,\lambda\right)$
such that 
\begin{align*}
c & <\mathbb{P}_{p_{\lambda}\left(N\right)}\left(\mathcal{H}_{o}\left(\left[0,aN\right]\boxtimes\left[0,bN\right]\right)\right)<C\\
c & <\mathbb{P}_{p_{\lambda}\left(N\right)}\left(\mathcal{H}_{c}\left(\left[0,aN\right]\boxtimes\left[0,bN\right]\right)\right)<C
\end{align*}
for $N\geq N_{0}.$
\end{cor}

\begin{cor}
\label{cor: set prob of crossing}Let $\delta\in\left(0,1\right),$
and $a,b\in\left(0,\infty\right).$ There exists $\lambda_{1}=\lambda_{1}\left(\delta,a,b\right)>0$
and $N_{1}=N_{1}\left(\delta,a,b\right)$ such that for all $\lambda\geq\lambda_{1}$

\[
\mathbb{P}_{p_{\lambda}\left(N\right)}\left(\mathcal{H}_{o}\left(\left[0,aN\right]\boxtimes\left[0,bN\right]\right)\right)>1-\delta
\]
for $N\geq N_{1}.$ Furthermore, there exists $\lambda_{2}=\lambda_{2}\left(\delta,a,b\right)<0$
and $N_{2}=N_{2}\left(\delta,a,b\right)$ such that for all $\lambda\leq\lambda_{2}$ 

\[
\mathbb{P}_{p_{\lambda}\left(N\right)}\left(\mathcal{H}_{c}\left(\left[0,aN\right]\boxtimes\left[0,bN\right]\right)\right)>1-\delta
\]
for $N\geq N_{2}.$
\end{cor}
Similar RSW techniques show that it is unlikely to have crossing in
a thin and long parallelogram in the hard direction in the critical
window. See Remark 40 \cite{Nolin2008} for more details.
\begin{cor}
\label{cor: thin crossing}Let $\lambda\in\mathbb{R},$ and $a,b\in\left(0,1\right).$
There exists positive constants $c=c\left(\lambda\right),C=C\left(\lambda\right)$
and $N_{0}=N_{0}\left(\lambda,a,b\right)$ such that

\[
\mathbb{P}_{p_{\lambda}\left(N\right)}\left(\mathcal{H}_{o}\left(\left[0,aN\right]\boxtimes\left[0,bN\right]\right)\right)\leq C\exp\left(-c\frac{a}{b}\right)
\]
for $N\geq N_{0}.$ 
\end{cor}
The following event plays a crucial role in the proof of our main
result.
\begin{defn}
Let $a,b\in\left(0,1\right),$ $\lambda_{1},\lambda_{2}\in\mathbb{R},$
and $N\in\mathbb{N}$ with $a<b.$ Let $\mathcal{NC}\left(a,b,\lambda_{1},\lambda_{2},K,N\right)$
denote the event that for all parallelograms $B=\left[0,aN\right]\boxtimes\left[0,bN\right]+z$
with $z\in B\left(KN\right),$ there is neither a $p_{\lambda_{1}}\left(N\right)$-open
nor a $p_{\lambda_{2}}\left(N\right)$-closed horizontal crossing
in $B.$
\end{defn}
The following Corollary \ref{cor: no thin crossing} follows from
Corollary \ref{cor: thin crossing} by arguments analogous to the
proof of Corollary \ref{cor: no many arms}. 
\begin{cor}
\label{cor: no thin crossing}Let $a,b\in\left(0,1\right),$ $\lambda_{1},\lambda_{2}\in\mathbb{R},$
and $N\in\mathbb{N}$ with $a<b.$ There are positive constants $c=c\left(\lambda_{1},\lambda_{2}\right),\, C=C\left(\lambda_{1},\lambda_{2}\right)$
and $N_{0}=N_{0}\left(a,b,\lambda_{1},\lambda_{2}\right)$ such that
\[
\mathbb{P}\left(\mathcal{NC}\left(a,b,\lambda_{1},\lambda_{2},K,N\right)\right)\geq1-Ca^{-2}\exp\left(-c\frac{b}{a}\right)
\]
for $N\geq N_{0}.$
\end{cor}
We finish this section by stating two lemmas which will be used explicitly
in the proof of our main result.
\begin{lem}
\label{lem: not too many small clusters at lambda}For any fixed $\lambda\in\mathbb{R},$
for any $a,b\in\left(0,\infty\right)$ and $\varepsilon>0,$ there
is are positive integer $k=k\left(\lambda,a,b,\varepsilon\right)$
and $N_{0}=N_{0}\left(\lambda,a,b,\varepsilon\right)$ such that
\[
\mathbb{P}_{p_{\lambda}\left(N\right)}\left(\mbox{there are at least }k\mbox{ disjoint closed arms in }A\left(aN,bN\right)\right)<\varepsilon
\]
for $N\geq N_{0}.$\end{lem}
\begin{proof}
[Proof of Lemma \ref{lem: not too many small clusters at lambda}]
This is a consequence of Corollary \ref{cor: posprob of crossing}
and the BK inequality (Theorem \ref{thm: BK}). The proof also appears
in the proof of Lemma 15 of \cite{Nolin2008}. \end{proof}
\begin{defn}
\label{def: N_c} Let $a,b,c,d,f\in\mathbb{R}$ with $a\leq b,$ $c\leq d$
and $f>0.$ We say that there is an open (closed) $f$-net in $B=\left[a,b\right]\boxtimes\left[c,d\right]$
if there is an open (closed) vertical crossing in the parallelograms
$\left[a+i\left\lfloor f\right\rfloor ,a+\left(i+1\right)\left\lfloor f\right\rfloor -1\right]\boxtimes\left[c,d\right],$
and there is an open (closed) horizontal crossing in the parallelograms
$\left[a,b\right]\boxtimes\left[c+j\left\lfloor f\right\rfloor ,c+\left(j+1\right)\left\lfloor f\right\rfloor -1\right]$
for $i=0,1,\ldots,\left\lfloor \left(b-a\right)/\left\lfloor f\right\rfloor \right\rfloor $
and $j=0,1,\ldots,\left\lfloor \left(d-c\right)/\left\lfloor f\right\rfloor \right\rfloor .$

For $\lambda\in\mathbb{R}$ and $\delta\in\left(0,\infty\right),$
$\mathcal{N}_{c}\left(\lambda,\delta,K,N\right)$ ($\mathcal{N}_{o}\left(\lambda,\delta,K,N\right)$)
denotes the event that there is a $p_{\lambda}\left(N\right)$-closed
($p_{\lambda}\left(N\right)$-open) $\delta N$-net in $B\left(KN\right).$\end{defn}
\begin{lem}
\label{lem: there is a net}Let $\varepsilon,\delta,K>0.$ There exists
$\lambda_{1}=\lambda_{1}\left(\varepsilon,\delta,K\right)\in\mathbb{R}$
and $N_{1}=N_{1}\left(\varepsilon,\delta,K\right)$ such that 
\[
\mathbb{P}\left(\mathcal{N}_{o}\left(\lambda_{1},\delta,K,N\right)\right)>1-\varepsilon
\]
for $N\geq N_{1}.$ Moreover there exists $\lambda_{2}=\lambda_{2}\left(\varepsilon,\delta,K\right)\in\mathbb{R}$
and $N_{2}=N_{2}\left(\varepsilon,\delta,K\right)$ such that 
\[
\mathbb{P}\left(\mathcal{N}_{c}\left(\lambda_{2},\delta,K,N\right)\right)>1-\varepsilon
\]
for $N\geq N_{2}.$\end{lem}
\begin{proof}
[Proof of Lemma \ref{lem: there is a net}] This is a consequence
of Corollary $\ref{cor: set prob of crossing}$ and the FKG inequality
(Theorem \ref{thm: FKG}).
\end{proof}

\section{\label{sec: pf main}Proof of the main results}

We prove our main results Theorem \ref{thm: main} and Corollary \ref{cor: macroscopic clusters}
in this section assuming Proposition \ref{prop: active diameter}
and \ref{prop: big then freeze}.
\begin{defn}
In the $N$-parameter frozen percolation process we call a vertex
\emph{frozen} at some time $t\in\left[0,1\right]$, if either it or
one of its neighbours have an open cluster with diameter bigger than
$N$ at time $t.$ If a site is not frozen at time $t,$ then we say
it is \emph{active} at time $t.$ Note that both frozen and active
sites can be open or closed. We say that $F$ is a \emph{(open) frozen
cluster} at time $t\in\left[0,1\right]$ if it is a connected component
of the open vertices at time $t$ with $\diam\left(F\right)\geq N.$
In the case where $t=1,$ we simply say that $F$ is a frozen cluster.
\end{defn}
Recall Definition \ref{def: N_c}. We observe the following.

\begin{obs} \label{obs: closed grid-> no freezing}Let $K>0$ and
$N\in\mathbb{N}.$ Then in the $N$-parameter frozen percolation process
there is no frozen cluster at time $p_{\lambda}\left(N\right)$ in
$B\left(KN\right)$ on the event $\mathcal{N}_{c}\left(\lambda,1/6,K+2,N\right).$
Hence on $\mathcal{N}_{c}\left(\lambda,1/6,K+2,N\right),$ a vertex
in $B\left(KN\right)$ is open (closed) in the $N$-parameter frozen
percolation process at time $p_{\lambda}\left(N\right)$ if and only
if it is $p_{\lambda}\left(N\right)$-open ($p_{\lambda}\left(N\right)$-closed).

\end{obs}

We show that the number of frozen clusters intersecting $B\left(KN\right)$
in the $N$-parameter frozen percolation process is tight in $N.$
\begin{lem}
\label{lem: few frozen clusters}Let $K>0$ and $N\in\mathbb{N}.$
Let $FC\left(t,K,N\right)$ denote the number of frozen clusters intersecting
$B\left(KN\right)$ at time $t\in\left[0,1\right]$ in the $N$-parameter
frozen percolation process. Then for all $\varepsilon>0$ there exists
$L=L\left(\varepsilon,K\right)$ and $N_{0}=N_{0}\left(\varepsilon,K\right)$
such that 
\[
\mathbb{P}_{N}\left(FC\left(1,K,N\right)>L\right)<\varepsilon
\]
for $N\geq N_{0}.$\end{lem}
\begin{proof}
[Proof of Lemma \ref{lem: few frozen clusters}] By Lemma \ref{lem: there is a net}
we set $\lambda=\lambda\left(\varepsilon,K\right)\in\mathbb{R}$ such
that 
\begin{align}
\mathbb{P}_{N}\left(\mathcal{N}_{c}\left(\lambda,1/6,K+4,N\right)\right) & >1-\frac{1}{2}\varepsilon\label{eq: pf few frozen clusters - 1}
\end{align}
for $N\geq N_{1}\left(\varepsilon,K\right).$ Let $F$ be an open
frozen cluster which intersects $B\left(KN\right).$ From Observation
\ref{obs: closed grid-> no freezing} we get the vertices of $\partial F$
are closed at $p_{\lambda}\left(N\right)$ in the $N$-parameter percolation
process on the event $\mathcal{N}_{c}\left(\lambda,1/6,K+4,N\right).$

Let us cover the parallelogram $B\left(KN\right)$ with the annuli
\[
A_{z}=A\left(\left\lfloor N/20\right\rfloor z;\left\lfloor N/20\right\rfloor ,\left\lfloor N/10\right\rfloor \right)\mbox{ with }z\in B\left(\left\lceil 20K\right\rceil \right).
\]
Suppose that there is an open frozen cluster in the $N$-parameter
frozen percolation which has a vertex in $B\left(KN\right).$ The
construction of the annuli above gives that there is $z\in B\left(\left\lceil 20K\right\rceil \right)$
such that $B\left(\left\lfloor N/20\right\rfloor z;\left\lfloor N/20\right\rfloor \right),$
the inner parallelogram of $A_{z},$ contains a vertex of this open
frozen cluster. Since the diameter of $B\left(\left\lfloor N/20\right\rfloor z;\left\lfloor N/10\right\rfloor \right)$
is less than $N,$ this cluster has to cross the annulus $A_{z}.$
Hence for each open frozen cluster intersecting $B\left(KN\right),$
we find at least one open frozen crossing of an annulus $A_{z}$.
Moreover, if there are $k\geq2$ different frozen clusters crossing
the annulus $A_{z},$ then there are at least $k$ disjoint closed
frozen arms which separate the open frozen clusters in $A_{z}$ at
time $1.$ By the arguments above, these arms are $p_{\lambda}\left(N\right)$-closed.
Thus the number of different frozen clusters intersecting $B\left(\left\lfloor N/20\right\rfloor z;\left\lfloor N/20\right\rfloor \right)$
is bounded above by $1\vee l_{z},$ where $l_{z}$ is the number of
disjoint $p_{\lambda}\left(N\right)$-closed arms of $A_{z}.$ Hence
by the translation variance of the $N$-parameter frozen percolation
process we have 
\begin{align}
\mathbb{P}_{N}\left(FC\left(1,K,N\right)\geq L,\,\mathcal{N}_{c}\left(\lambda,1/24\right)\right) & \leq\mathbb{P}_{p_{\lambda}\left(N\right)}\left(\sum_{z\in B\left(\left\lceil 20K\right\rceil \right)}\left(1\vee l_{z}\right)\geq L\right)\nonumber \\
 & \leq\mathbb{P}_{p_{\lambda}\left(N\right)}\left(\exists z\in B\left(\left\lceil 20K\right\rceil \right)\mbox{ such that }l_{z}\geq\left(2\left\lceil 20K\right\rceil +1\right)^{-2}L\right)\nonumber \\
 & \leq\left(2\left\lceil 20K\right\rceil +1\right)^{2}\mathbb{P}_{p_{\lambda}\left(N\right)}\left(l_{0}\geq\left(2\left\lceil 20K\right\rceil +1\right)^{-2}L\right)\label{eq: pf few frozen clusters - 2}
\end{align}

By Lemma \ref{lem: not too many small clusters at lambda} we set
$L=L\left(\varepsilon,K\right)\geq\left(2\left\lceil 20K\right\rceil +1\right)^{2}$
and $N_{2}=N_{2}\left(\varepsilon,K\right)$ such that
\[
\mathbb{P}_{p_{\lambda}\left(N\right)}\left(l_{0}\geq L/200^{2}\right)<\frac{1}{2}\left(2\left\lceil 20K\right\rceil +1\right)^{-2}\varepsilon
\]
for $N\geq N_{2}.$ This combined with (\ref{eq: pf few frozen clusters - 2})
gives that 
\begin{equation}
\mathbb{P}_{N}\left(FC\left(1,K,N\right)\geq L,\,\mathcal{N}_{c}\left(\lambda,1/6,K+4,N\right)\right)<\frac{1}{2}\varepsilon\label{eq: pf few frozen clusters - 3}
\end{equation}
for $N\geq N_{2}.$ We set $N_{0}:=N_{1}\vee N_{2}.$ A combination
of (\ref{eq: pf few frozen clusters - 1}) and (\ref{eq: pf few frozen clusters - 3})
finishes the proof of Lemma \ref{lem: few frozen clusters}.\end{proof}
\begin{defn}
\label{def: C_a} For $v\in V$ and $\lambda\in\mathbb{R}$ let $\mathcal{C}_{a}\left(v;\lambda\right)=\mathcal{C}_{a}\left(v;\lambda,N\right)$
denote the active cluster of $v$ in the $N$-parameter frozen percolation
process at time $p_{\lambda}\left(N\right).$ We omit the first argument
from the notation above when $v=\underline{0}.$
\end{defn}
We state the two propositions below which play a crucial role in the
proof of Theorem \ref{thm: main}. The proof of these propositions
are rather technical, so we postpone them to the next section. The
first proposition shows that for $\alpha>0,$ it is unlikely to have
an active cluster at time $p_{\lambda}\left(N\right)$ which intersects
$B\left(KN\right)$ and has diameter close to $\alpha N.$ 
\begin{prop}
\label{prop: active diameter}For all $\lambda\in\mathbb{R}$ and
$\varepsilon,K,\alpha>0,$ there exist $\theta=\theta\left(\lambda,\alpha,\varepsilon,K\right)\in\left(0,1/2\right)$
and $N_{0}=N_{0}\left(\lambda,\alpha,\varepsilon,K\right)$ such that

\[
\mathbb{P}_{N}\left(\exists v\in B\left(KN\right)\mbox{ s.t. }\diam\left(\mathcal{C}_{a}\left(v;\lambda\right)\right)\in\left(\left(\alpha-\theta\right)N,\left(\alpha+\theta\right)N\right)\right)<\varepsilon
\]
for $N\geq N_{0}.$
\end{prop}
The second proposition claims that if there is a vertex $v$ such
that $\diam\left(\mathcal{C}_{a}\left(v;\lambda_{1},N\right)\right)\geq\left(1+\theta\right)N$
then some part of $\mathcal{C}_{a}\left(v;\lambda_{1},N\right)$ freezes
`soon`:
\begin{prop}
\label{prop: big then freeze} Let $\theta\in\left(0,1\right),$ $\varepsilon>0$
and $\lambda_{1},K,\in\mathbb{R}.$ Recall the notation $FC\left(t,K+2,N\right)$
from Lemma \ref{lem: few frozen clusters}. There exists $\lambda_{2}=\lambda_{2}\left(\lambda_{1},\theta,\varepsilon\right)$
and $N_{0}=N_{0}\left(\lambda_{1},\theta,\varepsilon\right)$ such
that the probability of the intersection of the events
\begin{itemize}
\item $\exists v\in B\left(KN\right)$ such that $\diam\left(\mathcal{C}_{a}\left(v;\lambda_{1},N\right)\right)\geq\left(1+\theta\right)N,$
and 
\item none of the clusters intersecting $B\left(\left(K+2\right)N\right)$
freeze in the time interval $\left(p_{\lambda_{1}}\left(N\right),p_{\lambda_{2}}\left(N\right)\right],$
i.e. 
\[
FC\left(p_{\lambda_{1}}\left(N\right),K+2,N\right)=FC\left(p_{\lambda_{2}}\left(N\right),K+2,N\right)
\]

\end{itemize}
is less than $\varepsilon$ for $N\geq N_{0}.$
\end{prop}
Before we turn to the proof of our main results we make a remark on
how to adapt the proofs for the $N$-parameter frozen bond percolation
process on the square lattice.
\begin{rem}
\label{rem: mod to square} The arguments in Section \ref{sec: pf main},
\ref{sec: pf prop big then freeze}, \ref{sec: pf of prop active diameter}
and in the Appendix can be easily adapted to the $N$-parameter frozen
bond percolation on the square lattice. Some care is required when
we use Corollary \ref{cor: no many arms}: As we already noted in
Remark \ref{rem: arms in square lattice}, the direct analogue of
Corollary \ref{cor: no many arms} does not hold on the square lattice.
However, one can check that the version of Corollary \ref{cor: no many arms}
which was proposed in Remark \ref{rem: arms in square lattice} is
enough for the proofs appearing in Section \ref{sec: pf main}, \ref{sec: pf prop big then freeze},
\ref{sec: pf of prop active diameter}.
\end{rem}

\subsection{\label{sub: pf main}Proof of Theorem \ref{thm: main}}
\begin{proof}
[Proof of Theorem \ref{thm: main}] The proof follows the following
informal strategy. Consider the following procedure. We set $\lambda_{1}=0.$
We look at the $N$-parameter percolation process at time $p_{\lambda_{1}}\left(N\right).$
We have two cases.

In the first case all the active clusters at time $p_{\lambda_{1}}\left(N\right)$
intersecting $B\left(KN\right)$ have diameter less than $N.$ Hence
no cluster intersecting $B\left(KN\right)$ can freeze after $p_{\lambda_{1}}\left(N\right).$
We terminate the procedure. 

In the second case there is $v\in B\left(KN\right)$ such that the
active cluster $\mathcal{C}_{a}\left(v;\lambda_{1},N\right)$ has
diameter at least $N.$ Using Proposition \ref{prop: active diameter}
we set $\theta_{1}$ such that the diameter of this cluster is at
least $\left(1+\theta_{1}\right)N$ with probability close to $1.$
If $\diam\left(\mathcal{C}_{a}\left(v;\lambda_{1},N\right)\right)\leq\left(1+\theta_{1}\right)N,$
then we stop the procedure. If $\diam\left(\mathcal{C}_{a}\left(v;\lambda_{1},N\right)\right)>\left(1+\theta_{1}\right)N,$
then using Proposition \ref{prop: big then freeze} we set $\lambda_{2}\geq\lambda_{1}$
such that some part of $\mathcal{C}_{a}\left(v;\lambda_{1},N\right)\cap B\left(\left(K+2\right)N\right)$
freezes in the time interval $\left[p_{\lambda_{1}}\left(N\right),p_{\lambda_{2}}\left(N\right)\right]$
with probability close to $1.$ If indeed some part of $\mathcal{C}_{a}\left(v;\lambda_{1},N\right)\cap B\left(\left(K+2\right)N\right)$
freezes in the time interval $\left[p_{\lambda_{1}}\left(N\right),p_{\lambda_{2}}\left(N\right)\right],$
then we iterate the procedure starting from time $p_{\lambda_{2}}\left(N\right).$
Otherwise we terminate the procedure.

Using Lemma \ref{lem: few frozen clusters} we set $L$ such that
the event where there are at least $L$ frozen clusters intersecting
$B\left(\left(K+2\right)N\right)$ at time $1$ has probability smaller
than $\varepsilon/2.$ In each step of the procedure either the procedure
stops, or the number of frozen clusters intersecting $B\left(\left(K+2\right)N\right)$
increases by at least $1.$ Hence the event that the procedure runs
for at least $L$ steps has probability at most $\varepsilon/2.$ 

Moreover, we set the parameters $\lambda_{i},\theta_{i}$ for $i\geq1$
above such that with probability at least $1-\varepsilon/2$ we terminate
the procedure when there are no active clusters intersecting $B\left(KN\right)$
with diameter at least $N.$ Thus with probability at least $1-\varepsilon$
the procedure stops within $L$ steps, and we stop when there are
no active clusters with diameter at least $N$ intersecting $B\left(KN\right).$
Hence $\lambda=\lambda_{L+1}$ satisfies the conditions of Theorem
\ref{thm: main}, which finishes the proof of Theorem \ref{thm: main}.

\medskip

Let us turn to the precise proof. By Lemma \ref{lem: few frozen clusters},
there is $L=L\left(\varepsilon,K\right)$ and $N_{1}'=N_{1}'\left(\varepsilon,K\right)$
such that 
\begin{equation}
\mathbb{P}_{N}\left(FC\left(1,K+2,N\right)\geq L\right)\leq\varepsilon/2,\label{eq: pf main - 1}
\end{equation}
where $F\left(t,K+2,N\right)$ counts the number of frozen clusters
intersecting $B\left(\left(K+2\right)N\right)$ at time $t\in\left[0,1\right].$ 

We define the deterministic sequence $\left(\lambda_{i},N_{i}',\theta_{i},N_{i}''\right)_{i\in\mathbb{N}}$
inductively as follows. We start by setting $\lambda_{1}=0.$ 

Suppose that we have already defined $\lambda_{i}$ for some $i\in\mathbb{N}.$
We use Proposition \ref{prop: active diameter} to set $\theta_{i}=\theta_{i}\left(\varepsilon\right)$
and $N_{i}''=N_{i}''\left(\varepsilon\right)$ such that 
\[
\mathbb{P}_{N}\left(\exists v\in B\left(KN\right)\mbox{ s.t. }\diam\left(\mathcal{C}_{a}\left(v,\lambda_{i}\right)\right)\in\left[N,\left(1+\theta_{i}\right)N\right)\right)<\varepsilon2^{-i-2}
\]
 for $N\geq N_{i}''.$

Suppose that we have already defined $\theta_{i}$ for some $i\in\mathbb{N}.$
Then by Proposition \ref{prop: big then freeze} we set $\lambda_{i+1}=\lambda_{i+1}\left(\varepsilon\right)$
and $N_{i+1}'=N_{i+1}'\left(\varepsilon\right)$ such that the probability
of the intersection of the events
\begin{itemize}
\item $\exists v\in B\left(KN\right)$ such that $\diam\left(\mathcal{C}_{a}\left(v;\lambda_{i}\right)\right)\geq\left(1+\theta_{i}\right)N,$
and 
\item $FC\left(p_{\lambda_{i}}\left(N\right),K+2,N\right)=FC\left(p_{\lambda_{i+1}}\left(N\right),K+2,N\right)$
\end{itemize}
is less than $2^{-i-2}\varepsilon$ for $N\geq N_{i+1}'.$ Note that
the event 
\[
\left\{ FC\left(p_{\lambda_{i}}\left(N\right),K+2,N\right)=FC\left(p_{\lambda_{i+1}}\left(N\right),K+2,N\right),\, FC\left(p_{\lambda_{i}}\left(N\right),K,N\right)<FC\left(1,K,N\right)\right\} 
\]
is a subset of the union of the events appearing in the definition
of $\theta_{i}$ and $\lambda_{i+1}$ for $i\geq1.$ Thus the construction
above gives that 
\begin{align}
\mathbb{P}_{N}\left(FC\left(p_{\lambda_{i}}\left(N\right),K+2,N\right)=FC\left(p_{\lambda_{i+1}}\left(N\right),K+2,N\right),\, FC\left(p_{\lambda_{i}}\left(N\right),K,N\right)<FC\left(1,K,N\right)\right) & \leq2^{-i-1}\varepsilon\label{eq: pf main - 2}
\end{align}
for $i\geq1.$

We set $N_{0}=\bigvee_{i=1}^{L+1}\left(N_{i}'\vee N_{i}''\right).$
By (\ref{eq: pf main - 1}) we have
\begin{align*}
\mathbb{P}_{N}(\mbox{a cluster intersecting} & \left.B\left(KN\right)\mbox{ freezes after time }p_{\lambda_{L+1}}\left(N\right)\right)\\
 & =\mathbb{P}_{N}\left(FC\left(p_{\lambda_{L+1}}\left(N\right),K,N\right)<FC\left(1,K,N\right)\right)\\
 & \leq\mathbb{P}_{N}\left(L<F\left(1,K+2,N\right)\right)+\mathbb{P}_{N}\left(\begin{array}{c}
FC\left(p_{\lambda_{L+1}}\left(N\right),K+2,N\right)\leq L\\
FC\left(p_{\lambda_{L+1}}\left(N\right),K,N\right)<FC\left(1,K,N\right)
\end{array}\right)\\
 & \leq\varepsilon/2+\mathbb{P}_{N}\left(\bigcup_{i=1}^{L+1}\left\{ \begin{array}{c}
FC\left(p_{\lambda_{i}}\left(N\right),K+2,N\right)=FC\left(p_{\lambda_{i+1}}\left(N\right),K+2,N\right)\\
FC\left(p_{\lambda_{i+1}}\left(N\right),K,N\right)<FC\left(1,K,N\right)
\end{array}\right\} \right)\\
 & \leq\varepsilon/2+\sum_{i=1}^{L+1}\mathbb{P}_{N}\left(\begin{array}{c}
FC\left(p_{\lambda_{i}}\left(N\right),K+2,N\right)=FC\left(p_{\lambda_{i+1}}\left(N\right),K+2,N\right)\\
FC\left(p_{\lambda_{i+1}}\left(N\right),K,N\right)<FC\left(1,K,N\right)
\end{array}\right)\\
 & \leq\varepsilon/2+\sum_{i=1}^{L+1}2^{-i-1}\varepsilon<\varepsilon
\end{align*}
for $N\geq N_{0}$ where we applied (\ref{eq: pf main - 2}) in the
last line. This finishes the proof of Theorem \ref{thm: main}.
\end{proof}

\subsection{\label{sub: pf main cor}Proof of Corollary \ref{cor: macroscopic clusters}}
\begin{proof}
[Proof of Corollary \ref{cor: macroscopic clusters} ] For $\lambda\in\mathbb{R}$
and $N\in\mathbb{N}$ let $NF\left(\lambda\right)=NF\left(\lambda,N\right)$
denote the event that no cluster intersecting $B\left(5N\right)$
freezes after time $p_{\lambda}\left(N\right).$ By Theorem \ref{thm: main}
there is $\lambda=\lambda\left(\varepsilon\right)$ and $N_{1}=N_{1}\left(\varepsilon\right)$
such that 
\begin{equation}
\mathbb{P}_{N}\left(NF\left(\lambda\right)\right)>1-\varepsilon/3\label{eq: pf main cor - 1}
\end{equation}
for $N\geq N_{1}.$

First we consider the case where the origin is in an open frozen cluster
at time $1,$ that is $\diam\left(C\left(1\right)\right)\geq N.$
Note that on the event $NF\left(\lambda\right),$ this frozen cluster
was formed before or at $p_{\lambda}\left(N\right).$ Hence on this
event there is a $p_{\lambda}\left(N\right)$-open path from the origin
to distance at least $N/2.$ Hence the event $\mathcal{A}_{1,0,o}^{\lambda,\lambda,N}\left(1,N/2\right)$
defined in Lemma \ref{lem: near critical arms} occurs. 

Let us turn to the case where $\diam\left(C\left(1\right)\right)<N.$
Recall the notation $\mathcal{C}_{a}\left(\lambda\right)$ from Definition
\ref{def: C_a}. It is easy to check that $C\left(1\right)=\mathcal{C}_{a}\left(\lambda\right)$
on the event$\left\{ \diam\left(C\left(1\right)\right)<N\right\} \cap NF\left(\lambda\right).$

If $\diam\left(\mathcal{C}_{a}\left(\lambda\right)\right)<aN,$ then
$\partial\mathcal{C}_{a}\left(\lambda\right)\cap B\left(2aN\right)\neq\emptyset$
for large $N.$ Since $v\in\partial\mathcal{C}_{a}\left(\lambda\right)\cap B\left(2aN\right)$
is frozen, it has a neighbour which has an open frozen path to distance
at least $N/2.$ On the event $NF\left(\lambda\right),$ this path
is $p_{\lambda}\left(N\right)$-open. Hence the event $\mathcal{A}_{1,0,o}^{\lambda,\lambda,N}\left(2aN,N/2\right)$
occurs. This combined with the argument above, for $a\in\left(0,1\right)$
and $N>N_{2}=1/a$ we have 
\[
\left\{ \diam\left(C\left(1\right)\right)\in\left[0,aN\right)\cup\left[N,\infty\right)\right\} \cap NF\left(\lambda\right)\subseteq\mathcal{A}_{1,0,o}^{\lambda,\lambda,N}\left(2aN,N/2\right).
\]

Hence by Lemma \ref{lem: near critical arms} there is $c=c\left(\lambda\right)$
and $N_{3}=N_{3}\left(\lambda\right)$ such that 
\begin{align*}
\mathbb{P}_{N}\left(\diam\left(C\left(1\right)\right)\in\left[0,aN\right)\cup\left[N,\infty\right),\, NF\left(\lambda\right)\right) & \leq\mathbb{P}\left(\mathcal{A}_{1,0,o}^{\lambda,\lambda,N}\left(2aN,N/2\right)\right)\\
 & \leq c\mathbb{P}_{1/2}\left(\mathcal{A}_{1,o}\left(2aN,N/2\right)\right)
\end{align*}
for $N\geq N_{3}.$ Theorem \ref{thm: arm exponents} gives that there
is $a=a\left(\varepsilon\right)$ and $N_{4}=N_{4}\left(\varepsilon\right)$
such that 
\begin{equation}
\mathbb{P}_{N}\left(\diam\left(C\left(1\right)\right)\in\left[0,aN\right)\cup\left[N,\infty\right),\, NF\left(\lambda\right)\right)\leq c\mathbb{P}_{1/2}\left(\mathcal{A}_{1,o}\left(2aN,N/2\right)\right)<\varepsilon/3.\label{eq: pf main cor - 3}
\end{equation}
for $N\geq N_{4}.$

Finally, Proposition \ref{prop: active diameter} gives $b=b\left(\varepsilon\right)$
and $N_{5}=N_{5}\left(\varepsilon\right)$ such that 
\begin{align}
\mathbb{P}_{N}\left(\diam\left(\mathcal{C}_{a}\left(\lambda\right)\right)\in\left[bN,N\right),\, NF\left(\lambda\right)\right) & \leq\mathbb{P}_{N}\left(\diam\left(\mathcal{C}_{a}\left(\lambda\right)\right)\in\left[bN,N\right)\right)\nonumber \\
 & \leq\varepsilon/3\label{eq: pf main cor - 4}
\end{align}
for $N\geq N_{5}.$ 

Since $C\left(1\right)=\mathcal{C}_{a}\left(\lambda\right)$ on the
event $\left\{ \diam\left(C\left(1\right)\right)<N\right\} \cap NF\left(\lambda\right),$
a combination of (\ref{eq: pf main cor - 1}), (\ref{eq: pf main cor - 3})
and (\ref{eq: pf main cor - 4}) finishes the proof of Corollary \ref{cor: macroscopic clusters}. 
\end{proof}

\section{\label{sec: pf prop big then freeze}Proof of Proposition \ref{prop: big then freeze}}

\subsection{\label{sub: notation}Notation}

Let us introduce some more notation. For $u=\left(u_{1},u_{2}\right),v=\left(v_{1},v_{2}\right)\in V,$
we say that $u$ is left (right) of $v$ if $u_{1}\leq v_{1}$ ($u_{1}\geq v_{1}$).
Similarly we say that $u$ is below (above) $v$ if $u_{2}\leq v_{2}$
($u_{2}\geq v_{2}$). For a finite set of vertices $W\subseteq V$
we say that $v=\left(v_{1},v_{2}\right)\in W$ is a leftmost (rightmost)
vertex of $W$ if for all $w=\left(w_{1},w_{2}\right)\in W,$ $v_{1}\leq w_{1}$
($v_{1}\geq w_{1}$). We define the lowest and highest vertices of
$W$ in an analogous way.

Recall that $v,w\in V$, $v\sim w$ denotes that $v$ and $w$ are
neighbours in $\mathbb{T}.$ We extend this notation for subsets of
$V:$ For $S,U\subset V,$ $S\sim U$ denotes that $\exists s\in S,\exists u\in U$
such that $s\sim u.$ Moreover, $S\nsim U$ denotes that $S\sim U$
does not hold.
\begin{defn}
\label{def: path}Let $n\in\mathbb{N}.$ We say that a sequence of
vertices $v^{1},v^{2},\ldots,v^{n},$ denoted by $\rho,$ is a path
if
\begin{itemize}
\item $v^{i}\sim v^{i+1}$ for $i=1,2,\ldots,\left(n-1\right),$ and 
\item $v^{i}\neq v^{j}$ when $i\neq j$ for $i,j=1,2,\ldots,n.$ 
\end{itemize}
We say that $\rho$ is\emph{ }non self touching, if $u,w\in\rho$
with $u\sim w$ then there is some $i\in\mathbb{N}$ with $1\leq i\leq n-1$
such that either $u=v^{i}$ and $w=v^{i+1}$ or $u=v^{i+1}$ and $w=v^{i}.$
We consider our paths to be ordered: $v^{1}$ is the starting point
and $v^{n}$ is the ending point of $\rho.$ For $u,w\in\rho$ we
say that $u$ is after $w$ in $\rho,$ and denote it by $w\prec_{\rho}u$
if $u=v^{i}$ and $w=v^{j}$ for some $i,j\in\mathbb{N}$ with $1\leq j<i\leq n.$
For $u,w\in\rho,$ $u\preceq_{\rho}w$ denotes that either $u=w$
or $u\prec_{\rho}w.$ When it is clear from the context which path
we are considering, we omit the subscript $\rho.$ For $u,w,z\in\rho$
we say that $w$ is in between $u$ and $z$ if $u\preceq w\preceq z$
or $u\succeq w\succeq z.$ For $u,z\in\rho$ with $u\preceq_{\rho}z$
let $\rho_{u,z}$ denote the subpath of $\rho$ consisting of the
vertices between $u$ and $z.$

We say that two paths $\rho_{1},\rho_{2}$ are non-touching, if $\rho_{1}\nsim\rho_{2}.$
\end{defn}

\begin{defn}
\label{def: loop}Let $n\in\mathbb{N}$ and sequence of vertices $v^{1},v^{2},\ldots,v^{n},$
satisfying
\begin{itemize}
\item $v^{i}\sim v^{i+1\mod n}$ for $i=1,2,\ldots,n,$ and 
\item $v^{i}\neq v^{j}$ when $i\neq j$ for $i,j=1,2,\ldots,n.$ 
\end{itemize}
A loop $\nu$ is the equivalence class of the sequence $\left(v^{1},v^{2},\ldots,v^{n}\right)$
under cyclic permutations, i.e $\nu$ is the set of sequences $\left(v^{j},v^{j+1\mod n},\ldots,v^{j+n-1\mod n}\right)$
for $j=1,2,\ldots,n.$ $\nu$ is non-self touching if for all $\left(w^{1},w^{2},\ldots,w^{n}\right)\in\nu,$
the path $\left(w^{1},w^{2},\ldots,w^{n-1}\right)$ is non-self touching.

With a slight abuse of notation, we say that a loop $\nu$ contains
a vertex $v$ and denote it by $v\in\nu$ if $v=v^{i}$ for some $i\in\left\{ 1,2,\ldots,n\right\} .$
Let $v,w\in\nu$ with $v\neq w$ and let $\rho$ denote the unique
path which starts at $v$ and represents $\nu.$ With the notation
of Definition \ref{def: path}, let $\nu_{v,w}:=\rho_{v,w}$ denote
the arc of $\nu$ starting at $v$ and ending at $w.$
\end{defn}

\subsection{\label{sub: thick paths}Thick paths}
\begin{defn}
\label{def: gridpath}Let $M\in\mathbb{N}$ be fixed. The $M$-grid
is the set of parallelograms $B\left(\left(2M+1\right)z;M\right)$
for $z\in V.$ Let $\pi$ be a sequence consisting of some parallelograms
of the $M$-grid. We say that $\pi$ is an $M$-gridpath, if for any
two consecutive parallelograms $B,B'$ of $\pi$ share a side, i.e
$\left|\partial B\cap B'\right|\geq2.$
\end{defn}

\begin{defn}
\label{def: (a,b)-nice}Let $C$ be a subgraph of $\mathbb{T},$ $D\subset V$
and $a,b\in\mathbb{N}.$ We say that $C$ is $\left(a,b\right)$-nice
in $D$, if it satisfies the conditions
\begin{enumerate}
\item $C$ is a connected induced subgraph of $\mathbb{T},$
\item $\partial C$ is a disjoint union of non-touching loops, each with
diameter bigger than $2b.$ 
\item \label{cond nice: 6 arms, then dead end}Let $u,v\in\partial C\cap D$
with $d\left(u,v\right)\leq a.$ Then $u,v$ are contained in the
same loop $\gamma$ of $\partial C,$ and $\diam\left(\gamma_{u,v}\right)\wedge\diam\left(\gamma_{v,u}\right)\leq b.$ 
\end{enumerate}
In the case where $D=V,$ we say that $C$ is $\left(a,b\right)$-nice.
\end{defn}
Let $C$ be $\left(a,b\right)$-nice for some $a,b\in\mathbb{N}.$
Condition \ref{cond nice: 6 arms, then dead end} of Definition \ref{def: (a,b)-nice},
roughly speaking, says that if there is a corridor in $C$ with width
less than $a,$ then it connects two parts of $C$ such that one part
has diameter at most $b.$ This suggests that when $b$ is small compared
to $\diam\left(C\right),$ then we can move a parallelogram with side
length $O\left(a\right)$ in $C$ between two distant points of $C.$
This intuitive argument leads us to the following lemma. 
\begin{lem}
\label{lem: det gridpath}Let $a,b\in\mathbb{N}$ with $a\geq2000.$
Let $C$ be an $\left(a,b\right)$-nice subgraph of $\mathbb{T}.$
Then there is a $\left\lfloor a/200-10\right\rfloor $-gridpath contained
in $C$ with diameter at least $\diam\left(C\right)-2b-2a-12.$
\end{lem}
We use the following `local' version of Lemma \ref{lem: det gridpath}:
\begin{lem}
\label{lem: local det gridpath}Let $a,b,c\in\mathbb{N}$ with $a\geq2000.$
Let $C$ be subgraph of $\mathbb{T}$ which is $\left(a,b\right)$-nice
in $B\left(c\right).$ Let $C'$ be a connected component of $C\cap B\left(c\right).$
Then there is a $\left\lfloor a/200-10\right\rfloor $-gridpath contained
in $C'$ with diameter at least $\diam\left(C'\right)-2b-2a-12.$\end{lem}
\begin{proof}
[Proof of Lemma \ref{lem: det gridpath} and \ref{lem: local det gridpath}]
The proof of Lemma \ref{lem: det gridpath} and \ref{lem: local det gridpath}
have geometric/topologic nature, hence it is moved to Section \ref{sub: there are thick paths}
of the Appendix.
\end{proof}
We recall and prove Proposition \ref{prop: big then freeze} in the
following.
\begin{namedthm}
[Proposition \ref{prop: big then freeze}] Let $\theta\in\left(0,1\right),$
$\varepsilon>0$ and $\lambda_{1}K,\in\mathbb{R}.$ Recall the notation
$FC\left(t,K+2,N\right)$ from Lemma \ref{lem: few frozen clusters}.
There exists $\lambda_{2}=\lambda_{2}\left(\lambda_{1},\theta,\varepsilon\right)$
and $N_{0}=N_{0}\left(\lambda_{1},\theta,\varepsilon\right)$ such
that the probability of the intersection of the events
\begin{itemize}
\item $\exists v\in B\left(KN\right)$ such that $\diam\left(\mathcal{C}_{a}\left(v;\lambda_{1},N\right)\right)\geq\left(1+\theta\right)N,$
and 
\item none of the clusters intersecting $B\left(\left(K+2\right)N\right)$
freeze in the time interval $\left(p_{\lambda_{1}}\left(N\right),p_{\lambda_{2}}\left(N\right)\right],$
i.e. 
\[
FC\left(p_{\lambda_{1}}\left(N\right),K+2,N\right)=FC\left(p_{\lambda_{2}}\left(N\right),K+2,N\right)
\]

\end{itemize}
\end{namedthm}
is less than $\varepsilon$ for $N\geq N_{0}.$
\begin{proof}
[Proof of Proposition \ref{prop: big then freeze}] By Lemma \ref{lem: there is a net}
we choose $\lambda_{0}=\lambda_{0}\left(\varepsilon,K\right)\leq\lambda_{1}$
and $N_{1}=N_{1}\left(\varepsilon,K\right)$ such that 
\begin{equation}
\mathbb{P}\left(\mathcal{N}_{c}\left(\lambda_{0},1/6,K+6,N\right)\right)>1-\varepsilon/3.\label{eq: thich path - 1}
\end{equation}
By Corollary \ref{cor: no many arms} we choose $\eta<\theta/10$
and $N_{2}=N_{2}\left(\eta,\theta,\lambda_{0},\lambda_{1},K\right)$
such that 
\begin{equation}
\mathbb{P}\left(\mathcal{NA}\left(2\eta,\theta/10,\lambda_{0},\lambda_{1},K+4,N\right)\right)>1-\varepsilon/3\label{eq: thich path - 2}
\end{equation}
for all $N\geq N_{2}.$ Let 
\[
E:=\mathcal{N}_{c}\left(\lambda_{0},1/6,K+6,N\right)\cap\mathcal{NA}\left(2\eta,\theta/10,\lambda_{0},\lambda_{1},K+4,N\right).
\]

\begin{claim}
\label{calim: pf big then freeze}Let $u\in B\left(KN\right)$ with
$\diam\left(\mathcal{C}_{a}\left(u;\lambda_{1},N\right)\right)\geq\left(1+\theta\right)N.$
Then $\mathcal{C}_{a}\left(u;\lambda_{1},N\right)$ is $\left(\eta N,\frac{\theta}{10}N\right)$-nice
in $B\left(u;2N\right)$ on the event $E.$\end{claim}
\begin{proof}
[Proof of Claim \ref{calim: pf big then freeze}]Let us check the
conditions of Definition \ref{def: (a,b)-nice}. The Condition 1 is
satisfied by the definition of $\mathcal{C}_{a}\left(u;\lambda_{1},N\right).$

All the holes of $\mathcal{C}_{a}\left(u;\lambda_{1},N\right)$ contain
a frozen cluster, which have diameter at least $N.$ This combined
with $2\frac{\theta}{10}N<N,$ shows that Condition 2 of Definition
\ref{def: (a,b)-nice} holds. 

Let $x,y\in\partial\mathcal{C}_{a}\left(u;\lambda_{1},K\right)\cap B\left(u;2N\right)$
with $d\left(x,y\right)\leq\eta N.$ We have two cases.

\emph{Case 1.} $x,y$ lie in different loops of $\partial\mathcal{C}_{a}\left(u;\lambda_{1},N\right).$
For $i=x,y,$ let $\gamma_{i}$ denote the loop containing $i.$ Furthermore,
let $\tilde{\gamma}_{i}$ denote the connected component of $i$ in
$\gamma_{i}\cap B\left(i;2N\right).$ We have $\diam\left(\tilde{\gamma}_{i}\right)\geq N.$
Moreover, $\tilde{\gamma}_{i}\subset B\left(i;2N\right)\subset B\left(\left(K+4\right)N\right).$
Observation \ref{obs: closed grid-> no freezing} gives that on the
event $\mathcal{N}_{c}\left(\lambda_{0},1/6,K+6,N\right),$ $\tilde{\gamma}_{i}$
is $p_{\lambda_{0}}\left(N\right)$-closed. Hence each of $\tilde{\gamma}_{x}$
and $\tilde{\gamma}_{y}$ gives two closed $p_{\lambda_{0}}\left(N\right)$-closed
arms in $A\left(x;2\eta N,N/2\right).$ Moreover, the frozen clusters
neighbouring $x$ and $y$ provide two disjoint $p_{\lambda_{1}}\left(N\right)$-open
arms. Hence there are $6$ disjoint arms in $A\left(x;2\eta N,N/2\right),$
thus $\mathcal{NA}^{c}\left(2\eta,\theta/10,\lambda_{0},\lambda_{1},K+4,N\right)$
occurs. 

\emph{Case 2. }$x,y$ lie on the same loop of $\partial\mathcal{C}_{a}\left(u;\lambda_{1},N\right).$
This case can be treated similarly to Case 1, with the difference
that if $x,y$ violate Condition \ref{cond nice: 6 arms, then dead end}
of Definition \ref{def: (a,b)-nice} then we get $6$ arms in $A\left(x;2\eta N,\frac{\theta}{10}N\right).$
Hence $\mathcal{NA}^{c}\left(2\eta,\theta/10,\lambda_{0},\lambda_{1},K+4,N\right)$
occurs.

Hence in both cases $E^{c}$ occurs. Thus on the event $E$ all the
conditions of Definition \ref{def: (a,b)-nice} are satisfied for
$\mathcal{C}_{a}\left(u;\lambda_{1},N\right),$ which finishes the
proof of Claim \ref{calim: pf big then freeze}.
\end{proof}
Let us turn back to the proof of Proposition \ref{prop: big then freeze}.
Let $u\in B\left(KN\right)$ with $\diam\left(\mathcal{C}_{a}\left(u;\lambda_{1},N\right)\right)\geq\left(1+\theta\right)N.$
Let $\tilde{\mathcal{C}}_{a}\left(u,\lambda_{1},N\right)$ denote
the connected component of $u$ in $\mathcal{C}_{a}\left(u,\lambda_{1},N\right)\cap B\left(u;2N\right).$
Since $\diam\left(\mathcal{C}_{a}\left(u;\lambda_{1},N\right)\right)\geq\left(1+\theta\right)N$
and $\theta<1,$ we have $\diam\left(\tilde{\mathcal{C}}_{a}\left(u;\lambda_{1},N\right)\right)\geq\left(1+\theta\right)N.$
By Lemma \ref{lem: local det gridpath} we set $\eta=\eta\left(\theta\right)\in\left(0,\theta/100\right)$
and $N_{3}=N_{3}\left(\theta\right)$ such that on the event $E$
for all $u\in B\left(KN\right),$ with $\diam\left(\mathcal{C}_{a}\left(u;\lambda_{1},N\right)\right)\geq\left(1+\theta\right)N$
there is a $\left\lfloor \eta N\right\rfloor $-gridpath $\rho_{u}\subset\tilde{\mathcal{C}}_{a}\left(u;\lambda_{1},N\right)$
with $\diam\left(\rho_{u}\right)\geq\left(1+\theta/2\right)N$ for
$N\geq N_{3}.$ 

Lemma \ref{lem: there is a net} gives that there is $\lambda_{2}=\lambda_{2}\left(\varepsilon,\eta,K\right)$
and $N_{4}=N_{4}\left(\varepsilon,\eta,K\right)$ such that 
\begin{equation}
\mathbb{P}\left(\mathcal{N}_{o}\left(\lambda_{2},\eta/2,K+4,N\right)\right)>1-\varepsilon/3\label{eq: thich path - 3}
\end{equation}
for $N\geq N_{4}\left(\varepsilon,\eta,K\right).$ We set $N_{0}:=\bigvee_{i=1}^{4}N_{i}.$
Let 
\begin{align*}
G:= & E\cap\mathcal{N}_{o}\left(\lambda_{2},\eta/2,K+4,N\right),\\
M:= & \left\{ \exists v\in B\left(KN\right)\mbox{ s.t. }\diam\left(\mathcal{C}_{a}\left(v;\lambda_{1},N\right)\right)\geq\left(1+\theta\right)N\right\} \cap G.
\end{align*}
Combination of (\ref{eq: thich path - 1}), (\ref{eq: thich path - 2})
and (\ref{eq: thich path - 3}) gives that 
\begin{equation}
\mathbb{P}\left(G^{c}\right)<\varepsilon\label{eq: thich path - 4}
\end{equation}
for $N\geq N_{0}.$

Recall that for $N\geq N_{0},$ on the event $E$ for $u\in B\left(KN\right),$
with $\diam\left(\mathcal{C}_{a}\left(u;\lambda_{1},N\right)\right)\geq\left(1+\theta\right)N$
there is a $\left\lfloor \eta N\right\rfloor $-gridpath $\rho_{u}\subset\tilde{\mathcal{C}}_{a}\left(u;\lambda_{1},N\right)$
with $\diam\left(\rho_{u}\right)\geq\left(1+\theta/2\right)N.$ On
the event $\mathcal{N}_{o}\left(\lambda_{2},\eta/2,K+4,N\right),$
this gridpath $\rho_{u}\subseteq B\left(\left(K+2\right)N\right)$
contains a $p_{\lambda_{2}}\left(N\right)$-open component with diameter
at least $N.$ Hence on the event $M,$ at least one cluster intersecting
$B\left(\left(K+2\right)N\right)$ freezes in the time interval $\left(p_{\lambda_{1}}\left(N\right),p_{\lambda_{2}}\left(N\right)\right].$
That is
\[
M\subseteq\left\{ FC\left(p_{\lambda_{1}}\left(N\right),K+2,N\right)<FC\left(p_{\lambda_{2}}\left(N\right),K+2,N\right)\right\} .
\]
Thus 
\[
\left\{ \exists v\in B\left(KN\right)\mbox{ s.t. }\diam\left(\mathcal{C}_{a}\left(v;\lambda_{1},N\right)\right)\geq\left(1+\theta\right)N\right\} \cap\left\{ FC\left(p_{\lambda_{1}}\left(N\right),K+2,N\right)=FC\left(p_{\lambda_{2}}\left(N\right),K+2,N\right)\right\} \subset G^{c},
\]
which together with (\ref{eq: thich path - 4}) finishes the proof
of Proposition \ref{prop: big then freeze}.
\end{proof}

\section{\label{sec: pf of prop active diameter}Proof of Proposition \ref{prop: active diameter}}

\subsection{\label{sub: lowest of lowest in parallelograms}Lowest point of the
lowest crossing in parallelograms}

Recall the notation of Section \ref{sub: notation}.
\begin{defn}
\label{def: cal L}Let $R$ be a connected subgraph of $\mathbb{T}$
and let $r\subset\partial R$. We define $\mathcal{L}\left(R,r\right)$
as the (random) set of lowest vertices $v\in R$ such that $v$ is
closed, and there are two non-touching closed paths in $R$ starting
at a vertex neighbouring to $v$ and ending at $r.$
\end{defn}
Consider the site percolation model on the triangular lattice with
parameter $p\in\left[0,1\right].$ We investigate the distribution
of $\mathcal{L}\left(R,r\right)$ in the case where $p=p_{\lambda}\left(N\right),$
$R=B\left(bN\right)$ and $r=top\left(B\left(bN\right)\right):=\left[-bN,bN\right]\boxtimes\left\{ \left\lfloor bN\right\rfloor +1\right\} $
for $\lambda\in\mathbb{R}$ and $b>0.$ 
\begin{defn}
\label{def: HCr}For a parallelogram $B,$ let $HCr\left(B\right)$
denote set of paths in $B$ which connect the left and the right sides
of $B.$ For $\rho\in HCr\left(B\right),$ let $Be\left(\rho\right)=Be\left(\rho,B\right)$
denote the set of vertices in $B$ which are `under' $\rho.$ It is
the set of vertices $v\in B\setminus\rho$ which are connected to
the bottom side of $B.$ Furthermore, we define $Ab\left(\rho\right)=Ab\left(\rho,B\right):=B\setminus\left(\rho\cup Be\left(\rho,B\right)\right).$\end{defn}
\begin{lem}
\label{lem: lowest of lowest}Let $a,b\in\left(0,1\right)$ with $5a<b.$
For $k,l,N\in\mathbb{N}$ with $l<k$ we define the parallelogram
\begin{equation}
B_{l,k}:=\left[-aN,aN\right]\boxtimes\left(\left(2\frac{l}{k}-1\right)aN,\left(2\frac{l+1}{k}-1\right)aN\right]\label{eq: def B_l,k}
\end{equation}
and the event 
\begin{equation}
L_{l,k}=:\left\{ \mathcal{L}\left(B\left(bN\right),\, top\left(bN\right)\right)\cap B_{l,k}\neq\emptyset\right\} .\label{eq: def L_l,k}
\end{equation}
That is, $L_{l,k}$ is the event that at least one of the lowest vertices
of $B\left(bN\right)$ with two non-touching closed paths $B\left(bN\right)$
to the top side of $B\left(bN\right)$ is in the parallelogram $B_{l,k}.$ 

Let $\lambda_{1},\lambda_{2}\in\mathbb{R}.$ Then there exist $C=C\left(a,b,\lambda_{1},\lambda_{2}\right)$
and $N_{0}=N_{0}\left(a,b,\lambda_{1},\lambda_{2},k\right)$ such
that for all $\lambda\in\left[\lambda_{1},\lambda_{2}\right]$ and
$k,l\in\mathbb{N}$ with $l\leq k-1$ we have 
\begin{equation}
\mathbb{P}_{p_{\lambda}\left(N\right)}\left(L_{l,k}\right)\leq Ck^{-1}\label{eq: lowest of lowest}
\end{equation}
for $N\geq N_{0}.$ In particular, the upper bound in (\ref{eq: lowest of lowest})
is uniform in $l.$\end{lem}
\begin{proof}
[Proof of Lemma \ref{lem: lowest of lowest}] For $k\leq5$ the
statement is trivial, hence we assume that $k\geq5$ in the following.
We extend the notation in (\ref{eq: def B_l,k}) and (\ref{eq: def L_l,k})
for $l\in\left\{ -k,-k+1,\ldots,-1\right\} .$ 

First we show that there exist $c=c\left(a,b,\lambda_{1},\lambda_{2}\right)>0$
and $N_{0}=N_{0}\left(a,b,\lambda_{1},\lambda_{2}\right)$ such that
for all $l,m\in\left[-k,k-1\right]\cap\mathbb{Z}$ with $m+1\leq l$
we have
\begin{equation}
c\mathbb{P}_{p_{\lambda}\left(N\right)}\left(L_{l,k}\right)\leq\mathbb{P}_{p_{\lambda}\left(N\right)}\left(L_{m,k}\cup L_{m+1,k}\right)\label{eq: pf lowest of lowest - 0}
\end{equation}
for $N\geq N_{0}.$ Let $S=S\left(l,m,k\right):V\rightarrow V$ denote
a shift which moves the parallelogram $B_{l,k}$ to a subset of $B_{m,k}\cup B_{m+1,k}.$
The shift $S$ naturally induces a map on the configurations $\omega\in\left\{ o,c\right\} ^{V}$
by $S\left(\omega\right)\left(v\right)=\omega\left(S^{-1}\left(v\right)\right).$
Roughly speaking, we prove (\ref{eq: pf lowest of lowest - 0}) by
showing that positive proportion of the configurations $\omega\in L_{l,k}$
satisfy $S\left(\omega\right)\in L_{m,k}\cup L_{m+1,k}.$ We achieve
this by showing that, conditioning on $L_{l,k},$ all the crossing
events of Figure \ref{fig: lowest of lowest} occur with probability
bounded away from $0.$ Let us turn to the precise proof.

Let $k,l$ be given. Let $s_{L}$ ($s_{R}$) denote the left (right)
endpoint of $top\left(bN\right).$ We say that a path $\rho\subseteq B\left(bN\right)\cup top\left(bN\right)$
is good, if it
\begin{itemize}
\item starts at $s_{L}$ and ends at $s_{R},$
\item it is non-self touching
\item and one of its lowest points is in $B_{l,k}.$
\end{itemize}
Let $\rho$ be some given good path. Recall Definition \ref{def: HCr}
and let $Be\left(\rho\right)=Be\left(\rho,\left(B\left(bN\right)\right)\right).$
Let $H_{\rho}$ denote the event that there are two open paths in
$Be\left(\rho\right)\cap\left[-bN,bN\right]\boxtimes\left[aN,\frac{b-2a}{2}N\right]$
from the left and right sides of the parallelogram $\left[-bN,bN\right]\boxtimes\left[aN,\frac{b-2a}{2}N\right]$
to $\rho.$ Let $\gamma$ denote the lowest non-self touching path
in $B\left(bN\right)\cup top\left(bN\right)$ which starts at $s_{L}$
and ends at $s_{R},$ and of which all the vertices outside of $top\left(bN\right)$
are closed. On the event $L_{l,k}$ $\gamma$ is good.

Let $\rho$ be a fixed good path. Let $O_{\rho}$ denote the event
that there is path $\nu$ such that 
\begin{itemize}
\item $\nu\subseteq B_{0}:=\left[-bN,bN\right]\boxtimes\left[-bN,\frac{b}{4}N\right]$,
\item $\nu$ connects the left and the right sides of the parallelogram
$B_{1}:=\left[-bN,bN\right]\boxtimes\left[aN,\frac{b}{4}N\right],$
\item $\nu$ is a concatenation of some open paths which lie in $Be\left(\rho\right)\cap B_{1},$
and of some subpaths of $\rho.$
\end{itemize}
Clearly, $O_{\rho}$ is an increasing event. On $O_{\rho},$ let $\xi\left(\rho\right)$
denote the lowest path which satisfies the conditions in the definition
of $O_{\rho}.$ Recall the definition of decreasing events from Definition
\ref{def: inc events}, and the definition of $\gamma$ from Case
1. Let us condition on the event that all the vertices of $\rho\setminus top\left(bN\right)$
are closed. Then the event $\left\{ \gamma=\rho\right\} $ is increasing
on the configuration in $B\left(bN\right)\setminus\rho,$ and it only
depends on the configuration in $Be\left(\rho\right).$ Hence a combination
of FKG and Corollary \ref{cor: posprob of crossing} give that 
\begin{align}
\mathbb{P}_{p_{\lambda}\left(N\right)} & \left(L_{l,k}\cap O_{\gamma}\right)\nonumber \\
 & =\sum_{\rho\mbox{ good}}\mathbb{P}_{p_{\lambda}\left(N\right)}\left(\left.O_{\rho}\cap\left\{ \gamma=\rho\right\} \,\right|\,\rho\setminus top\left(bN\right)\mbox{ is closed}\right)\mathbb{P}_{p_{\lambda}\left(N\right)}\left(\rho\setminus top\left(bN\right)\mbox{ is closed}\right)\nonumber \\
 & \!\begin{multlined}[t][15.5cm]
	\geq\sum_{\rho\mbox{ good}}\mathbb{P}_{p_{\lambda}\left(N\right)}\left(\left.\left\{ \gamma=\rho\right\} \,\right|\,\rho\setminus top\left(bN\right)\mbox{ is closed}\right) \\[-0.5cm]
	\mathbb{P}_{p_{\lambda}\left(N\right)}\left(\left.O_{\rho}\,\right|\,\rho\setminus top\left(bN\right)\mbox{ is closed}\right)\mathbb{P}_{p_{\lambda}\left(N\right)}\left(\rho\setminus top\left(bN\right)\mbox{ is closed}\right)
     \end{multlined} \nonumber \\
 & \geq\sum_{\rho\mbox{ good}}\mathbb{P}_{p_{\lambda}\left(N\right)}\left(\left.\left\{ \gamma=\rho\right\} \,\right|\,\rho\setminus top\left(bN\right)\mbox{ is closed}\right)\mathbb{P}_{p_{\lambda}\left(N\right)}\left(\mathcal{H}_{o}\left(B_{1}\right)\right)\mathbb{P}_{p_{\lambda}\left(N\right)}\left(\rho\setminus top\left(bN\right)\mbox{ is closed}\right)\nonumber \\
 & \geq c_{1}\left(\lambda_{1},\lambda_{2},a,b\right)\mathbb{P}_{p_{\lambda}\left(N\right)}\left(L_{l,k}\right)\label{eq: pf lowest of lowest - 3}
\end{align}
for $c_{1}=c_{1}\left(a,b,\lambda_{1},\lambda_{2}\right)>0$ and for
$N\geq N_{1}=N_{1}\left(a,b,\lambda_{1},\lambda_{2}\right).$

For $W\subseteq V$ and $\omega\in\left\{ o,c\right\} ^{V},$ $\omega_{W}\in\left\{ o,c\right\} ^{W}$
denotes the restriction of $\omega$ to the configuration in $W.$
That is $\omega_{W}\left(v\right)=\omega\left(v\right)$ for $v\in W.$
Recall Definition \ref{def: HCr}. Let $\zeta\in HCr\left(B_{0}\right)$
be arbitrary. It is easy to check that the event $L_{l,k}\cap O_{\gamma}\cap\left\{ \xi\left(\gamma\right)=\zeta\right\} $
is decreasing in the configuration in $Ab\left(\zeta\right).$ Let
us take the parallelograms $B_{2}=\left[-bN,bN\right]\boxtimes\left[\frac{b}{4}N,\frac{b}{2}N\right],$
$B_{3}=\left[-bN,bN\right]\boxtimes\left[\frac{3}{4}bN,bN\right],$
$B_{4}=\left[-bN,-\frac{1}{2}bN\right]\boxtimes\left[\frac{1}{4}bN,\left(b+4a\right)N\right]$
and $B_{5}=\left[\frac{1}{2}bN,bN\right]\boxtimes\left[\frac{1}{4}bN,\left(b+4a\right)N\right].$
Let $\mathcal{D}=\mathcal{H}_{c}\left(B_{2}\right)\cap\mathcal{H}_{c}\left(B_{3}\right)\cap\mathcal{V}_{c}\left(B_{4}\right)\cap\mathcal{V}_{c}\left(B_{5}\right).$
Clearly, $\mathcal{D}$ is a decreasing event. Hence a combination
of FKG and Corollary \ref{cor: posprob of crossing} give that for
$c_{2}=c_{2}\left(a,b,\lambda_{1},\lambda_{2}\right)>0$ and $N\geq N_{2}=N_{2}\left(a,b,\lambda_{1},\lambda_{2}\right)$
we have
\begin{align}
\mathbb{P}_{p_{\lambda}\left(N\right)} & \left(L_{k,l}\cap O_{\gamma}\cap\mathcal{D}\right)\nonumber \\
 & =\sum_{\zeta}\sum_{\sigma}\mathbb{P}_{p_{\lambda}\left(N\right)}\left(\left.L_{k,l}\cap O_{\gamma}\cap\left\{ \xi\left(\gamma\right)=\zeta\right\} \cap\mathcal{D}\,\right|\,\omega_{\zeta\cup Be\left(\zeta\right)}=\sigma\right)\mathbb{P}_{p_{\lambda}\left(N\right)}\left(\omega_{\zeta\cup Be\left(\zeta\right)}=\sigma\right)\nonumber \\
 & \geq\sum_{\zeta}\sum_{\sigma}\mathbb{P}_{p_{\lambda}\left(N\right)}\left(\left.L_{k,l}\cap O_{\gamma}\cap\left\{ \xi\left(\gamma\right)=\zeta\right\} \,\right|\,\omega_{\zeta\cup Be\left(\zeta\right)}=\sigma\right)\mathbb{P}_{p_{\lambda}\left(N\right)}\left(\left.\mathcal{D}\,\right|\,\omega_{\zeta\cup Be\left(\zeta\right)}=\sigma\right)\mathbb{P}_{p_{\lambda}\left(N\right)}\left(\omega_{\zeta\cup Be\left(\zeta\right)}=\sigma\right)\nonumber \\
 & =\sum_{\zeta}\sum_{\sigma}\mathbb{P}_{p_{\lambda}\left(N\right)}\left(\left.L_{k,l}\cap O_{\gamma}\cap\left\{ \xi\left(\gamma\right)=\zeta\right\} \,\right|\,\omega_{\zeta\cup Be\left(\zeta\right)}=\sigma\right)\mathbb{P}_{p_{\lambda}\left(N\right)}\left(\mathcal{D}\right)\mathbb{P}_{p_{\lambda}\left(N\right)}\left(\omega_{\zeta\cup Be\left(\zeta\right)}=\sigma\right)\nonumber \\
 & \geq c_{2}\left(a,b,\lambda_{1},\lambda_{2}\right)\sum_{\zeta}\sum_{\sigma}\mathbb{P}_{p_{\lambda}\left(N\right)}\left(\left.L_{k,l}\cap O_{\gamma}\cap\left\{ \xi\left(\gamma\right)=\zeta\right\} \,\right|\,\omega_{\zeta\cup Be\left(\zeta\right)}=\sigma\right)\mathbb{P}_{p_{\lambda}\left(N\right)}\left(\omega_{\zeta\cup Be\left(\zeta\right)}=\sigma\right)\nonumber \\
 & =c_{2}\left(a,b,\lambda_{1},\lambda_{2}\right)\mathbb{P}_{p_{\lambda}\left(N\right)}\left(L_{k,l}\cap O_{\gamma}\right)\label{eq: pf lowest of lowest - 4}
\end{align}
where the summation in $\zeta$ is over $HCr\left(B_{0}\right)$ and
the summation in $\sigma$ is over $\left\{ o,c\right\} ^{\zeta\cup Be\left(\zeta\right)}.$
In the third line we used that $\mathcal{D}$ does not depend on the
configuration in $\zeta\cup Be\left(\zeta\right).$

\begin{figure}
\centering{}\includegraphics[scale=0.8]{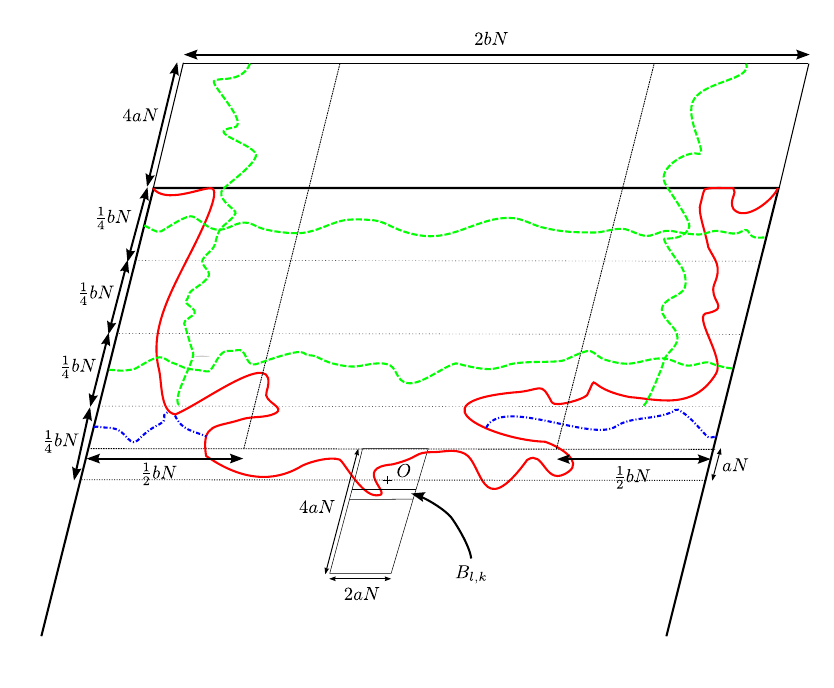}\caption{\label{fig: lowest of lowest}The continuous line represents $\gamma.$
The dashed paths are the closed crossings of $\mathcal{D},$ which
allow us to prolong $\gamma.$ The dashed-dotted paths are the open
parts of $\xi\left(\gamma\right).$ They, together with $\gamma,$
prevent the occurrence of closed vertices below the lowest point of
$\gamma$ with two closed arms to the top side of $B\left(bN\right)$
after the shift.}
\end{figure}

There is $N_{3}=N_{3}\left(k\right)$ such that for $N\geq N_{3}$
and for all $l,m\in\left[0,k-1\right]\cap\mathbb{Z}$ with $l>m$
there is a shift $S=S\left(l,m,k\right)$ which moves the parallelogram
$B_{l,k}$ to a subset of $B_{m,k}\cup B_{m+1,k}.$ Let us take a
configuration $\omega\in\left\{ o,c\right\} ^{V}$ which satisfies
$L_{k,l}\cap O_{\gamma}\cap\mathcal{D}.$ Then the shifted configuration
$S\left(\omega\right)$ satisfies $L_{m,k}\cup L_{m+1,k}.$ See Figure
\ref{fig: lowest of lowest} for more details. Hence for $N\geq N_{1}\vee N_{2}\vee N_{3}$
we have 
\begin{align}
\mathbb{P}_{p_{\lambda}\left(N\right)}\left(L_{m,k}\cup L_{m+1,k}\right) & \geq\mathbb{P}_{p_{\lambda}\left(N\right)}\left(L_{k,l}\cap O_{\gamma}\cap\mathcal{D}\right)\nonumber \\
 & \geq c_{1}c_{2}\mathbb{P}_{p_{\lambda}\left(N\right)}\left(L_{l,k}\right)\label{eq: pf lowest of lowest - 5}
\end{align}
by a combination of (\ref{eq: pf lowest of lowest - 3}) and (\ref{eq: pf lowest of lowest - 4}).
This finishes the proof of (\ref{eq: pf lowest of lowest - 0}).

Now we conclude the proof of Lemma \ref{lem: lowest of lowest}. By
summing over $m\in\left\{ -k,-k+1,\ldots,-2\right\} $ in (\ref{eq: pf lowest of lowest - 5})
we get that
\begin{align*}
\mathbb{P}_{p_{\lambda}\left(N\right)}\left(L_{k,l}\right) & \leq\left(k-1\right)^{-1}c_{1}c_{2}\sum_{m=-k}^{-2}\mathbb{P}_{p_{\lambda}\left(N\right)}\left(L_{m,k}\cup L_{m+1,k}\right)\\
 & \leq2c_{1}c_{2}k^{-1}\sum_{m=-k}^{-1}\mathbb{P}_{p_{\lambda}\left(N\right)}\left(L_{m,k}\right)\\
 & \leq Ck^{-1}
\end{align*}
for some $C=C\left(a,b,\lambda_{1},\lambda_{2}\right).$ In the last
line we used that $L_{m,k}\cap L_{m',k}=\emptyset$ for $m\neq m'.$
This finishes the proof of Lemma \ref{lem: lowest of lowest}.\end{proof}
\begin{rem}
Let $a,b,\lambda,\lambda_{1},\lambda_{2}$ be as in Lemma \ref{lem: lowest of lowest}.
Standard RSW techniques give that there is $c'=c'\left(a,b,\lambda_{1},\lambda_{2}\right)>0$
and $N_{0}=N_{0}\left(a,b,\lambda_{1},\lambda_{2}\right)$ such that
\[
\mathbb{P}_{p_{\lambda}\left(N\right)}\left(\mathcal{L}\left(B\left(bN\right),t\left(bN\right)\right)\cap B\left(aN\right)\neq\emptyset\right)\geq c'
\]
for $N\geq N_{0}.$ This, combined with arguments similar to the proof
of Lemma \ref{lem: lowest of lowest}, gives that there is $C'=C'\left(a,b,\lambda_{1},\lambda_{2}\right)>0$
and $N_{1}=N_{1}\left(a,b,\lambda_{1},\lambda_{2},k\right)$ such
that
\[
\mathbb{P}_{p_{\lambda}\left(N\right)}\left(L_{l,k}\right)\geq C'k^{-1}
\]
for $N\geq N_{1}$ uniformly for $l\leq k.$
\end{rem}

\subsection{\label{sub: lowesr of lowest in reular regions}Lowest point of the
lowest crossing in regular regions}

Recall Definition \ref{def: cal L}. Let $B\subset B'$ be parallelograms,
and let $R$ be a subgraph of $\mathbb{T}$ with $B\subset R\subset B'.$
Furthermore let $r\subset\partial R.$ Our next aim is to compare
the event $\mathcal{L}\left(R,r\right)\cap B\neq\emptyset$ to $\mathcal{L}\left(B',top\left(B'\right)\right)\cap B\neq\emptyset$
in the case where the pair $\left(R,r\right)$ is `regular'. We make
this precise in the following.

We say that a subgraph $H\subseteq\mathbb{T}$ is simply connected,
if it is connected and for all loops $\sigma\subseteq H,$ all of
the finite components of $\mathbb{T}\setminus\sigma$ are contained
in $H.$ 
\begin{defn}
\label{def: (a,b)-regular}Let $a,b\in\mathbb{N}$ such that $5a<b.$
A pair $\left(R,r\right)$ is $\left(a,b\right)$-regular, if
\begin{enumerate}
\item $R$ is a connected induced subgraph of $\mathbb{T},$
\item $B\left(a\right)\subseteq R\subseteq B\left(b\right),$
\item $r\subset\partial R,$ such that $\emptyset\neq r\subsetneqq\partial R.$
Furthermore, $r$ and $\partial R\setminus r$ are self-avoiding paths
such that $R$ is on the right hand side, as we walk along them.
\item \label{cond: r is high}$r\subseteq\left[-b,b\right]\boxtimes\left[5a,b\right].$
\end{enumerate}
\end{defn}
\begin{lem}
\label{lem: compare L}Let $a,b\in\left(0,1\right)$ with $5a<b$
and $\lambda\in\mathbb{R}.$ Let $\left(R,r\right)$ be $\left(aN,bN\right)$-regular.
For $k,l,N\in\mathbb{N}$ with $l<k$ we define the events 
\begin{align}
L_{l,k}\left(B\left(2bN\right),top\left(B\left(2bN\right)\right)\right): & =\left\{ \mathcal{L}\left(B\left(2bN\right),\, top\left(2bN\right)\right)\cap B_{l,k}\neq\emptyset\right\} ,\nonumber \\
L_{l,k}\left(R,r\right): & =\left\{ \mathcal{L}\left(R,r\right)\cap B_{l,k}\neq\emptyset\right\} ,\label{eq: def L_l,k(R,r)}
\end{align}
where 
\begin{equation}
B_{l,k}:=\left[-aN,aN\right]\boxtimes\left(\left(2\frac{l}{k}-1\right)aN,\left(2\frac{l+1}{k}-1\right)aN\right].\label{eq: def B_l,k - 1}
\end{equation}
Let $\lambda_{1},\lambda_{2}\in\mathbb{R}.$ Then there exist $C=C\left(a,b,\lambda_{1},\lambda_{2}\right)$
and $N_{0}=N_{0}\left(a,b,\lambda_{1},\lambda_{2},k\right)$ such
that for all $\lambda\in\left[\lambda_{1},\lambda_{2}\right]$ and
$k,l\in\mathbb{N}$ with $l\leq k-1$ we have 
\begin{equation}
\mathbb{P}_{p_{\lambda}\left(N\right)}\left(L_{l,k}\left(R,r\right)\right)\leq C\mathbb{P}_{p_{\lambda}\left(N\right)}\left(L_{l,k}\left(B\left(2bN\right),top\left(B\left(2bN\right)\right)\right)\right)\label{eq: compare L}
\end{equation}
for $N\geq N_{0}.$\end{lem}
\begin{proof}
[Proof of Lemma \ref{lem: compare L}] The proof follows the arguments
of the proof of Lemma \ref{lem: lowest of lowest}. Our aim is to
show that, conditioning on $L_{l,k}\left(R,r\right),$ the open and
closed crossings of Figure \ref{fig: lowest of lowest regular region}
occur with probability bounded away from $0$ cf. Figure \ref{fig: lowest of lowest}. 

\begin{figure}
\centering{}\includegraphics[scale=0.8]{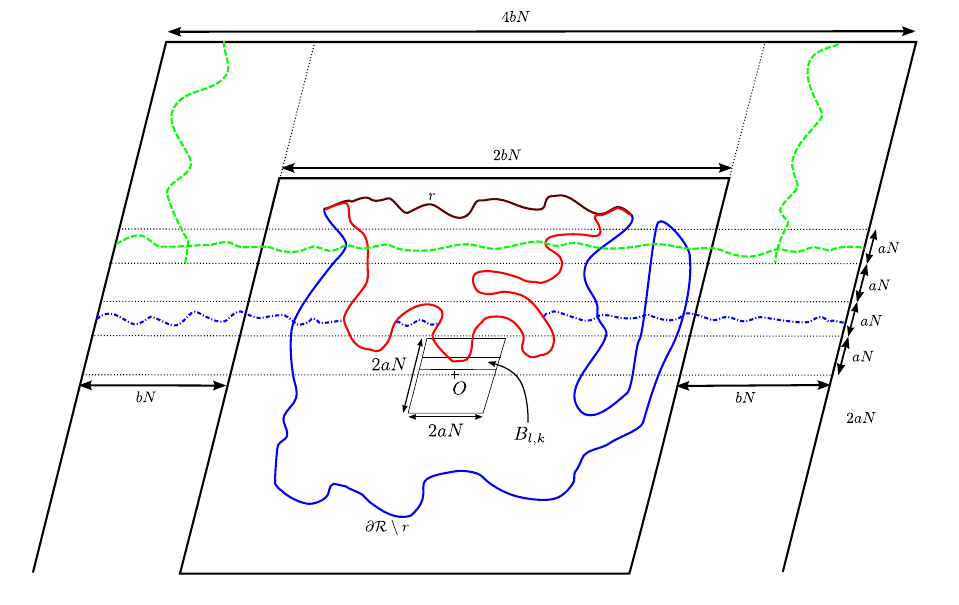}\caption{\label{fig: lowest of lowest regular region} The dashed paths are
the closed crossings of the event $\mathcal{D}$ which allow us to
prolong $\gamma.$ The dashed-dotted paths are the open parts of $\xi\left(\gamma\right).$
They, together with $\gamma,$ prevent the occurrence of closed vertices
below the lowest point of $\gamma$ with two closed arms to the top
side of $B\left(2bN\right).$}
\end{figure}

Let $s_{L}$ ($s_{R}$) denote the starting (ending) vertex of $r.$
We say that a path $\rho\subseteq R\cup r$ is good, if it
\begin{itemize}
\item starts at $s_{L}$ and ends at $s_{R},$
\item it is non-self touching
\item and one of its lowest points is in $B_{l,k}.$
\end{itemize}
Let $\rho$ be a fixed good path. Let $Be\left(\rho,R\right)$ denote
the set of vertices in $R$ `under' $\rho.$ It is the intersection
of $R$ with the connected component of $\partial R\setminus r$ in
$cl\left(R\right)\setminus\rho.$ Let $Ab\left(\rho,R\right):=R\setminus Be\left(\rho,R\right).$
Recall Definition \ref{def: HCr}.

Let $O_{\rho}$ denote the event that there is path $\nu$ such that 
\begin{itemize}
\item $\nu$ is non self-touching,
\item $\nu\subseteq B_{0}:=\left[-2bN,2bN\right]\boxtimes\left[-aN,2aN\right]$,
\item $\nu$ connects the left and the right side of the parallelogram $B_{1}:=\left[-2bN,2bN\right]\boxtimes\left[aN,2aN\right],$
\item $\nu\setminus R\subset B_{1}$ and the vertices in $\nu\setminus R$
are open,
\item each of the paths of $\nu\cap R$ is a concatenation of some open
paths which lie in $Be\left(\rho,B\left(bN\right)\right)\cap B_{1},$
and of some subpaths of $\rho.$
\end{itemize}
Let $\gamma$ denote the lowest non-self touching path in $R\cup r$
which starts at $s_{L}$ and ends at $s_{R},$ and of which all the
vertices outside of $r$ are closed. Note that on the event $L_{l,k}\left(R,r\right),$
$\gamma$ is good. By simple modifications of the arguments in the
proof of Lemma \ref{lem: lowest of lowest} we get that there are
$c_{1}=c_{1}\left(a,b,\lambda_{1},\lambda_{2}\right)>0$ and $N_{1}=N_{1}\left(a,b,\lambda_{1},\lambda_{2}\right)$
such that 
\begin{equation}
\mathbb{P}_{p_{\lambda}\left(N\right)}\left(L_{l,k}\left(R,r\right)\cap O_{\gamma}\right)\geq c_{1}\mathbb{P}_{p_{\lambda}\left(N\right)}\left(L_{l,k}\left(R,r\right)\right)\label{eq: pf compare L - 1}
\end{equation}
for $l,k\in\mathbb{N},$ $0\leq l\leq k-1,$ $\lambda\in\left[\lambda_{1},\lambda_{2}\right]$
for $N\geq N_{1}.$ 

Recall Definition \ref{def: HCr}. Let $\zeta\in HCr\left(B\left(2bN\right)\right).$
On the event $L_{l,k}\left(R,r\right)\cap O_{\gamma}$ we have $R\cap\left(\mathbb{Z}\boxtimes\left[3aN,bN\right]\right)\subset Ab\left(\xi\left(\gamma\right),B\left(2bN\right)\right).$
Hence the event $L_{l,k}\left(R,r\right)\cap O_{\gamma}\cap\left\{ \xi\left(\gamma\right)=\zeta\right\} $
is decreasing on the configuration in $Ab\left(\zeta,B\left(2bN\right)\right).$
Let $B_{2}=\left[-2bN,2bN\right]\boxtimes\left[3aN,4aN\right],$ $B_{3}=\left[-2bN,-bN\right]\boxtimes\left[3aN,2bN\right],$
$B_{4}=\left[bN,2bN\right]\boxtimes\left[3aN,2bN\right]$ and $\mathcal{D}=\mathcal{H}_{c}\left(B_{2}\right)\cap\mathcal{V}_{c}\left(B_{3}\right)\cap\mathcal{V}_{c}\left(B_{4}\right).$
The arguments of the proof of Lemma \ref{lem: lowest of lowest} give
that there exist $c_{2}=c_{2}\left(a,b,\lambda_{1},\lambda_{2}\right)>0$
and $N_{2}=N_{2}\left(a,b,\lambda_{1},\lambda_{2},k\right)$ such
that 
\begin{equation}
\mathbb{P}_{p_{\lambda}\left(N\right)}\left(L_{l,k}\left(R,r\right)\cap O_{\gamma}\cap\mathcal{D}\right)\geq c_{2}\mathbb{P}_{p_{\lambda}\left(N\right)}\left(L_{l,k}\left(R,r\right)\cap O_{\gamma}\right)\label{eq: pf compare L - 2}
\end{equation}
for $l,k\in\mathbb{N},$ $0\leq l\leq k-1,$ $\lambda\in\left[\lambda_{1},\lambda_{2}\right]$
for $N\geq N_{2}.$ Note that $L_{l,k}\left(R,r\right)\cap O_{\gamma}\cap\mathcal{D}\subset L_{l,k}\left(B\left(2bN\right),top\left(2bN\right)\right).$
See Figure \ref{fig: lowest of lowest regular region} for more details.
This combined with (\ref{eq: pf compare L - 1}) and (\ref{eq: pf compare L - 2})
finishes the proof of Lemma \ref{lem: compare L}.
\end{proof}
A combination of Lemma \ref{lem: lowest of lowest} and \ref{lem: compare L}
gives the following:
\begin{cor}
\label{cor: lowest of lowest regular regions} Suppose that the conditions
of Lemma \ref{lem: lowest of lowest} hold. Then there exist $c=c\left(a,b,\lambda_{1},\lambda_{2}\right)$
and $N_{0}=N_{0}\left(a,b,\lambda_{1},\lambda_{2},k\right)$ such
that 
\[
\mathbb{P}_{p_{\lambda}\left(N\right)}\left(L_{l,k}\left(R,r\right)\right)\leq ck^{-1}
\]
for $l=0,1,\ldots,k-1,$ $\lambda\in\left[\lambda_{1},\lambda_{2}\right]$
and $N\geq N_{0}.$
\end{cor}

\subsection{\label{sub: diam of active clusters}The diameter of the active clusters
close to time $1/2$}

We turn to the $N$-parameter frozen percolation process. In the introduction
we indicated that the $N$-parameter frozen percolation process exists
since it is a finite range interacting particle system. It is also
true that the process is measurable with respect to the $\tau$ values. 
\begin{defn}
\label{def: filtration of tau} For $t\in\left[0,1\right]$ and $J\subset V$
let 
\[
\mathcal{F}_{t}\left(J\right):=\sigma\left(\left\{ \tau_{w}<s\right\} \left|w\in J,\, s\in\left[0,1\right]\right.\right)
\]
denote the $\sigma$-algebra generated by the $\tau$ values of the
vertices in $J$ up to time $t.$
\end{defn}
The following lemma follows from the arguments in the second lecture
of \cite{Durrett1995}.
\begin{lem}
\label{lem: measurable w.r.t tau-s}For $N\in\mathbb{N},$ the $N$-parameter
frozen percolation process is adapted to the filtration $\mathcal{F}_{t}\left(V\right).$

\end{lem}
Recall the notation $\mathcal{C}_{a}\left(v;\lambda\right)$ from
Definition \ref{def: C_a}. We prove the following proposition.
\begin{namedthm}
[Proposition \ref{prop: active diameter}] For all $\lambda\in\mathbb{R}$
and $\varepsilon,K,\alpha>0,$ there exist $\theta=\theta\left(\lambda,\alpha,\varepsilon,K\right)>0$
and $N_{0}=N_{0}\left(\lambda,\alpha,\varepsilon,K\right)$ such that

\begin{equation}
\mathbb{P}_{N}\left(\exists v\in B\left(KN\right)\mbox{ s.t. }\diam\left(\mathcal{C}_{a}\left(v;\lambda\right)\right)\in\left(\left(\alpha-\theta\right)N,\left(\alpha+\theta\right)N\right)\right)<\varepsilon\label{eq: prop active diameter}
\end{equation}
for $N\geq N_{0}.$\end{namedthm}
\begin{proof}
[Proof of Proposition \ref{prop: active diameter}] Due to the length
of the proof, we first give an outline. Let $\lambda,\varepsilon,K,\alpha$
as in the statement of Proposition \ref{prop: active diameter}.

For simplicity, we only give a sketch which shows that we can choose
$\theta\in\left(0,\frac{1\wedge\alpha}{2}\right)$ such that 
\begin{equation}
\mathbb{P}_{N}\left(\diam\left(\mathcal{C}_{a}\left(\lambda\right)\right)\in\left(\left(\alpha-\theta\right)N,\left(\alpha+\theta\right)N\right)\right)<\varepsilon\label{eq: pf active diameter - 0}
\end{equation}
for large $N.$

Let us denote by $\tilde{x},\tilde{y}$ a pair of sites in the active
cluster of the origin for which $d\left(\tilde{x},\tilde{y}\right)=\diam\left(\mathcal{C}_{a}\left(\lambda\right)\right).$
We consider the case where $\tilde{x}$ is one of the lowest and $\tilde{y}$
is one of the highest vertices of the active cluster. The other case
where the diameter is achieved as a distance between a leftmost and
rightmost vertex can be treated in a similar way. Let $x$ ($y$)
denote a vertex which is a neighbour of $\tilde{x}$ ($\tilde{y}$),
and lies below (above) it. Note that $x$ and $y$ are closed frozen
vertices at time $p_{\lambda}\left(N\right).$ 

In Step 1 we apply Observation \ref{obs: closed grid-> no freezing}
and Lemma \ref{lem: there is a net} to set $\lambda_{0}$ so that
with probability close to $1,$ there are no frozen clusters at time
$p_{\lambda_{0}}\left(N\right)$ in $B\left(\left(\alpha+2\right)N\right).$
Hence in the case where $\lambda_{0}\geq\lambda$ the statement of
Proposition \ref{prop: active diameter} follows. In the following
we assume that $\lambda_{0}<\lambda,$ and the event in (\ref{eq: prop active diameter})
is non-empty. We investigate the configuration close to $x.$ In Step
2, we show that with probability close to $1,$ there is a unique
frozen cluster $F$ close to $x.$ By Step 1, we can assume that it
froze at time $p_{\lambda_{F}}\left(N\right)$ for $\lambda_{F}\in\left[\lambda_{0},\lambda\right].$
In Step 4, we show that with probability close to $1,$ there is a
graph $\mathcal{R}\subseteq\mathbb{T}$ such that its boundary consists
of a $p_{\lambda_{F}}\left(N\right)$-closed arc, denoted by $r_{c},$
and a $p_{\lambda_{F}}\left(N\right)$-open arc. In Step 3,5 and 6
we show that with probability close to $1,$ we can impose some extra
conditions on $\mathcal{R}$ and $r_{c}$ and on the configuration
in $\mathcal{R}.$ We get a pair $\left(\mathcal{R},r_{c}\right)$
with the following properties:
\begin{itemize}
\item $\partial\mathcal{R}$ is a certain outermost circuit, which is measurable
with respect to the $\tau$-values in $\mathbb{T\setminus\mathcal{R}},$
(Step 4)
\item $x$ is one of the lowest vertices of $\mathcal{R}$ with two non-touching
$p_{\lambda_{F}}\left(N\right)$-closed arms in $\mathcal{R}$ to
$r_{c},$ (Step 4)
\item no matter how we change the $\tau$ values in $\mathcal{R},$ the
$N$-parameter frozen percolation outside $\mathcal{R}$ does not
change up to time $p_{\lambda}\left(N\right),$ (Step 3)
\item satisfies a technical condition (Step 5)
\item $y\in\mathbb{T\setminus}cl\left(\mathcal{R}\right)$ (Step 4).
\end{itemize}
Let us condition on the $\tau$-values in $\mathbb{T}\setminus\mathcal{R}.$
The first and the third property of $\left(\mathcal{R},r_{c}\right)$
implies that at time $p_{\lambda_{F}}\left(N\right),$ the vertices
in $\mathcal{R}$ are open with probability $p_{\lambda_{F}}\left(N\right)$
and closed with probability $1-p_{\lambda_{F}}\left(N\right)$ \emph{independently}
from each other. This combined with $y\in\mathbb{T}\setminus\mathcal{R}$
allows us to decouple the locations of $x$ and $y.$ Since $d\left(\tilde{x},\tilde{y}\right)=\diam\left(\mathcal{C}_{a}\left(\lambda\right)\right),$
to prove (\ref{eq: pf active diameter - 0}), it is enough to show
that the second coordinate of $x$ is not concentrated when we condition
on the configuration in $\mathbb{T}\setminus\mathcal{R}.$ We would
like to use Corollary \ref{cor: lowest of lowest regular regions}
for the pair $\left(\mathcal{R},r_{c}\right).$ Unfortunately, this
pair $\left(\mathcal{R},r_{c}\right)$ might not satisfy all the conditions
of Definition \ref{def: (a,b)-regular}. To solve this problem we
use the technical condition of Step 5 and we construct the pair $\left(\tilde{\mathcal{R}},\tilde{r}_{c}\right)$
from $\left(\mathcal{R},r_{c}\right)$ using a deterministic procedure
in Step 6 such that
\begin{itemize}
\item $\tilde{\mathcal{R}}\subset\mathcal{R},$
\item a translated version of $\left(\tilde{\mathcal{R}},\tilde{r}_{c}\right)$
is $\left(\alpha_{3}N,\alpha_{2}N\right)$-regular as of Definition
\ref{def: (a,b)-regular} for some $\alpha_{2},\alpha_{3}>0,$ and
\item $x$ is one of the lowest vertices of $\tilde{\mathcal{R}}$ with
two non-touching $p_{\lambda_{F}}\left(N\right)$-closed arms in $\tilde{\mathcal{R}}$
to $\tilde{r}_{c}.$
\end{itemize}
We apply Corollary \ref{cor: lowest of lowest regular regions} to
$\left(\tilde{\mathcal{R}},\tilde{r}_{c}\right)$ and get the required
non-concentration result and finish the proof of Proposition \ref{prop: active diameter}.
We make this argument precise in Step 7.
\begin{rem*}
The structure of the proofs in Step 2-6 is an arm event hunting procedure.
We take a some small neighbourhood of $x.$ We deduce that if the
required condition is violated, then certain mixed near-critical arm
events or crossing events of thin parallelograms occur. These events
have upper bounds with exponents strictly larger than $2.$ This implies
that by choosing the neighbourhood small enough, we can set their
probability as small as we want. In particular, we get that the probability
of the event where the condition of the step is not satisfied is as
small as required, and finishes the proof of the step.
\end{rem*}
\medskip

Let us turn to the precise proof.

\textbf{Step 1. }\emph{We set $\lambda_{0}$ such that with probability
close to $1,$ at time $p_{\lambda_{0}}\left(N\right),$ none of the
open clusters intersecting $B\left(\left(2\alpha+K+2\right)N\right)$
are frozen.}

By Lemma \ref{lem: there is a net} we choose $\lambda_{0}=\lambda_{0}\left(\alpha,\varepsilon,K\right)$
and $N_{0}=N_{0}\left(\alpha,\varepsilon,K\right)$ such that the
event
\begin{equation}
E_{0}:=\mathcal{N}_{c}\left(\lambda_{0},1/24,2\alpha+K+4,N\right)\label{eq: pf active diameter - 0.5}
\end{equation}
has probability at least $1-\varepsilon/20$ for $N\geq N_{0}.$ Then
by Observation \ref{obs: closed grid-> no freezing}  we have that
none of open clusters intersecting $B\left(\left(2\alpha+K+2\right)N\right)$
are frozen. In particular, if a vertex $v\in B\left(\left(2\alpha+K+2\right)N\right)$
is closed at time $p_{\lambda}\left(N\right),$ then it is $p_{\lambda_{0}}\left(N\right)$-closed.
Moreover, if $v\in B\left(\left(2\alpha+K+2\right)N\right)$ is open
at time $p_{\lambda}\left(N\right),$ then it is $p_{\lambda}\left(N\right)$-open.
This finishes Step 1.
\begin{rem*}
Note that in the definition of $E_{0}$ above, we set the second argument
of $\mathcal{N}_{c}$ to $1/24,$ which is smaller than $1/6$ which
appears in Observation \ref{obs: closed grid-> no freezing}. The
reason for this choice will become clear in Step 3.
\end{rem*}
\medskip

Let $\theta\in\left(0,\frac{1\wedge\alpha}{2}\right).$ For $i=1,2,$
let $BA^{i}=BA^{i}\left(\theta\right)$ denote the set of vertices
$v\in B\left(KN\right)$ such that there are $\tilde{x}\left(v\right)=\left(\tilde{x}_{1}\left(v\right),\tilde{x}_{2}\left(v\right)\right),\,\tilde{y}\left(v\right)=\left(\tilde{y}_{1}\left(v\right),\tilde{y}_{2}\left(v\right)\right)\in\mathcal{C}_{a}\left(v;\lambda\right)$
such that 
\begin{equation}
\tilde{y}_{i}\left(v\right)-\tilde{x}_{i}\left(v\right)=d\left(\tilde{x}\left(v\right),\tilde{y}\left(v\right)\right)=\diam\left(\mathcal{C}_{a}\left(v;\lambda\right)\right)\in\left(\left(\alpha-\theta\right)N,\left(\alpha+\theta\right)N\right).\label{eq: pf active dimeter - 1}
\end{equation}

Note that 
\begin{equation}
\left\{ \exists v\in B\left(KN\right)\mbox{ s.t. }\diam\left(\mathcal{C}_{a}\left(v;\lambda\right)\right)\in\left(\left(\alpha-\theta\right)N,\left(\alpha+\theta\right)N\right)\right\} =\left\{ BA^{1}\cup BA^{2}\neq\emptyset\right\} .\label{eq: pf active diameter - 1.5}
\end{equation}
Let $u\in BA^{2}.$ In the following we define quantities which depend
on the value of $u.$ In notation we indicate the dependence on $u$
in the first appearance of these quantities, or when we want to emphasize
this dependence. For each $u\in BA^{2}$ we fix a pair $\left(\tilde{x},\tilde{y}\right)=\left(\tilde{x},\tilde{y}\right)\left(u\right)$
which satisfies (\ref{eq: pf active dimeter - 1}). It can happen
that there are more than one candidates for $\tilde{x}$ or $\tilde{y}.$
In this case we choose one of them in some deterministic way. (E.g
we can set $\tilde{x}$ ($\tilde{y}$) as the leftmost vertex among
the candidates.) Let $x=x\left(u\right)$ ($y\left(u\right)$) denote
a (deterministically chosen) neighbour of $\tilde{x}$ ($\tilde{y}$)
below $\tilde{x}$ (above $\tilde{y}$). The active cluster $\mathcal{C}_{a}\left(u;\lambda\right)$
lies between the horizontal lines passing through $x$ and $y$ denoted
by $e_{x}$ and $e_{y}.$ Since $\theta<\alpha/2,$ the outer boundary
of $\mathcal{C}_{a}\left(u;\lambda\right)$ provides two non-touching
closed half plane arms in $x+\mathbb{Z}\boxtimes\left[0,\infty\right)$
to distance $\alpha N/2$ starting from $x.$ Since $\partial\mathcal{C}_{a}\left(u;\lambda\right)\subset B\left(\left(2\alpha+K+2\right)N\right),$
by Step 1, on the event $E_{0}$ these arms are $p_{\lambda_{0}}\left(N\right)$-closed.
We denote the one on the left (right) hand side by $c_{L}=c_{L}\left(u\right)$
($c_{R}=c_{r}\left(u\right)$). Apart from their common starting point,
$c_{L}$ and $c_{R}$ do not even touch, since any active path connecting
$\tilde{x}$ to $\tilde{y}$ separates them. Since $x$ is a closed
frozen vertex, there is at least one open frozen neighbour of $x.$
From this vertex there is a $p_{\lambda}\left(N\right)$-open arm
$o_{B}=o_{B}\left(u\right)$ to distance at least $N/2.$ See Figure
\ref{fig: mixed arms} for more details.

\begin{figure}
\centering{}\includegraphics{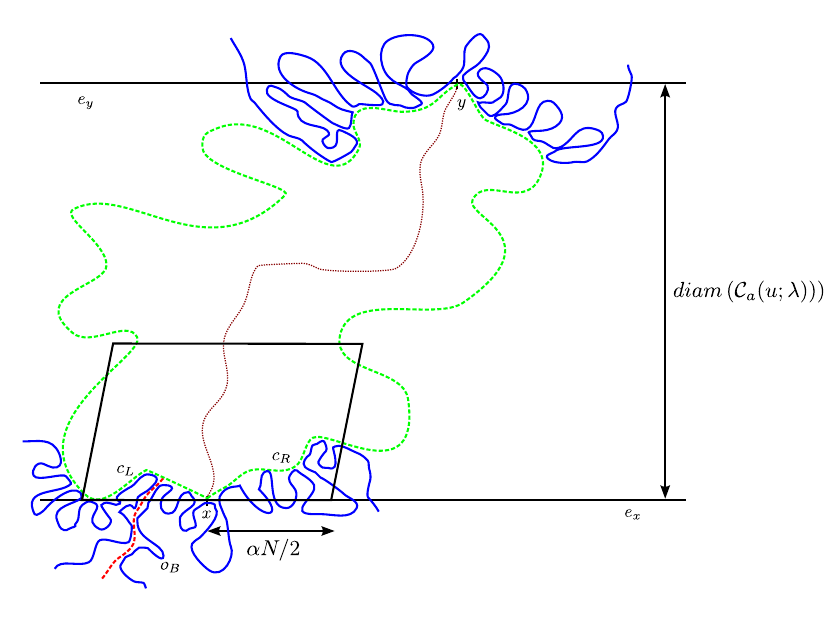}\caption{\label{fig: mixed arms} The closed boundary of $\mathcal{C}_{a}\left(\lambda\right)$
give rise to the closed arms $c_{L}$ and $c_{R}$ from $x$ to $\partial B\left(x;\alpha N/2\right).$
The frozen vertex neighbouring $x$ provides the arm $o_{B}.$ }
\end{figure}

Let $\beta,\beta'\in\left(0,1\right)$ with $\beta<\beta'.$ Recall
the definition of the events $\mathcal{NA}\left(\beta,\beta'\right):=\mathcal{NA}\left(\beta,\beta',\lambda,\lambda_{0},2\alpha+K+2,N\right)$
and $\mathcal{NC}\left(\beta,\beta'\right):=\mathcal{NC}\left(\beta,\beta',\lambda,\lambda_{0},2\alpha+K+2,N\right)$
from Corollary \ref{cor: no many arms} and \ref{cor: thin crossing}.
In the following we introduce the constants $\alpha_{i}>0$ for $i=1,2,3$
such that $\alpha_{i}/\alpha_{i+1}\gg1.$ Let $\alpha_{3}\in\left(0,\frac{\alpha\wedge1}{2}\right).$
Let $z=z\left(u\right)\in V$ such that $x=x\left(u\right)\in\left[-\alpha_{3}N,\alpha_{3}N\right]\boxtimes\left(-\alpha_{3}N,\alpha_{3}N\right]+\left\lfloor \alpha_{3}N\right\rfloor z.$
Note that $z\in B\left(\left\lceil \frac{2\alpha+K+2}{\alpha_{3}}\right\rceil \right).$
We define $B_{3}=B_{3}\left(u\right):=B\left(\left\lfloor \alpha_{3}N\right\rfloor z;\alpha_{3}N\right).$
Note that throughout the arguments below, we will assume that $\alpha_{1}>\alpha_{2}>\alpha_{3},$
however, we will set their precise values only in later stages of
the proof.

\medskip

\textbf{Step 2.} \emph{We show that with probability close to $1,$
there is only one frozen cluster close to $x=x\left(u\right)$ for
all $u\in BA^{2}.$}

Let $\alpha_{1}\in\left(0,\frac{\alpha\wedge1}{2}\right),$ $B_{1}=B_{1}\left(u\right):=B\left(\left\lfloor \alpha_{3}N\right\rfloor z;\alpha_{1}N\right)$
and $A_{1}=A_{1}\left(u\right):=A\left(\left\lfloor \alpha_{3}N\right\rfloor z;\alpha_{1}N,\frac{\alpha\wedge1}{2}N\right).$
Suppose that there are at least two different frozen clusters in $B_{1}.$
On the event $E_{0}$ we find $5,2$ mixed near critical arms in $A_{1}:$
the two $p_{\lambda_{0}}\left(N\right)$-closed arms $c_{L}$ and
$c_{R},$ the two $p_{\lambda}\left(N\right)$-open arms from the
two frozen clusters, and a $p_{\lambda_{0}}\left(N\right)$-closed
arm separating them. Let $E_{1}:=\mathcal{NA}\left(\alpha_{1},\frac{\alpha\wedge1}{2}\right).$
Hence we get:
\begin{claim}
\label{claim: unique F}On the event $E_{0}\cap E_{1},$ $\forall u\in BA^{2},$
there is a unique frozen cluster denoted by $F=F\left(u\right)$ which
intersects $B_{1}\left(u\right).$ Let $\lambda_{F}=\lambda_{F}\left(u\right)\in\left[\lambda_{0},\lambda\right]$
such that $F$ froze at $p_{\lambda_{F}}\left(N\right).$ On $E_{0}\cap E_{1},$
a vertex in $B_{1}\left(u\right)$ is open in the $N$-parameter frozen
percolation process at time $p_{\lambda_{F}}\left(N\right)$ if and
only if it is $p_{\lambda_{F}}\left(N\right)$-open. 
\end{claim}
In the following two steps we write open (closed) for $p_{\lambda_{F}}\left(N\right)$-open
($p_{\lambda_{F}}\left(N\right)$-closed) if it is not stated otherwise.
We finish Step 2 by applying Corollary \ref{cor: no many arms} and
we set $\alpha_{1}$ such that
\begin{equation}
\mathbb{P}\left(E_{1}\right)\geq1-\varepsilon/20\label{eq: pf active dimeter - 2}
\end{equation}
for $N\geq N_{1}\left(\varepsilon,\lambda_{0},\lambda,\alpha,K\right).$

\medskip

\textbf{Step 3.} \emph{We say that a circuit is $p_{\lambda_{F}}\left(N\right)$-open-closed,
or simply open-closed, if it consists of a $p_{\lambda_{F}}\left(N\right)$-open
and a $p_{\lambda_{F}}\left(N\right)$-closed arc. Suppose that there
is a $p_{\lambda_{F}}\left(N\right)$-open-closed circuit close to
and around $x$. We show that with probability close to $1,$ no matter
how we change the $\tau$ values inside this circuit, the $N$-parameter
frozen percolation process does not change till time $p_{\lambda}\left(N\right)$
outside of the circuit.}

Let $\alpha_{2}\in\left(0,\alpha_{1}\wedge\frac{1}{4}\right),$ and
$\beta_{2}\in\left(\alpha_{2},\alpha_{1}\right)$ be some intermediate
scale. We define the parallelograms
\begin{align*}
  B_{2} = & B_{2}\left(u\right):=B\left(\left\lfloor \alpha_{3}N\right\rfloor z;\alpha_{2}N\right),\\
  B_{2}'= & B_{2}'\left(u\right):=B\left(\left\lfloor \alpha_{3}N\right\rfloor z;\beta_{2}N\right),\\
  A_{2} = & A_{2}\left(u\right):=A\left(\left\lfloor \alpha_{3}N\right\rfloor z;\alpha_{2}N,\alpha_{1}N\right),\\
  A_{2}'= & A_{2}'\left(u\right):=A\left(\left\lfloor \alpha_{3}N\right\rfloor z;\beta_{2}N,\alpha_{1}N\right).
\end{align*}
Let $BL=BL\left(u\right)$ denote the set of bordering lines of $F\setminus B_{2}',$
that is the top- and bottom-most horizontal, left- and rightmost vertical
lines which intersect $F\setminus B_{2}'.$ We rule out the case where
there is a line in $BL$ which intersects $B_{2}'$ in the following
technical claim.
\begin{claim}
\label{claim: bdary lines of F are far from x}Let
\begin{equation}
E'_{2}=\mathcal{NA}\left(2\beta_{2},\alpha_{1}-2\beta_{2}\right)\cap\mathcal{NC}\left(2\beta_{2},2\alpha_{1}\right).\label{eq: pf active diameter - 2.1}
\end{equation}
Then 
\[
E_{0}\cap E_{1}\cap E_{2}'\subset E_{0}\cap E_{1}\cap\left\{ \forall u\in BA^{2},\forall e\in BL\left(u\right)\mbox{ we have }e\cap\left(F\setminus B'_{2}\right)=\emptyset\right\} .
\]
\end{claim}
\begin{proof}
[Proof of Claim \ref{claim: bdary lines of F are far from x}] Let
$u\in BA^{2}.$ When the bottom-most line of $F\setminus B_{2}'$
intersects $B_{2}',$ then $F\subseteq\left(\mathbb{Z}\boxtimes\left[-\beta_{2}N,\infty\right)\right)+\left\lfloor \alpha_{3}N\right\rfloor z.$
We see $4$ half plane arms: $c_{L},c_{R}$ give two closed and $o_{B}$
gives an open arm, a fourth closed half plane arm separates $F$ from
the line $\mathbb{Z}\boxtimes\left\{ \left\lfloor \beta_{2}N\right\rfloor \right\} +\left\lfloor \alpha_{3}N\right\rfloor z$.
Hence $\mathcal{NA}^{c}\left(2\beta_{2},\alpha_{1}-2\beta_{2}\right)$
occurs.

If the topmost line of $F\setminus B_{2}'$ intersects $B_{2}',$
then the closed arms $c_{L}$ and $c_{R}$ stay in the parallelogram
\[
\left[-\alpha_{1}N,\alpha_{1}N\right]\boxtimes\left[-\beta_{2}N,\beta_{2}N\right]+\left\lfloor \alpha_{3}N\right\rfloor z.
\]
In particular, $c_{L}$ gives a closed crossing of one of the parallelograms
\[
\left[-\alpha_{1}N,-\beta_{2}N\right]\boxtimes\left[-\beta_{2}N,\beta_{2}N\right]+\left\lfloor \alpha_{3}N\right\rfloor z\mbox{ or }\left[\beta_{2}N,\alpha_{1}N\right]\boxtimes\left[-\beta_{2}N,\beta_{2}N\right]+\left\lfloor \alpha_{3}N\right\rfloor z.
\]
That is, the event $\mathcal{NC}^{c}\left(2\beta_{2},2\alpha_{1}-2\beta_{2}\right)$
occurs.

When a leftmost bordering line of $F\setminus B_{2}'$ intersects
$B_{2}',$ then we find that the arms in $A\left(\left\lfloor \alpha_{3}N\right\rfloor z;\beta_{2}N,\alpha_{1}N\right)$
induced by $c_{L},c_{R}$ and $o_{B}$ stay in half plane 
\begin{equation}
\left[-2\beta_{2}N,\infty\right)\times\mathbb{R}+\left\lfloor \alpha_{3}N\right\rfloor z.\label{eq: pf claim bdary lines of F are far from x}
\end{equation}
The frozen cluster $F$ is separated from the line $\left\{ -2\beta_{2}N\right\} \times\mathbb{R}+\left\lfloor \alpha_{3}N\right\rfloor z.$
This provides an additional closed arm in the half plane (\ref{eq: pf claim bdary lines of F are far from x}),
which together the arms induced by $c_{L},c_{R}$ and $o_{B}$ give
$4$ half plane arms, hence the event $\mathcal{NA}^{c}\left(\beta_{2},\alpha_{1}-2\beta_{2}\right)$
occurs.

The case when the rightmost bordering line of $F\setminus B_{2}'$
intersects $B_{2}'$ can be treated similarly. 

With the notation (\ref{eq: pf active diameter - 2.1}) we get that
on the event $E_{0}\cap E_{1}\cap E'_{2},$ none of the lines of $BL$
intersect $B'_{2},$ which finishes the proof of Claim \ref{claim: bdary lines of F are far from x}. 
\end{proof}
Now we proceed with Step 3. Let $u\in BA^{2}.$ Suppose that there
is an open-closed circuit $OC=OC\left(u\right)$ around $x$ in $B_{2}.$
Let $I=I\left(u\right)$ denote the union of the finite connected
components of $\mathbb{T}\backslash OC.$ Let us change the $\tau$
values of the vertices in $I$ in some arbitrary non-degenerate way
(that is, the new $\tau$ values are all different), but keep the
original values outside $I.$ Let us run the $N$-parameter frozen
percolation dynamics for this modified set of $\tau$ values. We denote
this new process by $FPP'$ and $FPP$ denotes the original process.
Our next aim is to show that the processes $FPP$ and $FPP'$ coincide
on $V\setminus B_{2}$ till time $p_{\lambda}\left(N\right)$ on some
event $E_{2}$ independently from the choice of the new $\tau$ values.

Recall the definition of $E_{0}$ from (\ref{eq: pf active diameter - 0.5})
and the remark after Step 1. Since $\alpha_{2}<\alpha_{1}<1/24$ and
$I\subseteq B_{2},$ the definition of $E_{0}$ and Observation \ref{obs: closed grid-> no freezing}
gives that the processes $FPP$ and $FPP'$ coincide on $V\setminus I$
up to time $p_{\lambda_{0}}\left(N\right).$ In particular, the closed
arc of $OC$ stays closed till time $p_{\lambda_{F}}\left(N\right)$
in both processes. Hence it acts as a barrier for the effect of $\tau$
values in $I.$ By Step 2, the open arc of $OC$ is a subset of $F.$ 

\emph{Case 1.} The process $FPP'$ differs from $FPP$ outside of
$R$ at some time $t\in\left[0,p_{\lambda_{F}}\left(N\right)\right].$
\\
By Claim \ref{claim: bdary lines of F are far from x} on the event
$E_{0}\cap E_{1}\cap E_{2}'$ if these two processes differ outside
$I,$ then in the process $FPP'$ a frozen cluster $F'$ emerged before
time $p_{\lambda_{F}}\left(N\right)$ such that $F'\setminus I\neq F\setminus I.$
By the arguments above, we get that $F'$ froze in at time $p_{\lambda_{F'}}\left(N\right)$
with $\lambda_{F'}\in\left[\lambda_{0},\lambda_{F}\right].$ Let $BL'$
denote the set of bordering lines of $F'\setminus B_{2}'.$ With careful
examination of the proof of Claim \ref{claim: bdary lines of F are far from x}
one can see that the arguments applied there can also be applied to
the new process $FPP'.$ We get that, on the event $E_{0}\cap E_{1}\cap E_{2}'$
none of the lines of $BL'$ intersect $B_{2}'$ no matter how we modify
the $\tau$ values in $I.$ This implies that $F'\setminus I$ has
two connected components $F'_{1}$ and $F'_{2}$ such that $\diam\left(F_{i}'\right)<N$
for $i\in\left\{ 1,2\right\} ,$ but $\diam\left(F_{1}'\cup F_{2}'\right)\geq N.$
Since $I\subset B_{2},$ each of $F_{1}',F_{2}'$ contains a $p_{\lambda_{F'}}\left(N\right)$-open
arm in the annulus $A_{2}''=A_{2}''\left(u\right):=A\left(\left\lfloor \alpha_{3}N\right\rfloor z;\alpha_{2}N,\beta_{2}N\right).$
When for some $i\in\left\{ 1,2\right\} $ $F_{i}'$ lies above $c_{L}$
and $c_{R},$ then we get a $4,3$ near critical arm event: the closed
arms induced by $c_{L},c_{R}$ and the open arm induced by $F_{i}'$
stay above $e_{x},$ and $o_{B}$ provides the fourth arm in $A_{2}''.$
Hence $\mathcal{NA}^{c}\left(\alpha_{2},\beta_{2}\right)$ occurs.
If both of $F_{1}',F_{2}'$ lie below $c_{L}$ and $c_{R}$ then we
get a $5,2$ near critical mixed arm event in $A_{2}'':$ $c_{L},c_{R}$
induce closed half plane arms in $A_{2}''.$ $F_{1}',F_{2}'$ induce
two open arms. Since $F_{1}'$ and $F_{2}'$ are different connected
components of $F'\setminus I$, there is a fifth, $p_{\lambda_{F'}}\left(N\right)$-closed,
arm separating $F_{1}'$ and $F_{2}'$ in $A_{2}''.$ Hence $\mathcal{NA}^{c}\left(\alpha_{2},\beta_{2}\right)$
occurs. Let $E_{2}=E_{2}'\cap\mathcal{NA}\left(\alpha_{2},\beta_{2}\right).$ 

\emph{Case 2. }$FPP$ and $FPP'$ coincide on $V\setminus I$ till
$p_{\lambda_{F}}\left(N\right),$ but differ outside of $R$ at some
time $t\in\left(p_{\lambda_{F}}\left(N\right),p_{\lambda}\left(N\right)\right].$
\\
By Claim \ref{claim: bdary lines of F are far from x} and from that
the two processes coincide outside of $R,$ we get that a frozen cluster
$F'$ is formed at time $p_{\lambda_{F}}\left(N\right)$ in the new
process. Moreover, $F'\setminus I=F\setminus I.$ However, the two
processes differ at some time $t\in\left(p_{\lambda_{F}}\left(N\right),p_{\lambda}\left(N\right)\right],$
hence an additional frozen cluster $F''$ has to emerge in this time
period using some of the vertices in $I.$ This induces the $5,2$
near critical mixed arm event of Step 2. Hence we proved the following
claim.
\begin{claim}
\label{claim: conf in R is indep from rest}On the event $E_{0}\cap E_{1}\cap E_{2},$
we have that $\forall u\in BA^{2},$ if there is a $p_{\lambda_{F}}\left(N\right)$-open-closed
circuit around $x=x\left(u\right)$ in $B_{2}\left(u\right)$ then
no matter how we change the $\tau$ values inside this circuit, the
frozen percolation process outside it does not change till time $p_{\lambda}\left(N\right).$
\end{claim}
We finish Step 3 by applying Corollary \ref{cor: no many arms} and
\ref{cor: no thin crossing}: we fix the value of $\beta_{2}$ and
$\alpha_{2}$ such that

\begin{equation}
\mathbb{P}\left(E_{2}\right)\geq1-\varepsilon/20\label{eq: pf active dimeter - 3}
\end{equation}
for $N\geq N_{2}\left(\varepsilon,\lambda_{0},\lambda,\alpha,K\right).$

\medskip

\textbf{Step 4.} \emph{We show that with probability close to $1,$
there is a $p_{\lambda_{F}}\left(N\right)$-open-closed circuit around
$x,$ such that the location where its colour changes in the circuit
is `far` above $x.$ }

Let $u\in BA^{2}.$ Let $\alpha_{3}\in\left(0,\alpha_{2}\right),$
$B_{3}=B_{3}\left(u\right):=B\left(\left\lfloor \alpha_{3}N\right\rfloor z;\alpha_{3}N\right)$
and $A_{3}=A_{3}\left(u\right):=A\left(\left\lfloor \alpha_{3}N\right\rfloor z;\alpha_{3}N,\alpha_{2}N\right).$
Let $\delta_{3}\in\left(\alpha_{3},\alpha_{2}\right)$ be an intermediate
scale. We cut the annulus $A_{3}$ into three subannuli using two
other intermediate scales $\beta_{3},\beta_{3}'$ with $\alpha_{3}<\delta_{3}<\beta_{3}<\beta_{3}'<\alpha_{2}:$
\begin{align*}
A_{3,0}=A_{3,0}\left(u\right) & :=A\left(\left\lfloor \alpha_{3}N\right\rfloor z;\alpha_{3}N,\beta_{3}N\right),\\
A_{3,1}=A_{3,1}\left(u\right) & :=A\left(\left\lfloor \alpha_{3}N\right\rfloor z;\beta_{3}N,\beta_{3}'N\right),\\
A_{3,2}=A_{3,2}\left(u\right) & :=A\left(\left\lfloor \alpha_{3}N\right\rfloor z;\beta_{3}'N,\alpha_{2}N\right).
\end{align*}
Let $\bar{c}_{L}$ ($\bar{c}{}_{R}$) denote the closed arm induced
by $c_{L}$ ($c_{R}$) in $A_{3,1}.$

If $c_{L}$ and $c_{R}$ are not connected by a closed path in $A_{3,0}\cap\mathcal{C}_{a}\left(u;\lambda\right),$
then there is a open arm separating them. Hence we see a near critical
$4,3$ arm event: $c_{L},c_{R}$ and the separating open arm induce
disjoint half plane arms in $A_{3,0},$ and the fourth arm in $A_{3,0}$
is induced by $o_{B}$. Thus the event $\mathcal{NA}^{c}\left(\alpha_{3},\beta_{3}\right)$
occurs.

If $\bar{c}_{L}\subseteq\left[-\beta_{3}'N,-\beta_{3}N\right]\boxtimes\left[-\alpha_{3}N,\delta_{3}N\right]$
or $\bar{c}_{R}\in\left[\beta_{3}N,\beta_{3}'N\right]\boxtimes\left[-\alpha_{3}N,\delta_{3}N\right],$
then we find a closed horizontal crossing in a narrow parallelogram.
Hence the event $\mathcal{NC}^{c}\left(\alpha_{3}+\delta_{3},\beta_{3}'-\beta_{3}\right)$
occurs.

In the following we assume that both $\bar{c}_{L}$ and $\bar{c}_{R}$
leave the corresponding parallelograms. Let $w_{L}$ ($w_{R}$) be
an open frozen vertex neighbouring a vertex of $\bar{c}_{L}$ ($\bar{c}{}_{R}$)
which is outside of the aforementioned parallelogram.

Suppose that there is no open arc in $A_{3}$ connecting $w_{L}$
to $o_{B}.$ Since $w_{L}$ is open frozen at time $p_{\lambda_{F}}\left(N\right),$
it has a $p_{\lambda_{F}}\left(N\right)$-open path to distance $N/2.$
Let $o_{L}$ denote the part of this path till the first time it exits
$A_{3}.$ Note that $o_{L}$ and $o_{B}$ are disjoint, and they are
not connected by an open path inside $A_{3}.$ We have two cases depending
on where $o_{L}$ leaves $A_{3}.$ 

When it leaves $A_{3}$ by exiting its outer parallelogram, than we
get a $5,2$ near critical arm event in $A_{3,2}:$ two half plane
closed arms induced by $c_{L}$ and $c_{R},$ two open arms induced
by $o_{L}$ and $o_{B}$ an extra closed arm separates $o_{L}$ and
$o_{B}$ in $A_{3,2}.$ Hence the event $\mathcal{NA}^{c}\left(\beta_{3}',\alpha_{2}\right)$
occurs. See Figure \ref{fig:4and1/2arms}.

\begin{figure}
\begin{centering}
\includegraphics[scale=0.8]{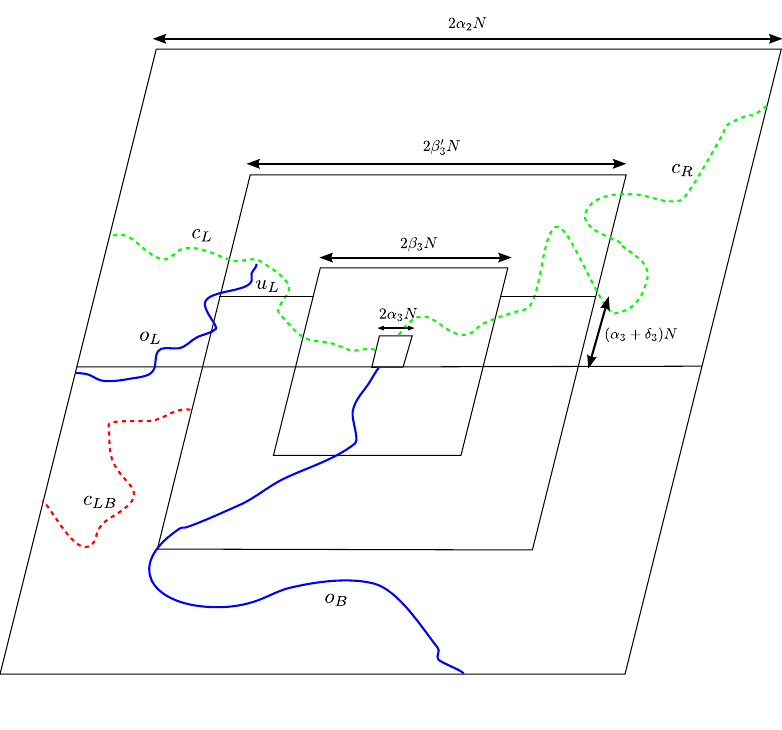}
\par\end{centering}

\caption{\label{fig:4and1/2arms} The closed arm $c_{LB}$ separates $o_{L}$
and $o_{B}$ in $A_{3,2}.$ Hence $c_{L},c_{R},o_{B},c_{LB},o_{L}$
give $5,2$ near critical mixed arms.}
\end{figure}

When $o_{L}$ leaves $A_{3}$ by entering its inner parallelogram,
then we get a similar $5,2$ arm event in $A_{3,0}.$ Thus $\mathcal{NA}^{c}\left(\alpha_{3},\beta_{3}\right)$
happens. In a similar way we can show that when $w_{R}$ is not connected
to $o_{B}$ in $A_{3},$ then $\mathcal{NA}^{c}\left(\beta_{3}',\alpha_{2}\right)\cup\mathcal{NA}^{c}\left(\alpha_{3},\beta_{3}\right)$
occurs. 

Let 
\[
E_{3}:=\mathcal{NC}\left(\alpha_{3}+\delta_{3},\beta_{3}'-\beta_{3}\right)\cap\mathcal{NA}\left(\alpha_{3},\beta_{3}\right)\cap\mathcal{NA}\left(\beta_{3}',\alpha_{2}\right)\cap\mathcal{NA}\left(\alpha_{3},\beta_{3}\right).
\]
Note that $w_{L},w_{R}\in\left(\mathbb{Z}\boxtimes\left[\delta_{2}N,\alpha_{2}N\right]\right)+\left\lfloor \alpha_{3}N\right\rfloor z,$
and that some parts of $c_{L}$ and $c_{R}$ are parts of the closed
arc of the open-closed circuit we constructed. See Figure \ref{fig:OC}
for more details. We arrive to the following claim.
\begin{claim}
\label{claim: nice open-closed circuit}On the event $E_{0}\cap E_{1}\cap E_{2}\cap E_{3},$
$\forall u\in BA^{2}$ there is a $p_{\lambda_{F}\left(u\right)}\left(N\right)$-open-closed
circuit $OC=OC\left(u\right)$ with the following properties:
\begin{enumerate}
\item it is contained in $A_{3}\left(u\right)$ and surrounds $B_{3}\left(u\right),$
\item the locations where the colour changes in $OC$ is contained $\left(\mathbb{Z}\boxtimes\left[\delta_{3}N,\alpha_{2}N\right]\right)+\left\lfloor \alpha_{3}N\right\rfloor z$
\item the endpoints of the closed part of $OC$ lie in the parallelogram
$\left[-\alpha_{2}N,\alpha_{2}N\right]\boxtimes\left[\delta_{3}N,\alpha_{3}N\right]+\left\lfloor \alpha_{3}N\right\rfloor z,$
\item as we walk from the outside of $B_{2}=B\left(\left\lfloor \alpha_{3}N\right\rfloor z;\alpha_{2}N\right)$
on any of the closed arms $c_{L}$ or $c_{R}$ towards $x,$ we hit
the closed part of $OC$ at its endpoints for the first time.
\end{enumerate}
\end{claim}
We finish Step 4 by choosing the values of $\beta_{3},\beta_{3}'$
and $\delta_{3}.$ The probability of $E_{3}$ is an increasing function
of $\alpha_{3}$ for $\beta_{3},\beta_{3}',\delta_{3}$ fixed. By
Corollary \ref{cor: no many arms} and \ref{cor: no thin crossing}
we choose the value of $\beta_{3},\beta_{3}',\delta_{3},\alpha_{3}$
such that the probability of the event $E_{3}$ is at least $1-\varepsilon/20.$
We only fix $\beta_{3},\beta_{3}',\delta_{3}$ and require $\alpha_{3}$
to be small but unspecified so that 
\begin{equation}
\mathbb{P}\left(E_{3}\right)\geq1-\varepsilon/20\label{eq: pf active dimeter - 4}
\end{equation}
 for $N\geq N_{3}\left(\varepsilon,\alpha_{3},\lambda_{0},\lambda,\alpha,K\right).$
We choose the value of $\alpha_{3}$ in Step 6.

\medskip

\begin{figure}
\centering{}\includegraphics[scale=0.8]{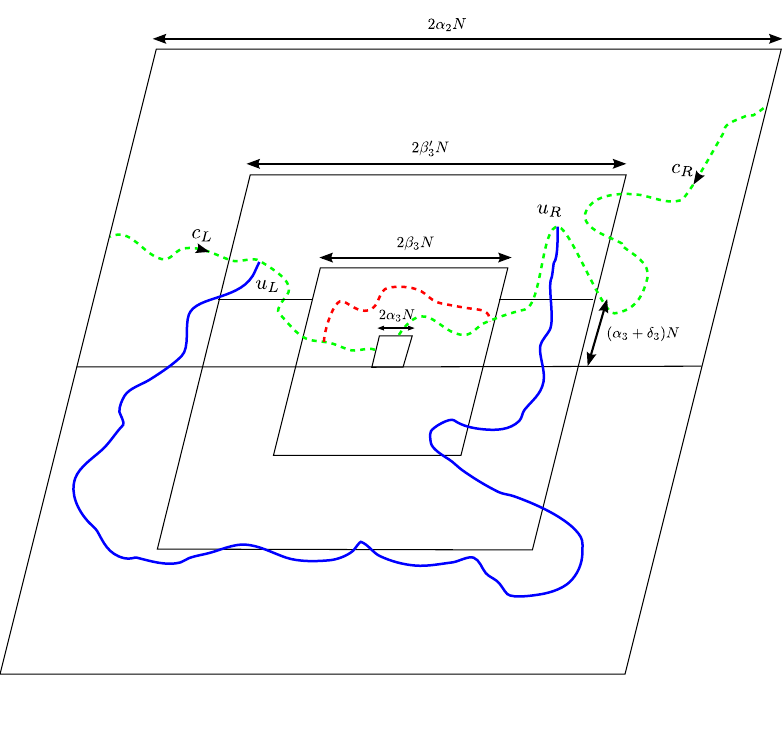}\caption{\label{fig:OC}The circuit around $B_{3}$ consists of the open arc
drawn with continuous line, subpaths of $c_{L}$ and $c_{R}$ and
the closed arc in $A_{3,0}.$ }
\end{figure}

Before Step 5, let us summarize what we have proved up to now. Let
$u\in BA^{2},$ and suppose that the event $E_{0}\cap E_{1}\cap E_{2}\cap E_{3}\left(\alpha_{3}\right)$
holds. It is easy to see that the outermost open-closed circuit which
satisfy the conditions of Claim \ref{claim: nice open-closed circuit}
is well-defined. Let $\mathcal{OC}$ denote this outermost circuit,
and $a_{c}$ ($a_{o}$) denote the closed (open) arcs of $\mathcal{OC}.$
Further simple considerations give:
\begin{claim}
\label{claim: OC measurable}On $E_{0}\cap E_{1}\cap E_{2}\cap E_{3}\left(\alpha_{3}\right),$
for any deterministic open-closed circuit $OC,$ one can check the
occurrence the event $\left\{ \mathcal{OC}=OC\right\} $ by looking
at the $\tau$ values in the closure of the unbounded component of
$V\setminus OC.$
\end{claim}
Let $\mathcal{R}$ denote the connected component of $B_{3}$ in $\mathbb{T}\setminus\mathcal{OC}.$
Let $r_{o}\subseteq a_{o}$ and $r_{c}\subseteq a_{c}$ denote the
open and closed parts of $\partial\mathcal{R}.$ The pair $\left(\mathcal{R},r_{c}\right),$
$\mathcal{OC}$ and the configuration in $\mathbb{T\setminus}\mathcal{R}$
satisfies the following conditions:
\begin{enumerate}
\item $\mathcal{R}$ is a connected induced subgraph of $\mathbb{T}$ (definition
of $\mathcal{R}$)
\item $B\left(\left\lfloor \alpha_{3}N\right\rfloor z;\alpha_{3}N\right)=B_{3}\subseteq\mathcal{R}\subseteq B_{2}=B\left(\left\lfloor \alpha_{3}N\right\rfloor z;\alpha_{2}N\right)$
(by Claim \ref{claim: nice open-closed circuit})
\item \label{cond: conf out R}$\partial\mathcal{R}$ is disjoint union
of non-empty self avoiding paths $r_{c}$ and $r_{o},$ which are
oriented such that $\mathcal{R}$ lies on the right when we walk along
them. We orient $a_{c},a_{o}$ so that the orientations of $a_{c}$
and $r_{c}$ ($a_{o}$ and $r_{o}$) are compatible. ($\mathcal{OC}$
is outermost)
\item \label{cond: position of r_c }$r_{c},a_{c}\subseteq\left[-\alpha_{2}N,\alpha_{2}N\right]\boxtimes\left[-\alpha_{3}N,\alpha_{2}N\right]+\left\lfloor \alpha_{3}N\right\rfloor z,$
(by the proof of Claim \ref{claim: nice open-closed circuit})
\item \label{cond: end of r_c}the endpoints of $a_{c}$ denoted by $s_{L}$
and $s_{R}$ lie in the parallelogram $\left[-\alpha_{2}N,\alpha_{2}N\right]\boxtimes\left[\delta_{3}N,\alpha_{3}N\right]+\left\lfloor \alpha_{3}N\right\rfloor z,$
(by Claim \ref{claim: nice open-closed circuit})
\item \label{cond: c_L hits OC at s_L}when we walk along $c_{L}$ ($c_{R}$)
towards $x,$ we hit $\mathcal{OC}$ first at vertex $s_{L}$ ($s_{R}$),
(by Claim \ref{claim: nice open-closed circuit})
\item \label{cond: open bdry point has closed path}for every vertex $v\in a_{o},$
there is a closed path in $B_{2}\setminus\mathcal{R}$ to $\partial B_{2},$
($\mathcal{OC}$ is outermost)
\item \label{cond: closed bdry point has open path}for every vertex $v\in a_{c},$
there is an open path in $B_{2}\setminus\mathcal{R}$ to $\partial B_{2}$
or to $\left(c_{L}\cup c_{R}\right)\setminus cl\left(\mathcal{R}\right).$
($\mathcal{OC}$ is outermost)
\end{enumerate}
Note that the first three conditions coincide with the first three
conditions for the pair $\left(\mathcal{R}-\left\lfloor \alpha_{3}N\right\rfloor z,r_{c}-\left\lfloor \alpha_{3}N\right\rfloor z\right)$
being $\left(\alpha_{3}N,\alpha_{2}N\right)$-outer-regular of Definition
\ref{def: (a,b)-regular}. We add an extra condition in the next step.

Note that the vertex $x$ has two non-touching closed arms to $r_{c}.$
Moreover, by Condition \ref{cond: c_L hits OC at s_L} above, $x$
is one of the lowest vertices in $\mathcal{R}$ with this property.
With the notation of Definition \ref{def: cal L} we have that $x\in\mathcal{L}\left(\mathcal{R},r_{c}\right)$
in the $N$-parameter frozen percolation process at time $p_{\lambda_{F}}\left(N\right).$

\medskip 

\textbf{Step 5.} Let $u\in BA^{2}.$ \emph{Suppose that the event
$E_{0}\cap E_{1}\cap E_{2}\cap E_{3}$ holds.} Let $\mathcal{W}=\mathcal{W}\left(u\right)$
denote the connected components of $\mathcal{R}\cap\left(\mathbb{Z}\boxtimes\left[-\left\lfloor \alpha_{3}N\right\rfloor +1,\left\lfloor 5\alpha_{3}N\right\rfloor -1\right]\right).$\emph{
Let $S_{mid}\left(\mathcal{R}\right)$ denote the unique element of
$\mathcal{W}$ which contains $B_{3}$ as a subset. We show that with
probability close to $1,$ $\partial S_{M}\cap r_{c}=\emptyset.$}

We define $e_{T}=e_{T}\left(u\right):=\left(\mathbb{Z}\boxtimes\left\{ \left\lfloor 5\alpha_{3}N\right\rfloor +1\right\} \right)+\left\lfloor \alpha_{3}N\right\rfloor z$
and $e_{B}=e_{B}\left(u\right):=\left(\mathbb{Z}\boxtimes\left\{ -\left\lfloor \alpha_{3}N\right\rfloor -1\right\} \right)+\left\lfloor \alpha_{3}N\right\rfloor z.$
Suppose that $\partial S_{mid}\cap r_{c}\neq\emptyset,$ let $w\in\partial S_{mid}\left(\mathcal{R}\right)\cap r_{c}\cap e_{T}.$
Consider the parallelogram $\bar{B}=B\left(w;\delta_{3}N/2\right).$
Let $w_{L}$ and $w_{R}$ denote the vertices of $a_{c}$ where we
exit $\bar{B}$ the first time as we walk on $r_{c}$ starting from
$w$ towards $s_{L}$ and $s_{R}$. The part of $a_{c}$ between $w_{L}$
and $w_{R}$ cuts $\bar{B}$ into two pieces. Let $\bar{B}_{I}$ ($\bar{B}_{E}$)
denote the part which is on the right (left) hand side of $a_{c}$
when we walk from $w_{L}$ to $w_{R}.$ Let $\bar{A}_{I}=\bar{B}_{I}\setminus B\left(w;6\alpha_{3}N\right)$
and $\bar{A}_{E}=\bar{B}_{E}\setminus B\left(w;6\alpha_{3}N\right).$
By Condition \ref{cond: closed bdry point has open path} above $\bar{A}_{E}$
contains an open arm. We claim that $\bar{A}_{I}$ also contains an
open arm. Suppose the contrary. Then there must be a closed non self-touching
arc in $\bar{A}_{I}$ preventing the occurrence of the open arm. Note
that this arc is contained in $\mathcal{R}.$ Then the lowest vertex
of this arc has two disjoint $p_{\lambda_{F}}\left(N\right)$-closed
arms to $a_{c},$ and it lies lower than $x\in B:=B\left(\left\lfloor \alpha_{3}N\right\rfloor z;\alpha_{3}N\right).$
This contradicts $x\in\mathcal{L}\left(\mathcal{R},r_{c}\right)$
which was shown in the lines before Step 4. See Figure \ref{fig: no ugly boundary}.
Hence $\bar{A}_{I}$ has an open arm, which together with the open
arm of $\bar{A}_{E}$ and the two closed arms of $w$ provide a $4,3$
near critical mixed arm event. Hence the event $E_{4}^{c}=\mathcal{NA}^{c}\left(6\alpha_{3},\delta_{3}/2\right)$
occurs. Thus we arrive to the following claim and we finish Step 5.
\begin{claim}
\label{claim: S_M has no closed bdary}On the event $E_{0}\cap E_{1}\cap E_{2}\cap E_{3}\cap E_{4},$
we have $\partial S_{mid}\left(\mathcal{R}\right)\cap r_{c}=\emptyset.$
\end{claim}
\begin{figure}
\begin{centering}
\includegraphics[scale=0.8]{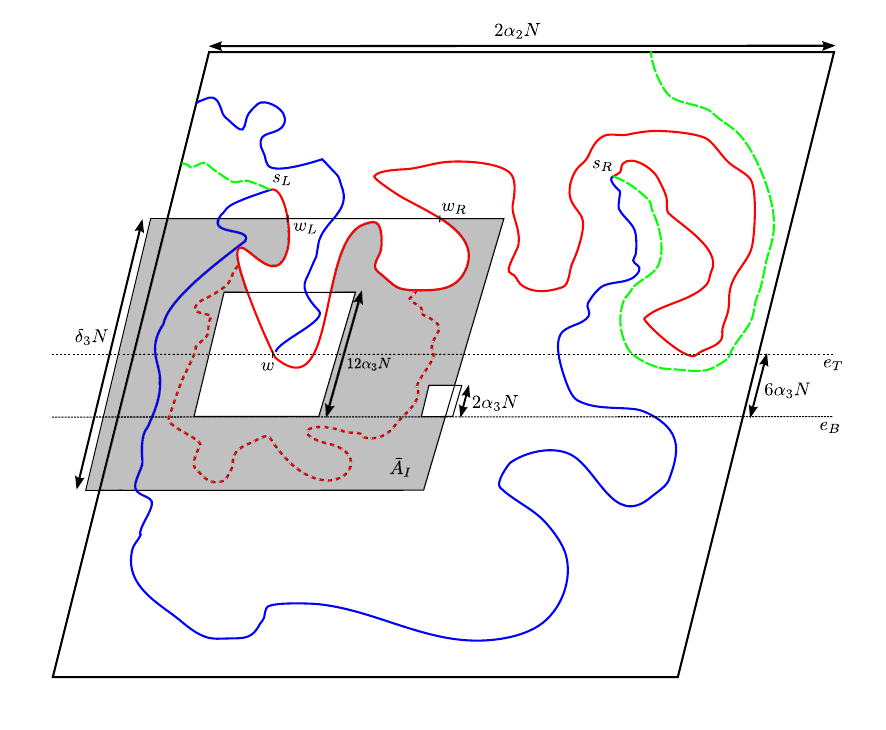}
\par\end{centering}

\caption{\label{fig: no ugly boundary} The grey area represents $\bar{A}_{I}.$
If there is no open arm in $\bar{A}_{I}$ then there is a closed arc
in $\bar{A}_{I}.$ This contradicts with $x$ being one of the lowest
vertices of $\mathcal{C}_{a}\left(\lambda\right).$}
\end{figure}

\medskip 

\textbf{Step 6.}\emph{ Recall Definition \ref{def: (a,b)-regular}.
We show that with probability close to $1,$ we can cut down some
parts of $\mathcal{R}$ and get a pair $\tilde{\mathcal{R}}$ and
$\tilde{r}_{c}$ such that the pair $\left(\tilde{\mathcal{R}}-\left\lfloor \alpha_{3}N\right\rfloor z,\tilde{r}_{c}-\left\lfloor \alpha_{3}N\right\rfloor z\right)$
is $\left(\alpha_{3}N,\alpha_{2}N\right)$-regular and 
\[
\mathcal{L}\left(\mathcal{R},r_{c}\right)\cap B=\mathcal{L}\left(\tilde{\mathcal{R}},\tilde{r}_{c}\right)\cap B.
\]
}

Let $u\in BA^{2}.$ Suppose that the event $E_{0}\cap E_{1}\cap E_{2}\cap E_{3}\cap E_{4}$
occurs. Let $\tilde{\mathcal{R}}=\tilde{\mathcal{R}}\left(u\right)$
be the connected component of $S_{mid}\left(\mathcal{R}\right)$ in
$\mathcal{R}\setminus\bigcup_{S\in\mathcal{W}:\,\partial S\cap r_{c}\neq\emptyset}cl\left(S\right)$
and $\tilde{r}_{c}=\partial\tilde{\mathcal{R}}\setminus r_{o}.$ The
conditions before Step 5 and Claim \ref{claim: S_M has no closed bdary}
gives that the pair\emph{ }$\left(\tilde{\mathcal{R}}-\left\lfloor \alpha_{3}N\right\rfloor z,\tilde{r}_{c}-\left\lfloor \alpha_{3}N\right\rfloor z\right)$
is $\left(\alpha_{3}N,\alpha_{2}N\right)$-regular.

For $R\subset\mathbb{T}$ and $r\subset\partial R$ let $\mathcal{TA}\left(R,r\right)$
denote the set of closed vertices $v\in R$ such that $v$ has two
non-touching closed arms in $R$ to $r.$ Let $M$ denote the connected
component of $S_{mid}\left(R\right)$ in $R\setminus e_{T}.$ We show
the following:
\begin{claim}
\label{claim: cut of R} Let 
\begin{equation}
E_{5}:=\mathcal{NA}\left(6\alpha_{3},\beta_{4}\right)\cup\mathcal{NA}\left(\beta_{4},\delta_{3}/2\right).\label{eq: pf active diameter - 4.1}
\end{equation}
On the event $\bigcap_{i=0}^{5}E_{i}$ $\forall u\in BA^{2},$ the
pair $\left(\tilde{\mathcal{R}}-\left\lfloor \alpha_{3}N\right\rfloor z,\tilde{r}_{c}-\left\lfloor \alpha_{3}N\right\rfloor z\right)$
is $\left(\alpha_{3}N,\alpha_{2}N\right)$-regular, and
\[
\mathcal{TA}\left(\mathcal{R},r_{c}\right)\cap M=\mathcal{TA}\left(\tilde{\mathcal{R}},\tilde{r}_{c}\right)\cap M.
\]
In particular, 
\[
\mathcal{L}\left(\mathcal{R},r_{c}\right)\cap B=\mathcal{L}\left(\tilde{\mathcal{R}},\tilde{r}_{c}\right)\cap B.
\]
\end{claim}
\begin{proof}
[Proof of Claim \ref{claim: cut of R}] From the definition of $\left(\mathcal{\tilde{R}},\tilde{r}_{c}\right)$
it follows that $\left(\mathcal{TA}\left(\mathcal{R},r_{c}\right)\cap M\right)\subset\left(\mathcal{TA}\left(\tilde{\mathcal{R}},\tilde{r}_{c}\right)\cap M\right).$
Hence it is enough to show that $\left(\mathcal{TA}\left(\tilde{\mathcal{R}},\tilde{r}_{c}\right)\backslash\mathcal{TA}\left(\mathcal{R},r_{c}\right)\right)\cap M=\emptyset.$
Suppose the contrary, that is $\exists v\in\left(\mathcal{TA}\left(\tilde{\mathcal{R}},\tilde{r}_{c}\right)\right.\!\allowbreak\backslash\!\left.\mathcal{TA}\left(\mathcal{R},r_{c}\right)\right)\cap M.$
Let $c_{v}^{1}$ and $c_{v}^{2}$ denote two non-touching closed arms
starting from $v$ and ending at $v^{1}\in\tilde{r}_{c}$ and $v^{2}\in\tilde{r}_{c}$
respectively. Since $v\in\mathcal{TA}\left(\mathcal{R},r_{c}\right)\setminus\mathcal{TA}\left(\tilde{\mathcal{R}},\tilde{r}_{c}\right),$
we can assume that $c_{v}^{1}$ cannot be extended in such a way that
it connects to $r_{c}$ and this extension is disjoint from and does
not touch $c_{v}^{2}.$ Hence $v^{1}\in\tilde{r}_{c}\setminus r_{c},$
and $v^{1}\in e_{T}.$ Let $S\in\mathcal{W}$ such that $v^{1}\in\partial S.$
Note that $\partial S\cap r_{c}\neq\emptyset.$ Let $s^{1},s^{2}$
denote the endpoints of the connected component of $v^{1}$ in $\partial S\cap e_{T}.$
At least one of $s^{1}$ and $s^{2}$ is in $r_{c}.$ Let $s^{1}\in r_{c}.$
Let $\beta_{4}\in\left(6\alpha_{3},\delta_{3}/2\right)$ be an intermediate
scale. We divide the annulus $A\left(v^{1};6\alpha_{3}N,\delta_{3}N/2\right)$
into the annuli
\begin{align*}
A_{4,0} & =A\left(v^{1};6\alpha_{3}N,\beta_{4}N\right),\\
A_{4,1} & =A\left(v^{1};\beta_{4}N,\delta_{3}N/2\right).
\end{align*}
We have two cases. If \emph{$c_{v}^{2}\cap A_{4,0}\neq\emptyset,$
}then we see $4$ half plane arms in $A_{4,1}:$ $a_{c}$ provides
two closed arms, and each of $c_{v}^{1}$ and $c_{v}^{2}$ gives one
closed arm. Hence the event $\mathcal{NA}^{c}\left(6\alpha_{3},\beta_{4}\right)$
occurs. If $c_{v}^{2}\cap A_{4,0}=\emptyset,$ we have $4$ half plane
arms in $A_{4,0}:$ $a_{c}$ provides two closed arms, $c_{v}^{1}$
another closed arm, moreover, we get an open arm which separates $c_{v}^{1}$
from $a_{c}.$ See Figure \ref{fig: cut} for more details. Hence
the event $\mathcal{NA}^{c}\left(\beta_{4},\delta_{3}/2\right)$ occurs.
By (\ref{eq: pf active diameter - 4.1}) this finishes the proof of
Claim \ref{claim: cut of R}.
\end{proof}
\begin{figure}
\centering{}\includegraphics[scale=0.8]{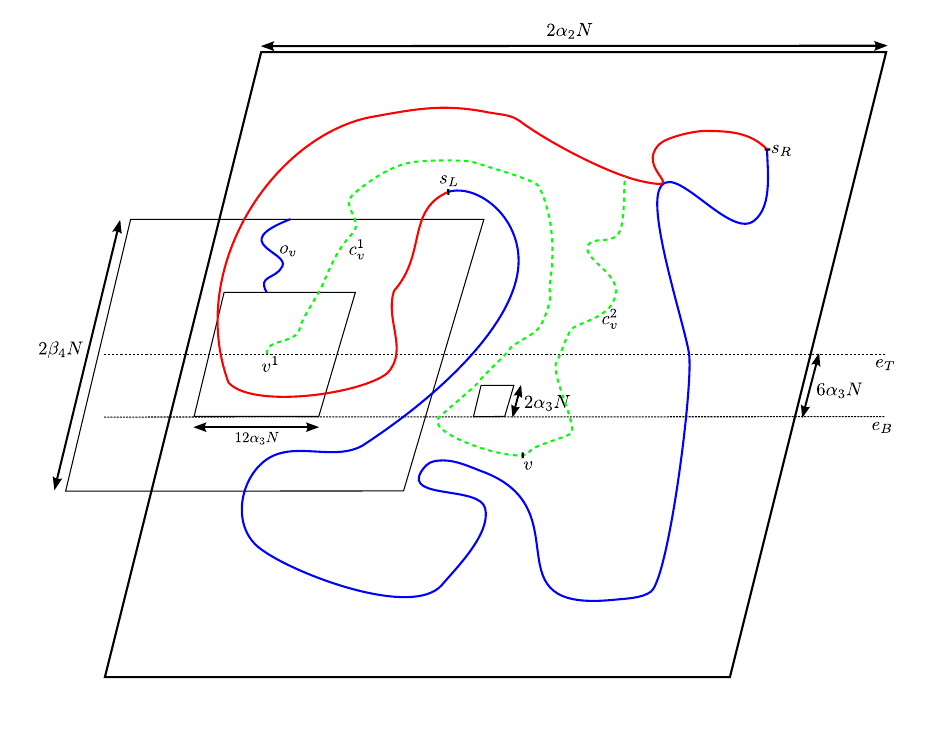}\caption{\label{fig: cut} If $c_{v}^{2}\cap A_{4,0}=\emptyset,$ we see $4$
half plane arms in $A_{4,0}:$ the two closed induced by $a_{c},$
a closed arm $c_{v}^{1},$ and an open arm $o_{v}$ which separates
$c_{v}^{1}$ from $a_{c}.$}
\end{figure}

By Corollary \ref{cor: no many arms} we set $\alpha_{3}$ such that
\begin{equation}
\mathbb{P}\left(E_{4}\cap E_{5}\right)\geq1-\varepsilon/20\label{eq: pf active dimeter - 5}
\end{equation}
for $N\geq N_{5}\left(\varepsilon,\alpha_{3},\lambda_{0},\lambda,\alpha,K\right).$
Let
\[
E=E_{0}\cap E_{1}\cap E_{2}\cap E_{3}\cap E_{4}\cap E_{5}.
\]
The combination of the lines in the beginning of Step 1, (\ref{eq: pf active dimeter - 2}),
(\ref{eq: pf active dimeter - 3}), (\ref{eq: pf active dimeter - 4})
and (\ref{eq: pf active dimeter - 5}) gives that 
\begin{equation}
\mathbb{P}\left(E\right)\geq1-\varepsilon/4\label{eq: pf acitve diameter - 6}
\end{equation}
for $N\geq\bigvee_{i=0}^{5}N_{i}.$ This finishes Step 6. 

\medskip

\textbf{Step 7.} \emph{We set $\theta>0$ such that $\mathbb{P}_{N}\left(BA^{2}\neq\emptyset\right)<\varepsilon/2$
for large $N,$ and conclude the proof of Proposition \ref{prop: active diameter}.}

For $v\in V,$ let 
\[
Z\left(v\right):=\left\{ \exists u\in BA^{2}\mbox{ such that }z\left(u\right)=v\right\} .
\]
Hence
\[
\left\{ BA{}^{2}\neq\emptyset\right\} =\bigcup_{v\in B\left(\left\lceil \frac{2\alpha+K+2}{\alpha_{3}}\right\rceil \right)}Z\left(v\right)
\]
and 
\begin{equation}
\mathbb{P}_{N}\left(BA{}^{2}\neq\emptyset,E\right)\leq\sum_{v\in B\left(\left\lceil \frac{2\alpha+K+2}{\alpha_{3}}\right\rceil \right)}\mathbb{P}_{N}\left(Z\left(v\right)\cap E\right)\label{eq: pf active diameter - 7}
\end{equation}
Note that on the event $Z\left(v\right)\cap E,$ Claim \ref{claim: unique F}
and the arguments above give that $\mathcal{C}_{a}\left(u;\lambda\right),$
$F\left(u\right),$ $\lambda_{F}\left(u\right),$ $\mathcal{R}\left(u\right),$
$r_{c}\left(u\right),$ $\tilde{\mathcal{R}}\left(u\right)$ and $\tilde{r}_{c}\left(u\right)$
do not depend on the choice of $u\in BA^{2}$ as long as $z\left(u\right)=v.$
Except for $\mathcal{C}_{a}\left(u,\lambda\right),$ we omit the argument
$u$ from the notation above. 

We set $k:=\left\lfloor 1/2\theta\right\rfloor .$ Recall that $d\left(x,y\right)=d\left(\tilde{x},\tilde{y}\right)+\sqrt{3}=\diam\left(\mathcal{C}_{a}\left(u;\lambda\right)\right)+\sqrt{3},$
and $\diam\left(\mathcal{C}_{a}\left(u;\lambda\right)\right)\in\left(\left(\alpha-\theta\right)N,\left(\alpha+\theta\right)N\right).$
On the event $Z\left(v\right)$ there is a unique $l=l\left(y\right)\in\left[0,k-1\right]\cap\mathbb{Z}$
such that $x\in B_{l,k}$ where 
\[
B_{l,k}=B_{l,k}\left(v\right):=\left[-\alpha_{3}N,\alpha_{3}N\right]\boxtimes\left(\left(2\frac{l}{k}-1\right)\alpha_{3}N,\left(2\frac{l+1}{k}-1\right)\alpha_{3}N\right]+\left\lfloor \alpha_{3}N\right\rfloor v.
\]

Recall from the lines above Step 5 we have $x\in\mathcal{L}\left(\mathcal{R},r_{c}\right).$
From Claim \ref{claim: cut of R} we have $\mathcal{L}\left(\mathcal{R},r_{c}\right)\cap B=\mathcal{L}\left(\tilde{\mathcal{R}},\tilde{r}_{c}\right)\cap B$
where $B=B\left(\left\lfloor \alpha_{3}N\right\rfloor v;\alpha_{3}N\right).$
Hence on the event $Z\left(v\right)\cap E,$ we have $\mathcal{L}\left(\tilde{\mathcal{R}},\tilde{r}_{c}\right)\cap B_{l,k}\neq\emptyset.$
Let $\left(R,r\right)$ be a fixed pair. Hence 
\begin{align}
\mathbb{P}_{N}\left(Z\left(v\right),E,\left(\mathcal{R},r_{c}\right)=\left(R,r\right)\right) & =\mathbb{P}_{N}\left(Z\left(v\right),E,\left(\mathcal{R},r_{c}\right)=\left(R,r\right),\,\mathcal{L}\left(\tilde{R},\tilde{r}\right)\cap B_{l,k}\neq\emptyset\mbox{ at time }p_{\lambda_{F}}\left(N\right)\right)\label{eq: pf active diameter - 8}
\end{align}
where $\left(\tilde{R},\tilde{r}\right)$ denotes the pair we get
when we cut down some parts of $R$ as in Step 6.

Recall Definition \ref{def: filtration of tau}. Lemma \ref{lem: measurable w.r.t tau-s}
gives that the $N$-parameter frozen percolation process is adapted
to the filtration $\left(\mathcal{F}_{t}\left(V\right)\right)_{t\in\left[0,1\right]}.$
Hence for all $u\in BA^{2},$ $l$ and $\lambda_{F}$ are $\mathcal{F}_{p_{\lambda}\left(N\right)}\left(V\right)$
-measurable functions, and $\left\{ \left(\mathcal{R},r_{c}\right)=\left(R,r\right)\right\} \in\mathcal{F}_{p_{\lambda}\left(N\right)}\left(V\right).$
By Claim \ref{claim: conf in R is indep from rest} we have that on
the event $Z\left(v\right)\cap E\cap\left\{ \left(\mathcal{R},r_{c}\right)=\left(R,r\right)\right\} $
the $\tau$-values in $R$ do not influence the frozen percolation
process in $V\setminus R$ up to time $p_{\lambda}\left(N\right).$
This combined with Claim \ref{claim: OC measurable} gives that there
is a function $f$ such that $f\left(R,\bar{l},\bar{\lambda}_{F}\right)$
is $\mathcal{F}_{p_{\lambda}\left(N\right)}\left(V\setminus R\right)$-measurable
for all $R,\bar{l},\bar{\lambda_{F}}.$ Moreover, it satisfies 
\begin{align}
\mathbf{1}\left\{ Z\left(v\right),E,\left(\mathcal{R},r_{c}\right)=\left(R,r\right),l=\bar{l},\lambda_{F}\in d\bar{\lambda}_{F}\right\} = & f\left(R,\bar{l},\bar{\lambda}_{F}\right)\mathbf{1}\left\{ Z\left(v\right),E\right\} ,\label{eq: pf active diameter - 8.1}
\end{align}
for $\bar{l}\in\left[0,k-1\right]\cap\mathbb{Z}$ and Lebesgue almost
every $\bar{\lambda}_{F}\in\left[0,1\right].$

Hence

\begin{multline*}
\mathbb{P}_{N}\left(\left.\begin{array}{c}
Z\left(v\right),E,l=\bar{l},\lambda_{F}\in d\bar{\lambda}_{F}\\
\mathcal{L}\left(\tilde{R},\tilde{r}\right)\cap B_{\bar{l},k}\neq\emptyset\mbox{ at time }p_{\bar{\lambda}_{F}}\left(N\right)
\end{array}\right|\mathcal{F}_{p_{\lambda}\left(N\right)}\left(V\setminus R\right)\right)\\
=f\left(R,\bar{l},\bar{\lambda}_{F}\right)\mathbb{P}_{N}\left(\left.\begin{array}{c}
Z\left(v\right),E\\
\mathcal{L}\left(\tilde{R},\tilde{r}\right)\cap B_{\bar{l},k}\neq\emptyset\mbox{ at time }p_{\bar{\lambda}_{F}}\left(N\right)
\end{array}\right|\mathcal{F}_{p_{\lambda}\left(N\right)}\left(V\setminus R\right)\right)
\end{multline*}
for $\bar{l}\in\left[0,k-1\right]\cap\mathbb{Z}$ and Lebesgue almost
every $\bar{\lambda}_{F}\in\left[0,1\right].$

From Step 6, we have that $\tilde{R}\subseteq R.$ Claim \ref{claim: cut of R}
shows that we can apply Corollary \ref{cor: lowest of lowest regular regions}
in the following. We have 

\begin{align}
\mathbb{P}_{N}\left(Z\left(v\right),E,\mathcal{L}\left(\tilde{R},\tilde{r}\right)\cap B_{\bar{l},k}\neq\emptyset\right. & \!\!\left.\left.\mbox{ at time }p_{\bar{\lambda}_{F}}\left(N\right)\right|\mathcal{F}_{p_{\lambda}\left(N\right)}\left(V\setminus R\right)\right)\nonumber \\
 & \leq\mathbb{P}_{N}\left(\left.\mathcal{L}\left(\tilde{R},\tilde{r}\right)\cap B_{\bar{l},k}\neq\emptyset\mbox{ at time }p_{\bar{\lambda}_{F}}\left(N\right)\right|\mathcal{F}_{p_{\lambda}\left(N\right)}\left(V\setminus R\right)\right)\nonumber \\
 & =\mathbb{P}_{p_{\bar{\lambda}_{F}}\left(N\right)}\left(\mathcal{L}\left(\tilde{R},\tilde{r}\right)\cap B_{\bar{l},k}\neq\emptyset\right)\nonumber \\
 & \leq c_{1}k^{-1}\label{eq: pf active diameter - 9}
\end{align}
for $N\geq N_{6}\left(\lambda_{0},\lambda,\alpha_{3},\alpha_{2},k\right)$
with $c_{1}=c_{1}\left(\lambda_{0},\lambda,\alpha_{3},\alpha_{2}\right)$
of Corollary \ref{cor: lowest of lowest regular regions}. A combination
of (\ref{eq: pf active diameter - 9}) and (\ref{eq: pf active diameter - 8.1})
gives that

\begin{align*}
\mathbb{P}_{N}\left(Z\left(v\right),E,\right. & \!\!\left.\left(\mathcal{R},r_{c}\right)=\left(R,r\right),l=\bar{l},\lambda_{F}=\bar{\lambda}_{F},\mathcal{L}\left(\tilde{R},\tilde{r}\right)\cap B_{l,k}\neq\emptyset\mbox{ at time }p_{\lambda_{F}}\left(N\right)\left|\mathcal{F}_{p_{\lambda}\left(N\right)}\left(V\setminus R\right)\right.\right)\\
 & \leq c_{1}k^{-1}f\left(R,\bar{l},\bar{\lambda}_{F}\right)
\end{align*}
for $N\geq N_{6}.$ Hence
\begin{align}
\mathbb{P}_{N}\left(Z\left(v\right)\cap E\right) & \leq c_{1}k^{-1}.\label{eq: pf active diameter - 9.2}
\end{align}
for $N\geq N_{6}.$ 

(\ref{eq: pf active diameter - 9.2}) combined with (\ref{eq: pf active diameter - 7})
gives that

\begin{align}
\mathbb{P}_{N}\left(BA{}^{2}\neq\emptyset,E\right) & \leq\sum_{v\in B\left(\left\lceil \frac{2\alpha+K+2}{\alpha_{3}}\right\rceil \right)}\mathbb{P}_{N}\left(Z\left(v\right)\cap E\right)\nonumber \\
 & \leq c_{2}k^{-1}\label{eq: pf active diameter - 10}
\end{align}
with $c_{2}=c_{2}\left(\lambda_{0},\lambda,\alpha_{3},\alpha_{2},K\right)$
for $N\geq N_{6}.$ We set $\theta$ such that $k=\left\lfloor 1/2\theta\right\rfloor >4c_{2}/\varepsilon.$
A combination of (\ref{eq: pf active diameter - 10}) and (\ref{eq: pf acitve diameter - 6})
gives that 

\begin{align}
\mathbb{P}_{N}\left(BA^{2}\neq\emptyset\right) & \leq\mathbb{P}_{N}\left(BA^{2}\neq\emptyset,\, E\right)+\mathbb{P}_{N}\left(E^{c}\right)\nonumber \\
 & \leq c_{2}k^{-1}+\varepsilon/4<\varepsilon/2\label{eq: pf active diameter - 11}
\end{align}
for $N\geq N'=\bigvee_{i=0}^{6}N_{i}.$

A proof analogous to that of (\ref{eq: pf active diameter - 11})
gives that there is $N''=N''\left(\alpha,\lambda,K\right)$ 
\begin{equation}
\mathbb{P}_{N}\left(BA^{1}\neq\emptyset\right)<\varepsilon/2\label{eq: pf active diameter - 12}
\end{equation}
for $N\geq N''.$ A combination of (\ref{eq: pf active diameter - 1.5}),
(\ref{eq: pf active diameter - 11}) and (\ref{eq: pf active diameter - 12})
finishes the proof of Proposition \ref{prop: active diameter}.\appendix
\end{proof}

\section{\label{sec: appendix}Appendix}

\subsection{\label{sub: winding}Winding number of arms}

Here we prove Proposition \ref{prop: mixed arm exp}. The proof is
motivated by \cite{Beffara2011}. There, among many other things,
it was shown that when there are $k$ disjoint open arms in $A\left(M,aM\right)$
($a>1$), then, with conditional probability at least $1-a^{-\varepsilon},$
and uniformly in $M,$ are also $k$ disjoint open arms which wind
around the origin at least $c\log a$ times where $c,\varepsilon$
are positive constants.

We prove a slightly different result, namely that if we have $k$
disjoint arms with any colour sequence $\sigma\in\left\{ o,c\right\} ^{k}$
in $A\left(M,aM\right),$ than with conditional probability at least
$1-a^{-\varepsilon},$ these arms wind around the origin at in at
least $c\log a$ disjoint subannuli of $A\left(a,b\right)$ for some
$c,\varepsilon>0.$ Following \cite{Nolin2008}, we recall the notion
of well separated arms. We modify Definition 7 of \cite{Nolin2008}
for annuli:
\begin{defn}
\label{def: well sep outside}Consider some annulus $A=A\left(v;M,\tau M\right)$
and a parallelogram $B=B\left(v;\tau M\right)$ for $M\in\mathbb{N},$
$\tau\in\left(1,\infty\right)$ and $v\in V.$ Let $s_{T},s_{B},s_{L},s_{R}$
denote the top, bottom, left and right sides of $B.$ Let $\mathcal{C}=\left\{ c_{i}\right\} _{1\leq i\leq j}$
be a set of $j$ disjoint arms in $A$ such that for each $i,$ all
of the vertices of $c_{i}$ are open or all of them are closed. Let
$z_{i}$ be the endpoint of $c_{i}$ on $\partial B\left(v;\tau M\right).$
Let $\eta\in\left(0,1\right],$ we attach a parallelogram $r_{i}$
to $z_{i}$ as follows: 
\[
r_{i}=\begin{cases}
z_{i}+\left[-\eta M,\eta M\right]\boxtimes\left[0,2\sqrt{\eta}M\right] & \mbox{if }z_{i}\in s_{T}\\
z_{i}+\left[-\eta M,\eta M\right]\boxtimes\left[0,-2\sqrt{\eta}M\right] & \mbox{if }z_{i}\in s_{B}\\
z_{i}+\left[-2\sqrt{\eta}M,0\right]\boxtimes\left[-\eta M,\eta M\right] & \mbox{if }z_{i}\in s_{L}\\
z_{i}+\left[0,2\sqrt{\eta}M\right]\boxtimes\left[-\eta M,\eta M\right] & \mbox{if }z_{i}\in s_{R}.
\end{cases}
\]

We say that $\mathcal{C}$ is $\eta$-well-separated on the outside,
if the two following conditions are satisfied:
\begin{enumerate}
\item The extremities $z_{i}$ $i=1,2,\ldots,j$ are neither too close to
each other: 
\[
\forall i\neq l,\, d\left(z_{i},z_{l}\right)\geq10\sqrt{\eta}M,
\]
nor too close to the corners $Z_{l}$ $l=1,2,3,4$ of $B:$
\[
\forall i,j,\, d\left(z_{i},Z_{l}\right)\geq10\sqrt{\eta}M.
\]

\item Each $r_{i}$ is crossed vertically when $z_{i}\in s_{T}\cup s_{B},$
and horizontally when $z_{i}\in s_{L}\cup s_{R}$ by some crossing
$\tilde{c}_{i}$ of the same colour as $c_{i},$ and 
\[
c_{i}\mbox{ is connected to }\tilde{c}_{i}\mbox{ in }z_{i}+A\left(1,\sqrt{\eta}M\right).
\]
 
\end{enumerate}
\end{defn}
We say that a set $\mathcal{C}=\left\{ c_{i}\right\} _{1\leq i\leq j}$
of disjoint arms in $A$ \emph{can be made $\eta$-well-separated}
on the outside, if there exists an set $\mathcal{C}'=\left\{ c_{i}'\right\} _{1\leq i\leq j}$
of disjoint arms in $A$ which is $\eta$-well-separated on the outside,
and $c_{i}'$ has the same colour and endpoint on $\partial B\left(v;M\right)$
as $c_{i}$ for $i=1,2,\ldots,j.$

Similarly to Definition \ref{def: well sep outside}, we define the
$\eta$-well-separation on the inside. The following statement follows
from Lemma 15 of \cite{Nolin2008}.
\begin{lem}
\label{lem: well sep}For $\tau\in\left(1,\infty\right),$ and $\delta>0,$
there exists $\eta\left(\delta\right)>0$ such that for any positive
integer $N,$ we have 
\[
\mathbb{P}_{1/2}\left(\mbox{any set of disjoint arms in \ensuremath{A\left(N,\tau N\right)}can be made \ensuremath{\eta}-well-separated on the outside}\right)\geq1-\delta.
\]
Moreover, the same statement holds for well separated arms on the
inside.
\end{lem}
We prove the following proposition.
\begin{prop}
\label{prop: arms do wind}Let $k,N\in\mathbb{N},$ $a\in\left(10,\infty\right),$
and $\sigma$ a colour sequence of length $k.$ We divide the annulus
$A\left(N,aN\right)$ into the annuli $A_{i}=A\left(2^{i}N,2^{i+1}N\right)$
for $i=0,1,\ldots,\left\lfloor \log_{2}\left(a\right)\right\rfloor -1.$
Let $W$ denote the set of indices $i$ for which all the arms in
$A_{3i+1}$ wind around the origin at least once in the counter-clockwise
direction for $i=0,1,\ldots,\left\lfloor \log_{2}\left(a\right)/3\right\rfloor -1.$
There are positive constants $c=c\left(k\right),$ $\varepsilon=\varepsilon\left(k\right)$
and $N_{0}=N_{0}\left(k\right)$ such that 
\[
\mathbb{P}_{1/2}\left(\mathcal{A}_{k,\sigma}\left(N,aN\right),\,\left|W\right|\geq c\log_{2}a\right)\geq\left(1-a^{-\varepsilon}\right)\pi_{k,\sigma}\left(N,aN\right)
\]
for all $a\in\left(1,\infty\right)$ and $N\geq N_{0}.$ \end{prop}
\begin{rem}
Proposition \ref{prop: mixed arm exp} follows from Proposition \ref{prop: arms do wind},
since $W=\emptyset$ on the event $\mathcal{A}_{k,l,\sigma}\left(N,aN\right)$
when $l\geq1.$\end{rem}
\begin{proof}
[Proof of Proposition \ref{prop: arms do wind}] For $a\leq2,$
the statement is trivial. Hence in the rest of the proof we suppose
that $a>2.$ Classical RSW techniques \cite{Grimmett1999} give that
for all $k\in\mathbb{N}$ there is $\varepsilon_{1}=\varepsilon_{1}\left(k\right)>0$
such that 
\begin{equation}
\pi_{k,\sigma}\left(N,aN\right)\geq a^{-\varepsilon_{1}}\label{eq: pf prop arms do wind - 1}
\end{equation}
uniformly in $a\geq2,$ $N\geq1$ and $\sigma\in\left\{ o,c\right\} ^{k}.$

Let $\eta\in\left(0,1/10\right).$ Let $IS_{i}$ ($OS_{i}$) denote
the event that any set of disjoint arms of $A_{i}$ can be made $\eta$-well-separated
on the inside (outside). Let $WS$ denote the set of indices $i\in\left\{ 0,1,\ldots,\left\lfloor \frac{\log_{2}a}{3}\right\rfloor -1\right\} $
for which $OS_{3i}$ and $IS_{3i+2}$ both hold. Notice that the events
$\left\{ i\in WS\right\} $ for $i=1,2,\ldots,\left\lfloor \frac{\log_{2}a}{3}\right\rfloor -1$
are independent. Moreover, by Lemma \ref{lem: well sep}, for any
$\delta>0$ there is $\eta\left(\delta\right)\in\left(0,1/10\right)$
such that
\[
\mathbb{P}_{1/2}\left(i\in WS\right)\geq1-\delta.
\]
Combining this with Hoeffding's inequality we get that $c_{0},\delta,\eta$
such that
\begin{equation}
\mathbb{P}_{1/2}\left(\left|WS\right|\leq c_{0}\log a\right)\leq a^{-2\varepsilon_{1}}.\label{eq: pf prop arms do wind - 2}
\end{equation}

This and (\ref{eq: pf prop arms do wind - 1}) gives that 
\begin{align}
\mathbb{P}_{1/2}\left(\mathcal{A}_{k,\sigma}\left(N,aN\right)\cap\left\{ \left|WS\right|>c_{0}\log\left(a\right)\right\} \right) & \geq\pi_{k,\sigma}\left(N,aN\right)-\mathbb{P}_{1/2}\left(\left|WS\right|\leq c_{0}\log a\right)\nonumber \\
 & \geq\pi_{k,\sigma}\left(N,aN\right)-a^{-2\varepsilon_{1}}\nonumber \\
 & \geq\left(1-a^{-\varepsilon_{1}}\right)\pi_{k,\sigma}\left(N,aN\right)\label{eq: pf prop arms do wind - 3}
\end{align}
 for all $N.$

Let us fix an integer $i\in\left\{ 0,1,\ldots,\left\lfloor \frac{\log_{2}a}{3}\right\rfloor -1\right\} .$
Condition on the event $\mathcal{A}_{k,\sigma}\left(N,2^{3i+1}N\right)\cap\mathcal{A}_{k,\sigma}\left(2^{3i+2}N,aN\right)\cap\left\{ i\in WS\right\} $
and on the configuration in $A\left(N,aN\right)\setminus A_{3i+1}.$
This conditioning gives that all the arms in $A_{3i}$ can be made
$\eta$-well-separated on the outside, and all the arms in $A_{3i+2}$
can be made $\eta$-well-separated on the inside. This imposes some
conditions on the configuration in $A_{3i+1}:$ there is a finite
collection of disjoint parallelograms in which certain crossing events
have to be satisfied. In order to have $k$ arms with colour sequence
$\sigma$ in $A\left(N,aN\right),$ it is enough to connect, with
the right colour, the $k$-tuple of parallelograms corresponding to
the well separated versions of these arms on the inner parallelogram
to those on the outer parallelogram of $A_{3i+1}.$ There might be
more than one choice for this pair of $k$-tuples of parallelograms.
In this case we choose a pair in some deterministic way.

We connect the corresponding pairs of parallelograms by disjoint tubes
of width $\sqrt{\eta}2^{3i+1}N$ in $A_{3i+1}$ as in the proof of
Lemma 4 of \cite{Kesten1987} (see Figure 9 of \cite{Kesten1987}),
with the difference that these connections are special: We chose these
tubes such that each of them winds around the origin at least twice
in the counter-clockwise direction. We add an additional tube which
avoids the ones above, connects the boundaries of the inner and the
outer parallelograms of $A_{3i+1}$ and winds around the origin at
least twice in the counter-clockwise direction.

With standard RSW techniques one can show that the probability of
the event that the original tubes are crossed in the hard direction
by a path with the appropriate colour, and the additional tube is
crossed in the hard direction with an open and a closed path is at
least $h>0.$ Here $h=h\left(k,\eta\right)$ is independent of $i,N$
and the location of the parallelograms we connected. The open and
closed crossings of the additional tube forces all the arms of $A\left(N,aN\right)$
to wind around the origin in $A_{3i+1}$ at least once in the counter-clockwise
direction. Hence the event $\left\{ i\in W\right\} $ occurs.

Thus the probability of $\left\{ i\in W\right\} $ conditioned on
the event $\mathcal{A}_{k,\sigma}\cap\left\{ i\in WS\right\} $ and
on the configuration in $A\left(N,aN\right)\setminus A_{3i+1}$ is
at least $h.$ Note that the event $\left\{ i\in W\right\} $ only
depends on the configuration in $A_{3i+1}.$ Hence, when we condition
on the event $\mathcal{A}_{k,\sigma}\left(N,aN\right)$ and on the
realization of $WS,$ the set $W$ stochastically dominates a set
$Z,$ where the elements of $Z$ are sampled from $WS$ independently
from each other with probability $h.$

Hence for $c>0$ we have
\begin{align}
\mathbb{P}_{1/2} & \left(\left|W\right|\geq c\log_{2}a\left|\mathcal{A}_{k,\sigma}\left(N,aN\right)\right.\right)\geq\mathbb{P}_{1/2}\left(\left|W\right|\geq c\log_{2}a,\left|WS\right|\geq c_{0}\log_{2}a\left|\mathcal{A}_{k,\sigma}\left(N,aN\right)\right.\right)\nonumber \\
 & =\sum_{S}\mathbb{P}_{1/2}\left(\left|W\right|\geq c\log_{2}a\left|\mathcal{A}_{k,\sigma}\left(N,aN\right),\, WS=S\right.\right)\mathbb{P}_{1/2}\left(WS=S\left|\mathcal{A}_{k,\sigma}\left(N,aN\right)\right.\right)\nonumber \\
 & \geq\sum_{S}\mathbb{P}_{1/2}\left(\left|Z\right|\geq c\log_{2}a\left|\mathcal{A}_{k,\sigma}\left(N,aN\right),\, WS=S\right.\right)\mathbb{P}_{1/2}\left(WS=S\left|\mathcal{A}_{k,\sigma}\left(N,aN\right)\right.\right),\label{eq: pf prop arms do wind - 4}
\end{align}
where the summation over $S\subseteq\left\{ 0,1,\ldots\left\lfloor \frac{\log_{2}a}{3}\right\rfloor -1\right\} $
with $\left|S\right|\geq c_{0}\log_{2}a$. We split this sum in (\ref{eq: pf prop arms do wind - 4})
depending on the number of elements of $S,$ and we get 

\begin{align}
\mathbb{P}_{1/2} & \left(\left|W\right|\geq c\log_{2}a\left|\mathcal{A}_{k,\sigma}\left(N,aN\right)\right.\right)\nonumber \\
 & \geq\mathbb{P}\left(Y\geq c\log_{2}a\right)\sum_{l\geq c_{0}\log_{2}a}\mathbb{P}_{1/2}\left(\left|WS\right|=l\left|\mathcal{A}_{k,\sigma}\left(N,aN\right)\right.\right)\nonumber \\
 & =\mathbb{P}\left(Y\geq c\log_{2}a\right)\mathbb{P}_{1/2}\left(\left|WS\right|\geq c_{0}\log_{2}a\left|\mathcal{A}_{k,\sigma}\left(N,aN\right)\right.\right),\label{eq: pf prop arms do wind - 5}
\end{align}
where $Y$ is a random variable with distribution $Binom\left(c_{0}\log_{2}a,h\right).$
Using Hoeffding's inequality, we set $c=c\left(h\right),\varepsilon_{2}\left(h\right)>0$
such that 
\begin{equation}
\mathbb{P}\left(Y\geq c\log_{2}a\right)\geq1-a^{-\varepsilon_{2}}.\label{eq: pf prop arms do wind - 6}
\end{equation}
By substituting (\ref{eq: pf prop arms do wind - 6}) and (\ref{eq: pf prop arms do wind - 3})
to (\ref{eq: pf prop arms do wind - 5}) we get that
\[
\mathbb{P}_{1/2}\left(\left|W\right|\geq c\log_{2}a\left|\mathcal{A}_{k,\sigma}\left(N,aN\right)\right.\right)\geq\left(1-a^{-\varepsilon_{1}}\right)\left(1-a^{-\varepsilon_{2}}\right)
\]
for all $a>2$ and $N,$ which finishes the proof of Proposition \ref{prop: arms do wind}.
\end{proof}
With suitable adjustments of arguments above, one can show that the
following generalization of Proposition \ref{prop: mixed arm exp}
holds.
\begin{prop}
\label{prop: gen mixed arm exp}For any $k\in\mathbb{N},$ there are
positive constants $c=c\left(k\right),\,\varepsilon=\varepsilon\left(k\right)$
such that for all $l,l'\in\mathbb{N}$ with $0\leq l\leq l'\leq k$
\[
\pi_{k,l,\sigma}\left(n_{0}\left(k\right),N\right)\leq cN^{-\varepsilon}\pi_{k,l',\sigma}\left(n_{0}\left(k\right),N\right)
\]
uniformly in $N$ and in the colour sequence $\sigma.$
\end{prop}

\subsection{\label{sub: there are thick paths}Existence of long thick paths
in nice regions}

Recall the Definition \ref{def: gridpath} and \ref{def: (a,b)-nice}.
First we prove Lemma \ref{lem: det gridpath} which is the special
case of Lemma \ref{lem: local det gridpath} where $C$ is $\left(a,b\right)$-nice.
Then we show how to modify the proof of Lemma \ref{lem: det gridpath}
to deduce Lemma \ref{lem: local det gridpath}.
\begin{namedthm}
[Lemma \ref{lem: det gridpath}] Let $a,b\in\mathbb{N}$ with $a\geq2000.$
Let $C$ be an $\left(a,b\right)$-nice subgraph of $\mathbb{T}.$
Then there is a $\left\lfloor a/200-10\right\rfloor $-gridpath contained
in $C$ with diameter at least $\diam\left(C\right)-2b-2a-12.$\end{namedthm}
\begin{rem}
We believe that the constants in Lemma \ref{lem: det gridpath} are
not optimal.\end{rem}
\begin{proof}
[Proof of Lemma \ref{lem: det gridpath}] Recall the lines below
Definition \ref{def: (a,b)-nice}. To prove Lemma \ref{lem: det gridpath},
it is enough to find a path $\zeta$ in $C$ such that $\diam\left(\zeta\right)\geq d-2b-2a-12$
and $\zeta+B\left(a/100-5\right)\subset C.$ We construct $\zeta$
by the following strategy.

We put hexagons on the vertices of $\mathbb{T}$ in the `usual' way:
The hexagon corresponding to the vertex $v$ is the regular hexagon
with side length $1/\sqrt{3}$ centred around $v$ with one of its
sides is vertical. These hexagons give a tiling of the plane $\mathbb{R}^{2}.$
Using this tiling, we look at $C$ as the region in $\mathbb{R}^{2}$
which is the union of the hexagons which are centred around the vertices
of $C.$ 

Let $x,y\in C$ such that $d\left(x,y\right)=\diam\left(C\right).$
Let $\gamma\subset\mathbb{R}^{2}$ be a shortest curve connecting
$x$ and $y$ in the region $C,$ that is, $\gamma$ is a continuous
map of $\left[0,1\right]$ such that $0$ is mapped to $x$ and $1$
is mapped to $y.$ We get the path $\zeta$ from $\gamma$ as follows.
First we cut down two pieces of $\gamma$ one from its beginning and
one from its end. We call the resulting path $\gamma^{2}.$ Then we
walk along $\gamma^{2},$ and if there is a point of $\partial C$
`close by' on the left (right) of $\gamma^{2},$ then we make a `small'
detour to the right (left). We get the path $\zeta$ from $\gamma^{2}$
after these detours. We show that $\zeta$ indeed satisfies the conditions
above, and finish the proof of Lemma \ref{lem: det gridpath}.

\medskip

We gave a strategy which involved continuous curves and regions in
the plane $\mathbb{R}^{2}.$ We adapt it to the triangular lattice
in the following precise proof.

Let $x=\left(x_{1},x_{2}\right),y=\left(y_{1},y_{2}\right)\in C$
such that $d\left(x,y\right)=\diam\left(C\right).$ We further assume
that $x_{1}<y_{1}$ and $d\left(x,y\right)=y_{1}-x_{1}.$ The other
case where $d\left(x,y\right)=y_{2}-x_{2}$ can be treated similarly.
Let $\tilde{\gamma}$ denote a shortest (having the least number of
vertices) path which starts at $x,$ ends at $y,$ and it is contained
in $C.$

Note that there are $\binom{2n}{n}$ shortest paths between the vertices
$0$ and $n\underline{e}_{1}+n\underline{e}_{2}$ in $\mathbb{T}.$
However, most of them do not follow closely the straight line between
the points $0$ and $n\underline{e}_{1}+n\underline{e}_{2}.$ Hence
$\tilde{\gamma}$ usually does not resemble a shortest continuous
curve connecting $x$ and $y.$

\textbf{Step 1.} \emph{We choose a specific shortest path between
$x$ and $y.$}

For $u,v\in\mathbb{T},$ let $s\left(u,v\right)$ denote the line
segment connecting $u$ and $v$ in $\mathbb{R}^{2}.$ This segment
naturally induces an oriented path $\sigma\left(u,v\right)$ in $\mathbb{T}$
as a sequence of the midpoints of the hexagons which are intersected
by $s\left(u,v\right)$ as we walk along it from $u$ to $v.$ Note
that it can happen that the segment $s\left(u,v\right)$ contains
a side of a hexagon. In this case, we put only one of the neighbouring
hexagons to $\sigma\left(u,v\right).$ We say that $\sigma\left(u,v\right)$
is a triangular grid approximation of the segment $s\left(u,v\right).$
Note that $\sigma\left(u,v\right)$ is a shortest path between $u$
and $v$ in $\mathbb{T}.$

Recall the notation in Section \ref{sub: notation}. Let $v,u,u'\in\tilde{\gamma}$
with $v\prec u,u'$ and $u\sim u'.$ Then for all $w\in\sigma\left(v,u\right)$
there is $w'\in\sigma\left(v,u'\right)$ with $w\sim w'.$ Hence for
$v\in\tilde{\gamma}$ there are two cases: 
\begin{itemize}
\item either $\forall u\in\tilde{\gamma}_{v,y}\setminus\left\{ v\right\} $
we have $\sigma\left(v,u\right)\setminus\left\{ v\right\} \nsim\partial C,$
or
\item $\exists w=w\left(v\right)\in\tilde{\gamma}_{v,y}\setminus\left\{ v\right\} $
such that $\forall u\in\tilde{\gamma}_{v,w}\setminus\left\{ v,w\right\} $
we have $\sigma\left(v,u\right)\setminus\left\{ v\right\} \nsim\partial C,$
but $\sigma\left(v,w\right)\setminus\left\{ v\right\} \sim\partial C.$
\end{itemize}
We perform the following procedure. We start at $x.$ If the first
case above holds for $v=x,$ then we replace $\tilde{\gamma}$ by
$\sigma\left(x,y\right)$ and finish the procedure. In the second
case we replace $\tilde{\gamma}_{x,w\left(x\right)}$ by $\sigma\left(x,w\left(x\right)\right),$
and repeat the procedure for $\tilde{\gamma}_{w\left(x\right),y}$
starting from $w\left(x\right).$ At each step of the procedure, we
move at least one vertex further on $\tilde{\gamma,}$ hence the procedure
terminates in at most $\left|\tilde{\gamma}\right|$ steps. Let $\gamma$
denote the path we get at the end. At each step of the procedure,
we make modifications such that the new path is in $C$ and its length
is the same as the old path's. Hence $\gamma\subset C$ and $\left|\gamma\right|=\left|\tilde{\gamma}\right|.$ 

We finish Step 1 by with the following consequences of the construction
above: $\gamma$ resembles a shortest curve in $\mathbb{R}^{2}:$
It is a sequence of triangular grid approximations of line segments
in $\mathbb{R}^{2}.$ Moreover, we have the following claim.
\begin{claim}
\label{claim: turns of gamma}As we walk along $\gamma,$ we turn
to the left (right) at $v\in\gamma$ if it has a neighbour in $\partial C$
on the left (right) of $\gamma.$ That is, if $u,v,w\in\gamma$ with
$u\prec v\prec w$ and $\sigma\left(u,v\right),\sigma\left(v,w\right)\subset\gamma,$
with $\sigma\left(u,v\right)\cup\sigma\left(v,w\right)\neq\sigma\left(u,v\right),$
then $v\sim\partial C\cap T\left(u,v,w\right),$ where $T\left(u,v,w\right)$
denotes the triangle spanned by the vertices $u,v,w.$ 
\end{claim}
\medskip

\textbf{Step 2. }\emph{We introduce some notation and assign labels
to some of the vertices of $\gamma.$}

Let 
\[
ST:=\left\{ v=\left(v_{1},v_{2}\right)\in V\,\left|\, x_{1}<v_{1}<y_{1}\right.\right\} .
\]
By possible shortening $\gamma$ and redefining $x$ and $y,$ we
can assume that $\gamma\subset cl\left(ST\right),$ $\gamma\cap\partial ST=\left\{ x,y\right\} $
and $d\left(x,y\right)=\diam\left(C\right).$ 

We set $\alpha:=\left\lfloor a/6\right\rfloor -2>0,$ and define 
\begin{align*}
ST^{i}: & =\left\{ v=\left(v_{1},v_{2}\right)\in V\,\left|\, x_{1}+b+i\alpha<v_{1}<y_{1}-b-i\alpha\right.\right\} 
\end{align*}
for $i\in\left\{ 1,2\right\} .$ Let $x^{i}$ ($y^{i}$) denote the
last (first) vertex of $\gamma$ which is in the half plane $\left\{ v=\left(v_{1},v_{2}\right)\in V\,\left|\, v_{1}\leq x_{1}+b+i\alpha\right.\right\}$
($\left\{ v\,\left|\, v_{1}\geq y_{1}-b-i\alpha\right.\right\} $).
Let $\gamma^{i}=\gamma_{x^{i},y^{i}}.$ Note that $ST^{1}\supset ST^{2}$
and $\gamma^{2}$ is a subpath of $\gamma^{1}.$ 

Let $i\in\left\{ 1,2\right\} .$ Since $\gamma^{i}$ is a shortest
path, it is non self-touching. This combined with $\gamma^{i}\cap\partial ST^{i}=\left\{ x^{i},y^{i}\right\} $
we get that $\gamma^{i},$ cuts $cl\left(ST^{i}\right)$ into two
connected components. Let $ST^{i}{}_{L}$ ($ST^{i}{}_{R}$) denote
connected component $cl\left(ST^{i}\right)\setminus\gamma^{i}$ which
is on the left (right) had side of $\gamma^{i}$ as we walk along
it.

For $v\in\gamma^{2},$ we put a label $l\left(v\right)\in\left\{ L,R,N,G\right\} $
as follows. We denote the set of vertices with label $X\in\left\{ L,R,N,G\right\} $
by $\gamma_{X}^{2}.$ First we define the labels $R$ and $L:$ For
$v\in\gamma^{2},$ we set $l\left(v\right)=L$ ($l\left(v\right)=R$)
if $ST^{1}{}_{L}\cap B\left(v;\alpha\right)\cap\partial C\neq\emptyset$
($ST^{1}{}_{R}\cap B\left(v;\alpha\right)\cap\partial C\neq\emptyset$).
To show that the labels $L,R$ are well-defined, we have to check
that for $v\in\gamma^{2}$ at most one of the sets $ST^{1}{}_{L}\cap B\left(v;\alpha\right)\cap\partial C$
and $ST^{1}{}_{R}\cap B\left(v;\alpha\right)\cap\partial C$ is non-empty.
Since $2\alpha<a,$ this follows from Condition \ref{cond nice: 6 arms, then dead end}
of Definition \ref{def: (a,b)-nice}. Let $\beta:=\left\lfloor \alpha/3\right\rfloor .$
For $v\in\gamma^{2}\setminus\left(\gamma_{L}^{2}\cup\gamma_{R}^{2}\right)$
we set $l\left(v\right)=G$ if $B\left(v;\beta\right)\cap\left(\gamma_{L}^{2}\cup_{R}^{2}\right)=\emptyset,$
and $l\left(v\right)=N$ otherwise.

Since $4\alpha+2\beta<a,$ it is a simple exercise to prove the following
claim using Condition \ref{cond nice: 6 arms, then dead end} of Definition
\ref{def: (a,b)-nice}, which finishes Step 2.
\begin{claim}
\label{claim: there is good between left and right}Let $u\in\gamma_{L}^{1}$
and $v\in\gamma_{R}^{1}.$ Then there is $w\in\gamma_{G}^{1}$ which
is in between $u$ and $v.$
\end{claim}
\medskip

\textbf{Step 3. }\emph{We define the neighbourhoods $\mathcal{F}_{v}$
and $\mathcal{G}_{v}$ for $v\in\gamma^{2}.$ }

If $l\left(v\right)\in\left\{ G,N\right\} $ then we set $\mathcal{F}_{v}:=B$$\left(v;\alpha\right)$
and $\mathcal{G}_{v}:=B\left(v;\beta\right).$ 

If $l\left(v\right)\in\left\{ L,R\right\} ,$ let $f^{1}$ ($f^{2}$)
as the last vertex when we go backwards (forward) from $v$ along
$\gamma$ which is in $B\left(v;\alpha\right).$ If it has label $L$
($R$) then we define $\mathcal{F}_{v}$ as the connected component
of $B\left(v;\alpha\right)\setminus\gamma_{f^{1},f^{2}}$ on the right
(left) hand side of $\gamma_{f^{1},f^{2}}.$ Similarly we define $g^{1}$
and $g^{2}$ in the box $B\left(v;\beta\right),$ and $\mathcal{G}_{v}.$ 

The combination of $4\alpha<a,$ Claim \ref{claim: turns of gamma}
and Condition \ref{cond nice: 6 arms, then dead end} of Definition
\ref{def: (a,b)-nice} gives that 
\[
\left(\gamma_{f^{1},g^{1}}\cup\gamma_{g^{2},f^{2}}\right)\cap B\left(v;\beta-1\right)=\emptyset.
\]
Hence we get
\begin{claim}
$\mathcal{F}_{v}\cap B\left(v;\beta\right)=\mathcal{G}_{v}$ for $v\in\gamma^{2}.$
\end{claim}
\medskip

\textbf{Step 4.} \emph{We investigate the neighbourhood $\mathcal{G}_{v}.$}
\begin{claim}
\label{claim: G_v is good}$\mathcal{G}_{v}\cap\partial C=\emptyset$
for $v\in\gamma^{2},$ and $\mathcal{G}_{v}\cap\gamma^{1}=\emptyset$
for $v\in\gamma_{L}^{2}\cup\gamma_{R}^{2}.$\end{claim}
\begin{proof}
[Proof of Claim \ref{claim: G_v is good}] First we show that $\mathcal{G}_{v}\cap\partial C=\emptyset$
with a proof by contradiction. Suppose that $\mathcal{G}_{v}\cap\partial C\neq\emptyset.$
The definition of labels give that if $\mathcal{G}_{v}\cap\partial C\neq\emptyset,$
then $l\left(v\right)=L$ or $R.$ We further suppose that $l\left(v\right)=L.$
The case where $l\left(v\right)=R$ can be treated similarly. We choose
$w$ so that it is one of the closest vertices to $v$ among the vertices
of $\mathcal{G}_{v}\cap\partial C.$ See Figure \ref{fig:shorten}. 

\begin{figure}
\centering{}\includegraphics{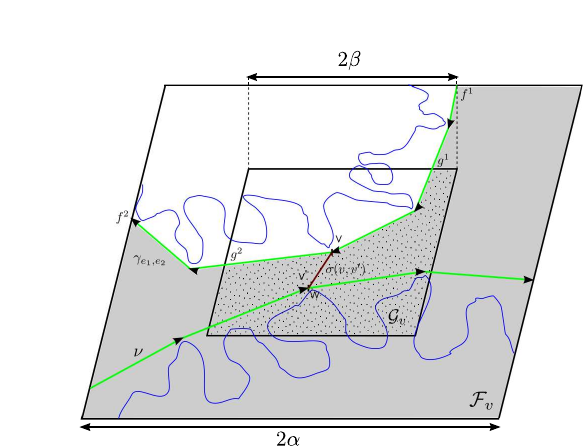}\caption{\label{fig:shorten}The path $\gamma_{x,v}\vee\sigma_{v,w}\vee\gamma_{w,y},$
is shorter than $\gamma$ by at least $\frac{2}{3}\alpha$ vertices.}
\end{figure}
By the definition of the label $L,$ we have that $w\in ST{}_{L}^{2}\cap B\left(v;\beta\right).$
Since $w\in\mathcal{G}_{v},$ i.e. $w$ is on the right hand side
of $\gamma_{f^{1},f^{2}}$ in $B\left(v;\alpha\right).$ Hence some
subpath of $\gamma^{1}\setminus\gamma_{f^{1},f^{2}},$ denoted by
$\nu,$ has to separate $w$ from $v$ in $\mathcal{F}_{v}.$ Let
us walk from $v$ to $w$ on $\sigma\left(v,w\right),$ till we hit
$\nu.$ Let us denote the explored path by $\sigma\left(v,v'\right),$
where $v'$ is the last point of the exploration. Let $\gamma'$ be
the path we get when we replace the part of $\gamma$ between $v$
and $v'$ by $\sigma\left(v,v'\right).$ Consider the case $v'\prec_{\gamma}v.$
The other case where $v'\succ_{\gamma}v$ can be treated similarly.
The number of vertices of $\sigma\left(v,v'\right)$ is at most $2\beta.$
However, the number of vertices in $\nu$ before $v'$ is at least
$\alpha-\beta.$ Moreover, $\left|\gamma_{f^{1},v}\right|\geq\alpha-\beta.$
Hence 
\begin{align}
\left|\gamma\right|-\left|\gamma'\right| & \geq2\left(\alpha-\beta\right)-2\beta\nonumber \\
 & \geq\frac{2}{3}\alpha>0.\label{eq: pf claim G_v is good}
\end{align}
The definition of $w$ gives that $\sigma\left(v,v'\right)\subset C,$
thus $\gamma'\subset C.$ Hence $\gamma'$ connects $x$ and $y$
in $C$ and by \ref{eq: pf claim G_v is good}, it is shorter than
$\gamma.$ This contradicts the definition of $\gamma,$ hence $\mathcal{G}_{v}\cap\partial C=\emptyset$
for $v\in\gamma^{2}.$

The proof of $\mathcal{G}_{v}\cap\gamma^{1}=\emptyset$ for $v\in\gamma^{2}$
is quite similar to the one above, hence we omit it, and finish the
proof of Claim \ref{claim: G_v is good} and conclude Step 4.
\end{proof}
\medskip

\textbf{Step 5. }\emph{We define the path $\zeta.$ }

We set $\varepsilon=\left\lfloor \beta/4\right\rfloor -2.$ For $j\in\left\{ L,R\right\} ,$
let 
\begin{equation}
U_{j}:=\bigcup_{v\in\gamma_{j}^{2}}B\left(v;\varepsilon\right).\label{eq: pf det gridpath - 1}
\end{equation}
$ST{}_{R}^{2}\setminus U_{L}\setminus\gamma^{2}$ ($ST{}_{L}^{2}\setminus U_{R}\setminus\gamma^{2}$
) has one infinite connected component which we denote by $Z_{R}$
($Z_{L}$). Let $\zeta_{j}$ denote the shortest path in $\partial Z_{j}\cap ST^{2}$
which connects the left and the right side of $ST_{2}.$ We orient
$\zeta_{L}$ ($\zeta_{R}$) so that $Z_{L}$ ($Z_{R}$) is on the
left (right) hand side. Note that $\zeta_{L},\zeta_{R}$ are left-right
crossings of $ST^{2}.$ 

Note that $\zeta_{L},\zeta_{R}$ and $\gamma^{2}$ are non self-touching
paths. Since $Z_{R},Z_{L}$ and $\gamma^{2}$ are disjoint, $\gamma^{2}$
is sandwiched between $\zeta_{L}$ and $\zeta_{R}.$ Hence $\zeta_{L},\zeta_{R},\gamma^{2}$
can have common vertices, but they cannot cross each other. Thus we
get the following claim.
\begin{claim}
\label{claim: intersection of zetas}Let $v\in\zeta_{L}\cap\zeta_{R}.$
Then $v\in\gamma^{2}.$
\end{claim}
Condition \ref{cond nice: 6 arms, then dead end} of Definition \ref{def: (a,b)-nice}
implies the following claim.
\begin{claim}
\label{claim: vertex after zeta cap zeta on gamma}Let $v\in\zeta_{L}\cap\zeta_{R}.$
If $w,$ the next vertex after $v$ on $\gamma^{2}$ exists, then
$w\in\zeta_{L}\cup\zeta_{R}.$
\end{claim}
Let $\overset{\rightarrow}{G}=\left(\overset{\rightarrow}{V},\overset{\rightarrow}{E}\right)$
be the directed graph induced by the directed paths $\zeta_{L},\zeta_{R}$
and $\gamma^{2}.$ That is $\overset{\rightarrow}{G}=\left(\overset{\rightarrow}{V},\overset{\rightarrow}{E}\right)$
where $\overset{\rightarrow}{V}=\zeta_{L}\cup\zeta_{R}\cup\gamma^{2},$
and $\left(u,v\right)\in\overset{\rightarrow}{E}$ if and only if
$u,v\in\nu,$ $u\sim v$ and $u\prec_{\nu}v$ for some $\nu\in\left\{ \zeta_{L},\zeta_{R},\gamma^{2}\right\} .$
Using the definition of $\zeta_{L}$ and $\zeta_{R}$ it is a simple
exercise to show the following claim. 
\begin{claim}
\label{claim: G has no directed loops}$\overset{\rightarrow}{G}$
has no directed loops. 
\end{claim}
For $j\in\left\{ L,R\right\} $ and $z\in\zeta_{j}$ let $n_{j}\left(z\right)$
be the first vertex of $\zeta_{j}\cap\gamma^{2}$ after $z$ on $\zeta_{j}.$
That is, $n_{j}\left(z\right)\in\zeta_{j}\cap\gamma^{2}$ with $n_{j}\left(z\right)\succeq_{\zeta_{j}}z$
and if $z'\in\zeta_{j}\cap\gamma^{2}$ with $z'\succ_{\zeta_{j}}z$
then $z'\succeq_{\zeta_{j}}n_{j}\left(z\right).$ If there is no such
vertex, then we set $n_{j}\left(z\right)=\emptyset.$

We define a directed path $\zeta$ by the following procedure. Let
$z_{j}$ denote the starting point of $\zeta_{j}$ for $j\in\left\{ L,R\right\} .$
$\zeta$ starts at the vertex $z$ defined as

\[
z:=\begin{cases}
z_{L} & \mbox{when }n_{L}\left(z_{L}\right)=\emptyset,\mbox{ or, when }n_{L}\left(z_{L}\right)\neq\emptyset\neq n_{R}\left(z_{R}\right)\mbox{ and }n_{L}\left(z_{L}\right)\succeq_{\gamma^{2}}n_{R}\left(z_{R}\right)\\
z_{R} & \mbox{otherwise.}
\end{cases}
\]
Suppose that we are at vertex $v$ in $\zeta.$ If $v$ is the endpoint
of $\zeta_{L}$ or $\zeta_{R},$ we terminate the procedure. Otherwise,
we define the next vertex of $\zeta,$ denoted by $w,$ as follows.
For $j\in\left\{ L,R\right\} ,$ if $v\in\zeta_{j},$ then $v_{j}$
denotes the next vertex after $v$ in $\zeta_{j}.$ 
\begin{itemize}
\item If $v\in\zeta_{L}\setminus\zeta_{R},$ then $w=v_{L}$
\item if $v\in\zeta_{R}\setminus\zeta_{L},$ then $w=v_{R}$
\item if $v\in\zeta_{L}\cap\zeta_{R},$ and if

\begin{itemize}
\item $v_{L},v_{R}\in\gamma^{2},$ then the definition of $\zeta_{L}$ and
$\zeta_{R}$ gives that $v_{L}=v_{R}$ and we take $w=v_{L}=v_{R}$
\item $v_{L}\in\gamma^{2},$ $v_{R}\notin\gamma^{2},$ then $w=v_{R}$ 
\item $v_{R}\in\gamma^{2},$ $v_{L}\notin\gamma^{2},$ then $w=v_{L}$ 
\item the case $v_{L},v_{R}\notin\gamma^{2}$ is impossible by Claim \ref{claim: vertex after zeta cap zeta on gamma}. 
\end{itemize}
\end{itemize}
We finish Step 5 by showing that $\zeta$ is well-defined. The definition
of $\zeta$ shows that if we view $\zeta$ as a directed graph, it
is a subgraph of $\overset{\rightarrow}{G}.$ Hence by Claim \ref{claim: G has no directed loops}
$\zeta$ has no directed loops. Thus $\zeta$ is self avoiding, and
the procedure above terminates after finitely many steps, when $\zeta$
reaches the endpoint of $\zeta_{L}$ or $\zeta_{R}.$

\medskip

\textbf{Step 6. }\emph{We prove the following claim and finish the
proof of Lemma}\textbf{\emph{ }}\emph{\ref{lem: det gridpath}.}
\begin{claim}
\label{claim: almost done} $\zeta+B\left(\varepsilon\right)\subset C$
and $\diam\left(\zeta\right)\geq d\left(x,y\right)-2b-4\alpha.$\end{claim}
\begin{proof}
[Proof of Claim \ref{claim: almost done}] The definition of $\zeta$
shows that $\zeta$ is a horizontal crossing of $ST^{2}.$ Hence $\diam\left(\zeta\right)\geq d\left(x,y\right)-2b-4a.$
We show that for all $v\in\zeta$ we have $v+B\left(\varepsilon\right)\subset C.$
There are two cases depending on whether $v$ is contained in $\gamma^{2}.$

\medskip

\emph{Case 1:} $v\in\zeta\setminus\gamma^{2}.$ Then $v\in\zeta_{L}\setminus\gamma^{2}$
or $v\in\zeta_{R}\setminus\gamma^{2}.$ We assume that $v\in\zeta_{L}\setminus\gamma^{2}.$
The case where $v\in\zeta_{R}\setminus\gamma^{2}$ can be treated
similarly. The definition of $\zeta_{L}$ gives that there is $w\in\gamma_{R}$
such that $v\in\left(B\left(w;\varepsilon+1\right)\setminus B\left(w;\varepsilon\right)\right)$
and $B\left(v;\varepsilon\right)\cap\gamma_{R}=\emptyset.$ This combined
with $4\alpha+4\varepsilon+2<a$ and Condition \ref{cond nice: 6 arms, then dead end}
of Definition \ref{def: (a,b)-nice} gives that $B\left(v;\varepsilon\right)\cap\left(\gamma_{L}^{2}\cup\gamma_{R}^{2}\right)=\emptyset.$

If $\gamma^{2}\cap B\left(v;\varepsilon\right)\neq\emptyset,$ then
$\exists u\in\left(\gamma_{G}^{2}\cup\gamma_{N}^{2}\right)\cap B\left(v;\varepsilon\right).$
Claim \ref{claim: G_v is good} implies that $C\supset\mathcal{G}_{u}=B\left(u;\beta\right)\supset B\left(v;\varepsilon\right)$
since $4\varepsilon<\beta.$

If $\gamma^{2}\cap B\left(v;\varepsilon\right)=\emptyset,$ then the
definition of $w$ and Claim \ref{claim: G_v is good} shows that
$C\supset\mathcal{G}_{w}\supset B\left(v;\varepsilon\right)$ since
$2\beta+2\varepsilon<\alpha.$ 

Hence $B\left(v;\varepsilon\right)\subset C$ in Case 1.

\emph{Case 2:} $v\in\zeta\cap\gamma^{2}.$ Since $\zeta\subset\zeta_{L}\cup\zeta_{R},$
we assume that $v\in\zeta_{L}.$ The case where $v\in\zeta_{R}$ can
be treated similarly. First we show that $v\notin\gamma_{L}^{2}\cap\zeta_{L}.$ 

Suppose the contrary, that is $v\in\gamma_{L}^{2}\cap\zeta_{L}.$
Let $w$ be the starting point of the connected component of $v$
in $\gamma^{2}\cap\zeta.$ By the definition of $\zeta,$ $w\in\zeta_{L}.$
Moreover, for $w'$ the vertex right before $w$ on $\zeta_{L},$
we have $w'\in\zeta_{L}\setminus\gamma_{2}.$ Hence there is $u'\in\gamma_{R}^{2}$
such that $w'\in B\left(u';\varepsilon+1\right).$ Since $v\in\gamma_{L}$
and $u'\in\gamma_{R}^{2},$ by Claim \ref{claim: there is good between left and right}
$\exists u\in\gamma_{G}^{2}$ which is between $u'$ and $v$ on $\gamma^{2}.$
Note that $w'\in\mathcal{G}_{u'}$. By Claim \ref{claim: G_v is good}
we have that $\gamma_{u',w}^{2}\subset\gamma^{2}\setminus\gamma_{G}^{2}.$
Hence $u$ is between $w$ and $v$ on $\gamma^{2}.$ From the definition
of $w,$ we get that $u\in\zeta\cap\zeta_{L}.$ 

Note that if we show that $u\in\zeta_{R},$ then we get a contradiction
by the definition of $\zeta.$ Hence in order to rule out the case
$v\in\gamma_{L}^{2}\cap\zeta_{L}$ it is enough to show that $u\in\zeta_{R}.$

Suppose the contrary, that is $u\notin\zeta_{R}.$ Recall the definition
of $U_{L}$ from \ref{eq: pf det gridpath - 1}. We introduce a new
set of labels on the vertices of $U_{L}$ as follows. For $q\in U_{L}$
there is a vertex $r\in\gamma_{L}$ such that $q\in B\left(r;\varepsilon\right).$
We define 
\[
l'\left(q\right):=\begin{cases}
B & \mbox{if }r\prec_{\gamma^{2}}u\\
A & \mbox{otherwise.}
\end{cases}
\]
Since the choice of $r$ above is not necessarily unique, we have
to show that $l'\left(q\right)$ is well-defined. It can be easily
checked by combining Claim \ref{claim: G_v is good}, $4\varepsilon+4<\beta$
and $u\in\gamma_{G}^{2}.$ Moreover a similar argument shows that
if $q,q'\in U_{L}$ with $q\sim q',$ then $l'\left(q\right)=l'\left(q'\right).$ 

Since $\gamma^{2}$ is non self-touching, $u\in\gamma^{2}$ is connected
to $\infty$ in $ST_{R}.$ Since $u\notin\zeta_{R}$ it is not connected
to $\infty$ in $Z_{R},$ there is a path $\nu\subset U_{L}$ which
separates $u$ from $\infty$ in $ST_{R}.$ We can choose $\nu$ such
that it starts and ends at a vertex neighbouring $\gamma^{2}.$ By
a possible shortening of $\nu,$ we can assume that if $u'\in\nu$
with $u'\sim\gamma^{2},$ than $u'$ is either the starting or the
endpoint of $\nu.$ Let $u_{1},u_{2}$ be neighbours of the starting
point and the endpoint of $\nu$ which are in $\gamma^{2}.$ The definition
of $\nu$ gives that $u$ is in between $u_{1}$ and $u_{2}$ on $\gamma^{2}.$
Using Condition \ref{cond nice: 6 arms, then dead end} of Definition
\ref{def: (a,b)-nice} and that $u\in\gamma_{G}^{2}$ it is easy to
check that $l'\left(u_{1}\right)\neq l'\left(u_{2}\right).$

On the other hand, $\nu$ is a connected subset of $U_{L},$ hence
$l'$ is constant on $\nu.$ This is a contradiction, thus $u\in\zeta_{R},$
which in turn shows that $v\in\gamma_{L}^{2}\cap\zeta_{L}.$

Hence $v\notin\gamma_{L}^{2}\cap\zeta_{L}$ but $v\in\zeta\cap\gamma^{2}\cap\zeta_{L}.$
The definition of $\zeta_{L}$ gives that $v\notin\gamma_{R}.$ Hence
$v\in\gamma_{N}^{2}\cup\gamma_{G}^{2}.$ By Claim \ref{claim: G_v is good}
we get $C\supset\mathcal{G}_{v}=B\left(v;\beta\right)\supset B\left(v;\varepsilon\right),$
and we are done in Case 2. Since there are no other cases left, the
proof of Claim \ref{claim: almost done} is finished.
\end{proof}
Since $\zeta+B\left(\varepsilon\right)\subset C$ and $\diam\left(\zeta\right)\geq d\left(x,y\right)-2b-4\alpha$
hence the $\left\lfloor \varepsilon/2\right\rfloor $-gridpath approximation
of $\zeta$ is contained in $C.$ It has diameter at least $d\left(x,y\right)-2b-4\alpha-\varepsilon\geq d\left(x,y\right)-2b-2a-12.$
Since $\varepsilon=\left\lfloor \beta/4\right\rfloor -2\geq a/100-5$
this concludes the proof of the Lemma \ref{lem: det gridpath}.
\end{proof}
We finish the appendix by proving Lemma \ref{lem: local det gridpath}.
\begin{namedthm}
[Lemma \ref{lem: local det gridpath}] Let $a,b,c\in\mathbb{N}$
with $a\geq2000.$ Let $C$ be subgraph of $\mathbb{T}$ which is
$\left(a,b\right)$-nice in $B\left(c\right).$ Let $C'$ be a connected
component of $C\cap B\left(c\right).$ Then there is a $\left\lfloor a/200-10\right\rfloor $-gridpath
contained in $C'$ with diameter at least $\diam\left(C'\right)-2b-2a-12.$\end{namedthm}
\begin{proof}
[Proof of Lemma \ref{lem: local det gridpath}] Let $x,y\in C'$
with $d\left(x,y\right)=\diam\left(C'\right).$ We choose $\tilde{\gamma}$
as one of the shortest paths connecting $x,y$ in $C'.$ From this
point on, we can follow the proof of Lemma \ref{lem: det gridpath}
since we will use Condition \ref{cond nice: 6 arms, then dead end}
of Definition \ref{def: (a,b)-nice} for pairs of vertices $u,v\in\partial C$
which are contained in $B\left(c\right).$ 
\end{proof}
\bibliographystyle{plain}
\bibliography{/ufs/demeter/LyX/References/myreflist}

Demeter Kiss

Centrum Wiskunde \& Informatica (CWI)

123 Science Park

1098 XG Amsterdam

The Netherlands

e-mail: D.Kiss@cwi.nl
\end{document}